\documentclass[11pt, oneside, reprint]{amsart}
\usepackage[margin=1in]{geometry}
\usepackage[parfill]{parskip}
\usepackage{graphicx}
\usepackage{amssymb}
\usepackage{amsmath}
\usepackage{autobreak}
\usepackage{palatino}
\usepackage{float}
\usepackage[dvipsnames]{xcolor}
\usepackage[hidelinks]{hyperref}
\usepackage{cleveref}
\usepackage{nameref}
\usepackage[foot]{amsaddr}
\usepackage{tikz}
\usetikzlibrary{decorations.pathreplacing}
\usepackage{dirtytalk}
\usepackage{manfnt}
\usepackage{wasysym}
\usepackage{breakcites}
\usepackage{xcolor}
\usepackage{amsthm}
\usepackage{subfiles}
\usepackage{physics}
\newtheorem{lemma}{Lemma}[section]
\newtheorem{corollary}{Corollary}[section]

\newtheorem{theorem}{Theorem}[section]
\newtheorem{definition}{Definition}[section]
\usepackage[normalem]{ulem}
\usepackage{xparse}
\usepackage{algorithm}
\usepackage{pifont}
\usepackage{subcaption}

\newcommand\bm[1]{\boldsymbol{#1}}

\Crefname{figure}{Figure}{Figures}
\crefname{figure}{Fig.}{Figs.}
\Crefname{equation}{Equation}{Equations}
\crefname{equation}{Eq.}{Eqs.}
\Crefname{section}{Section}{Sections}
\crefname{section}{Sec.}{Secs.}
\crefname{appendix}{Appendix}{Appendices}
\Crefname{appendix}{Appendix}{Appendices}

\newcommand\Nref[1]{\Cref{#1} (\nameref{#1})}

\NewDocumentCommand\sw{ s m }{
   \IfBooleanTF{#1}{%
      \textcolor{orange}{#2}%
   }{%
      \textcolor{orange}{[SW: #2]}%
   }%
}

\NewDocumentCommand\gb{ s m }{
   \IfBooleanTF{#1}{%
      \textcolor{blue}{#2}%
   }{%
      \textcolor{blue}{[Gb: #2]}%
   }%
}

\newcommand*\cube{{\mbox{\mancube}}}
\renewcommand\rm[1]{{\mathrm{#1}}}

\newcommand\tesseract{%
\begin{tikzpicture}[scale=0.1, baseline={([yshift=-.5ex] current bounding box.center)}, line width=0.6]
\node at (0.25, 0.25) {$\,$};
\node at (0.25-0.5, 0.25-0.5) {$\,$};
\node at (0.25+0.5, 0.25+0.5) {$\,$};
\begin{pgfinterruptboundingbox}
		\node [] (0) at (-1, 1) {};
		\node [] (1) at (2, 2) {};
		\node [] (2) at (1, -1) {};
		\node [] (3) at (-1, -1) {};
		\node [] (4) at (0, 2) {};
		\node [] (5) at (1, 1) {};
		\node [] (6) at (2, 0) {};
		\node [] (7) at (0, 0) {};
		\node [] (13) at (-0.25, -0.25) {};
		\node [] (15) at (0.75, -0.25) {};
		\node [] (16) at (0.75, 0.75) {};
		\node [] (17) at (-0.25, 0.75) {};
		\node [] (18) at (0.25, 1.25) {};
		\node [] (19) at (1.25, 1.25) {};
		\node [] (20) at (1.25, 0.25) {};
		\node [] (21) at (0.25, 0.25) {};
		\draw [gray] (18.center) to (17.center);
		\draw [gray] (15.center) to (20.center);
		\draw [gray] (16.center) to (5.center);
		\draw [gray] (18.center) to (21.center);
		\draw [gray] (21.center) to (20.center);
		\draw (2.center) to (6.center);
		\draw (6.center) to (1.center);
		\draw (1.center) to (4.center);
		\draw (4.center) to (0.center);
		\draw (0.center) to (3.center);
		\draw (3.center) to (2.center);
		\draw (0.center) to (5.center);
		\draw (5.center) to (2.center);
		\draw (5.center) to (1.center);
		\draw (17.center) to (13.center);
		\draw (13.center) to (15.center);
		\draw (15.center) to (16.center);
		\draw (16.center) to (17.center);
		\draw (18.center) to (19.center);
		\draw (18.center) to (4.center);
		\draw (13.center) to (3.center);
		\draw (15.center) to (2.center);
		\draw (20.center) to (6.center);
		\draw (0.center) to (17.center);
		\draw (13.center) to (21.center);
		\draw (19.center) to (20.center);
\end{pgfinterruptboundingbox}
\end{tikzpicture}}

\makeatletter 
\def\@tocline#1#2#3#4#5#6#7{\relax
  \ifnum #1>\c@tocdepth 
  \else
    \par \addpenalty\@secpenalty\addvspace{#2}%
    \begingroup \hyphenpenalty\@M
    \@ifempty{#4}{%
      \@tempdima\csname r@tocindent\number#1\endcsname\relax
    }{%
      \@tempdima#4\relax
    }%
    \parindent\z@ \leftskip#3\relax \advance\leftskip\@tempdima\relax
    \rightskip\@pnumwidth plus4em \parfillskip-\@pnumwidth
    #5\leavevmode\hskip-\@tempdima
      \ifcase #1
       \or\or \hskip 1em \or \hskip 2em \else \hskip 3em \fi%
      #6\nobreak\relax
    \hfill\hbox to\@pnumwidth{\@tocpagenum{#7}}\par
    \nobreak
    \endgroup
  \fi}
\makeatother

\makeatletter
\def\@cite#1#2{({#1\if@tempswa , #2\fi})}
\let\@internalcite\cite
\def\cite{\def\citeauthoryear##1##2{##1, ##2}\@internalcite}
\def\shortcite{\def\citeauthoryear##1##2{##2}\@internalcite}
\def\@biblabel#1{\def\citeauthoryear##1##2{##1, ##2}[#1]\hfill}
\makeatother
 
\begin{document}


\title[Existence and bounds of growth constants for restricted manifolds]{Existence and bounds of growth constants \\ for restricted walks, surfaces, and generalisations}
\author{Sun Woo P. Kim$^*$$^1$}
\author{Gabriele Pinna$^1$}
\thanks{$^*$ swk34@cantab.ac.uk}
\thanks{$^1$ Department of Physics, King's College London, Strand, London WC2R 2LS, United Kingdom}

\begin{abstract}
    We introduce classes of restricted walks, surfaces and their generalisations. For example, self-osculating walks (SOWs) are supersets of self-avoiding walks (SAWs) where edges are still not allowed to cross but may `kiss' at a vertex. They are analogous to osculating polygons introduced in \cite{jensen1998self} except that they are not required to be closed. The `automata' method of \cite{ponitz2000improved} can be adapted to such restricted walks. For example, we prove upper bounds for the connective constant for SOWs on the square and triangular lattices to be $\mu^\rm{SOW}_\square \leq 2.73911$ and $\mu^\rm{SOW}_\triangle \leq 4.44931$, respectively. In analogy, we also introduce self-osculating surfaces (SOSs), a superset of self-avoiding surfaces (SASs) which can be generated from fixed polyominoids (XDs). We further generalise and define self-avoiding $k$-manifolds (SAMs) and its supersets, self-osculating $k$-manifolds (SOMs) in the $d$-dim hypercubic lattice and $(d, k)$-XDs. By adapting the concatenation procedure \cite{van1989self}, we prove that their growth constants exist,
    and prove an explicit form for their upper and lower bounds.
    The upper bounds can be improved by adapting the `twig' method, originally developed for polyominoes \cite{eden1961two,klarner1973procedure}. For the cubic lattice, we find improved upper bounds for the growth constant of SASs as $\mu^\rm{SAS}_{\cube} \leq 17.11728$.
    
    \bigskip
    
\end{abstract}

\maketitle

{
\begin{align*}
    \begin{tikzpicture}[baseline={([yshift=-.5ex] current bounding box.center)}, scale=0.65]
    \foreach \x in {0,...,10}{
        \draw [gray!50, ultra thick] (\x, 0) to (\x, 10);
    }
    \foreach \y in {0,...,10}{
        \draw [gray!50, ultra thick] (0, \y) to (10, \y);
    } 
    \draw [rounded corners, ultra thick, blue!75] (5.2, 5) -- (6,5) -- (6,6) -- (7,6) -- (7,5) -- (6,5) -- (6,4) -- (7,4) -- (7,3) -- (6,3) -- (5,3) -- (5,4) -- (5,5) -- (4,5) -- (3,5) -- (3,6) -- (4,6) -- (4,7) -- (3,7) -- (3,6) -- (2,6) -- (2,7) -- (3,7) -- (3,8) -- (4,8) -- (4,7) -- (5,7) -- (6,7) -- (7,7) -- (7,8) -- (8,8) -- (8,9) -- (9,9) -- (9, 8) -- (8,8) -- (8,7) -- (8.8,7);
    \draw[ultra thick, blue!75] (5.2, 5.1) -- (5.2, 4.9);
    
    \draw[ultra thick, blue!75] (8.8, 7.1) -- (8.8, 6.9);
    \end{tikzpicture}
\end{align*}
}

\newpage

\tableofcontents

\newpage

\section{Introduction}
\subsection{Self-osculating walks, osculating domain walls, and other restricted walks} \label{sec:restricted-walks-intro}

A self-\newline avoiding walk (SAW) on a graph starting from a vertex $v_0$ is a sequence of vertices that are successively connected by edges, and no vertex is visited twice. The length of the walk $n$ is then given by the number of vertices, not counting $v_0$. There are many ways one could modify or generalise SAWs. For example, a neighbour-avoiding walk (NAW) is like the SAW except that any vertex in the sequence cannot neighbour any other vertex in the sequence except the one immediately before or after. An edge-avoiding walk (EAW) cannot pass through an edge more than once. Self-avoiding `polygons' form a subset of SAWs, composed of closed SAWs \cite{jensen2001osculating}.

In this work, we study a particular modification that we call a self-osculating walk (SOW). While the above examples can be defined on any graph, SOWs require a notion of geometry to define, and we will only consider the triangular, hexagonal (honeycomb), and $d$-dim hypercubic lattices (i.e. Euclidean tilings by regular polytopes)~\footnote{Without a loss of generality, we will set $v_0$ to be the origin.}. 

To motivate SOWs, we will first note that SAWs can be thought of as a kind of vertex model. For example, on the square lattice, it is given by the following vertex configurations up to rotations and reflections,
\begin{align}
    \begin{tikzpicture}[baseline={([yshift=-.5ex] current bounding box.center)}, scale=0.65]
    \node [] (0) at (0, 1) {};
    \node [] (1) at (-1, 0) {};
    \node [] (2) at (0, -1) {};
    \node [] (3) at (1, 0) {};
    \draw [gray!50, ultra thick] (0, -1) to (0, 1);
    \draw [gray!50, ultra thick] (-1, 0) to (1, 0);
    \draw[dashed] (1,1) -- (-1,1) -- (-1,-1) -- (1,-1) -- (1, 1);
    \end{tikzpicture}
    \quad
    \begin{tikzpicture}[baseline={([yshift=-.5ex] current bounding box.center)}, scale=0.65]
    \node [] (0) at (0, 1) {};
    \node [] (1) at (-1, 0) {};
    \node [] (2) at (0, -1) {};
    \node [] (3) at (1, 0) {};
    \draw [gray!50, ultra thick] (0, -1) to (0, 1);
    \draw [gray!50, ultra thick] (-1, 0) to (1, 0);
    \draw [ultra thick, blue!75] (0, -1) to (0, 1);
    \draw[dashed] (1,1) -- (-1,1) -- (-1,-1) -- (1,-1) -- (1, 1);
    \end{tikzpicture}
    \quad
    \begin{tikzpicture}[baseline={([yshift=-.5ex] current bounding box.center)}, scale=0.65]
    \node [] (0) at (0, 1) {};
    \node [] (1) at (-1, 0) {};
    \node [] (2) at (0, -1) {};
    \node [] (3) at (1, 0) {};
    \draw [gray!50, ultra thick] (0, -1) to (0, 1);
    \draw [gray!50, ultra thick] (-1, 0) to (1, 0);
    \draw [ultra thick, blue!75] (0, -1) -- (0, 0) -- (1, 0);
    \draw[dashed] (1,1) -- (-1,1) -- (-1,-1) -- (1,-1) -- (1, 1);
    \end{tikzpicture}
    \quad
        \begin{tikzpicture}[baseline={([yshift=-.5ex] current bounding box.center)}, scale=0.65]
    \node [] (0) at (0, 1) {};
    \node [] (1) at (-1, 0) {};
    \node [] (2) at (0, -1) {};
    \node [] (3) at (1, 0) {};
    \draw [gray!50, ultra thick] (0, -1) to (0, 1);
    \draw [gray!50, ultra thick] (-1, 0) to (1, 0);
    \draw[ultra thick, blue!75] (-0.1, 0) -- (0.1, 0);
    \draw [ultra thick, blue!75] (0, -1) -- (0, 0);
    \draw[dashed] (1,1) -- (-1,1) -- (-1,-1) -- (1,-1) -- (1, 1);
    \end{tikzpicture}.
\end{align}
Here, the `stub' on the rightmost vertex configuration denotes an open end, corresponding to the starting or terminating point of the walk. Then, a SAW of length $n$ can be thought of as a tiling of the lattice with such vertex configurations with the restriction that neighbouring vertex configurations must be compatible with each other (i.e. lines connect), the origin is occupied, there is one connected component, and the total length is $n$.

Now consider a `crossing' vertex configuration on the square lattice, which would be disallowed in SAWs. We can modify the edges in two ways such that they only `kiss' at a vertex, as follows:
\begin{equation}
    \begin{tikzpicture}[baseline={([yshift=-.5ex] current bounding box.center)}, scale=0.64]
	\draw [ultra thick, blue!75] (0, -1) to (0, 1);
        \draw [ultra thick, blue!75] (-1, 0) to (1, 0);
        \draw[dashed] (1,1) -- (-1,1) -- (-1,-1) -- (1,-1) -- (1, 1);
    \end{tikzpicture}
    \longrightarrow
    \begin{tikzpicture}[baseline={([yshift=-.5ex] current bounding box.center)}, scale=0.65]
    \draw [gray!50, ultra thick] (0, -1) to (0, 1);
    \draw [gray!50, ultra thick] (-1, 0) to (1, 0);
    \draw [bend left=45, looseness=1.75, ultra thick, blue!75] (0, 1) to (-1, 0);
    \draw [bend left=315, looseness=1.75, ultra thick, blue!75] (1, 0) to (0, -1);
    \draw[dashed] (1,1) -- (-1,1) -- (-1,-1) -- (1,-1) -- (1, 1);
    \end{tikzpicture}
     \quad
    \begin{tikzpicture}[baseline={([yshift=-.5ex] current bounding box.center)}, scale=0.65]
		\node [] (0) at (0, 1) {};
		\node [] (1) at (-1, 0) {};
		\node [] (2) at (0, -1) {};
		\node [] (3) at (1, 0) {};
    \draw [gray!50, ultra thick] (0, -1) to (0, 1);
    \draw [gray!50, ultra thick] (-1, 0) to (1, 0);
		\draw [bend right=45, looseness=1.75, ultra thick, blue!75] (0.center) to (3.center);
		\draw [bend left=45, looseness=1.75, ultra thick, blue!75] (1.center) to (2.center);
        \draw[dashed] (1,1) -- (-1,1) -- (-1,-1) -- (1,-1) -- (1, 1);
    \end{tikzpicture}.
\end{equation}
This could be thought of as replacing each sharp corner with a smooth one. We will define length of paths of such vertex configurations to be the same as the original crossing ones~\footnote{This could be thought of as placing a small circle of radius $\epsilon$ that touches the edges and taking the path along the circle as the corner is approached, then sending $\epsilon \rightarrow 0$.}.

Then, SOWs would be walks which allow for such vertex configurations to occur. A natural question would be which open ends are allowed. We have the following `bulk' vertex configurations (vertex configurations with no open ends) up to rotations and reflections,
\begin{equation} \label{eq:bulk-vertices-square}
\begin{tikzpicture}[baseline={([yshift=-.5ex] current bounding box.center)}, scale=0.65]
    \draw [gray!50, ultra thick] (0, -1) to (0, 1);
    \draw [gray!50, ultra thick] (-1, 0) to (1, 0);
        \draw[dashed] (1,1) -- (-1,1) -- (-1,-1) -- (1,-1) -- (1, 1);
\end{tikzpicture} \quad
\begin{tikzpicture}[baseline={([yshift=-.5ex] current bounding box.center)}, scale=0.65]
    \draw [gray!50, ultra thick] (0, -1) to (0, 1);
    \draw [gray!50, ultra thick] (-1, 0) to (1, 0);
		\draw [ultra thick, blue!75] (1, 0) to (-1, 0);
        \draw[dashed] (1,1) -- (-1,1) -- (-1,-1) -- (1,-1) -- (1, 1);
\end{tikzpicture} \quad
\begin{tikzpicture}[baseline={([yshift=-.5ex] current bounding box.center)}, scale=0.65]
    \draw [gray!50, ultra thick] (0, -1) to (0, 1);
    \draw [gray!50, ultra thick] (-1, 0) to (1, 0);
    \draw [bend left=45, looseness=1.75, ultra thick, blue!75] (0, 1) to (-1, 0);
        \draw[dashed] (1,1) -- (-1,1) -- (-1,-1) -- (1,-1) -- (1, 1);
\end{tikzpicture} \quad
\begin{tikzpicture}[baseline={([yshift=-.5ex] current bounding box.center)}, scale=0.65]
    \draw [gray!50, ultra thick] (0, -1) to (0, 1);
    \draw [gray!50, ultra thick] (-1, 0) to (1, 0);
		\draw [bend left=45, looseness=1.75, ultra thick, blue!75] (0, 1) to (-1, 0);
		\draw [bend left=315, looseness=1.75, ultra thick, blue!75] (1, 0) to (0, -1);
        \draw[dashed] (1,1) -- (-1,1) -- (-1,-1) -- (1,-1) -- (1, 1);
\end{tikzpicture},
\end{equation}
the last of which would be disallowed by SAWs.
We can obtain `boundary' vertex configurations by truncating some of the path on the bulk vertex configurations. This generates the following boundary vertex configurations
\begin{equation} \label{eq:boundary-vertices-square}
\begin{tikzpicture}[baseline={([yshift=-.5ex] current bounding box.center)}, scale=0.65]
    \draw [gray!50, ultra thick] (0, -1) to (0, 1);
    \draw [gray!50, ultra thick] (-1, 0) to (1, 0);
    \draw[dashed] (1,1) -- (-1,1) -- (-1,-1) -- (1,-1) -- (1, 1);
    \draw[ultra thick, blue!75] (0.2, 0) -- (1, 0);
    \draw[ultra thick, blue!75] (0.2, 0.1) -- (0.2, -0.1);
\end{tikzpicture} \quad
\begin{tikzpicture}[baseline={([yshift=-.5ex] current bounding box.center)}, scale=0.65]
    \draw [gray!50, ultra thick] (0, -1) to (0, 1);
    \draw [gray!50, ultra thick] (-1, 0) to (1, 0);
    \draw[dashed] (1,1) -- (-1,1) -- (-1,-1) -- (1,-1) -- (1, 1);
    \draw[ultra thick, blue!75] (0.2, 0) -- (1, 0);
    \draw[ultra thick, blue!75] (0.2, 0.1) -- (0.2, -0.1);
    \draw[ultra thick, blue!75] (-0.2, 0) -- (-1, 0);
    \draw[ultra thick, blue!75] (-0.2, 0.1) -- (-0.2, -0.1);
\end{tikzpicture} \quad
\begin{tikzpicture}[baseline={([yshift=-.5ex] current bounding box.center)}, scale=0.65]
    \draw [gray!50, ultra thick] (0, -1) to (0, 1);
    \draw [gray!50, ultra thick] (-1, 0) to (1, 0);
    \draw [bend left=45, looseness=1.75, ultra thick, blue!75] (0, 1) to (-1, 0);
    \draw[dashed] (1,1) -- (-1,1) -- (-1,-1) -- (1,-1) -- (1, 1);
    \draw[ultra thick, blue!75] (0.2, 0) -- (1, 0);
    \draw[ultra thick, blue!75] (0.2, 0.1) -- (0.2, -0.1);
\end{tikzpicture} \quad
\begin{tikzpicture}[baseline={([yshift=-.5ex] current bounding box.center)}, scale=0.65]
    \draw [gray!50, ultra thick] (0, -1) to (0, 1);
    \draw [gray!50, ultra thick] (-1, 0) to (1, 0);
    \draw[ultra thick, blue!75] (0, 0.2) -- (0, 1);
    \draw[ultra thick, blue!75] (-0.1, 0.2) -- (0.1, 0.2);
    \draw[ultra thick, blue!75] (0.2, 0) -- (1, 0);
    \draw[ultra thick, blue!75] (0.2, 0.1) -- (0.2, -0.1);
    \draw[dashed] (1,1) -- (-1,1) -- (-1,-1) -- (1,-1) -- (1, 1);
\end{tikzpicture}.
\end{equation}
Then, SOWs can be sufficiently specified by vertex configuration generators that can yield all bulk and boundary vertex blocks upon rotation, reflection, and truncation.

Note that in order to have osculating vertices, crossings must be possible. Therefore, there must be at least 4 edges connected to a vertex. Since a vertex on the hexagonal lattice is only connected to three edges, there cannot be any crossings. This means that $\mathrm{SAW}_{\hexagon} = \mathrm{SOW}_{\hexagon}$. The term `osculating' is borrowed from \cite{jensen1998self}, who introduced it for closed paths, or `polygons'. These polygons are identified up to translations.

SOWs can also be motivated by considering domain walls. Consider boolean variables that live on the faces of the lattice. Whenever neighbouring booleans differ, one can assign a domain wall in the edge that separate the two faces (i.e. the one that is neighboured by both of them). If we want to draw the domain walls so that they do not cross, for connectivity higher than $3$, there are multiple ways of assigning them. Here, one could choose a convention. For example regions with $0$s (unfilled circles) could stay together while regions with $1$s (filled circles) stay apart and only osculate or `kiss':
\begin{equation}
    \begin{tikzpicture}[baseline={([yshift=-.5ex] current bounding box.center)}, scale=0.65]
    \draw[dashed] (1,1) -- (-1,1) -- (-1,-1) -- (1,-1) -- (1, 1);
    \fill[black] (-0.5,-0.5) circle (5pt);
    \fill[white] (-0.5,0.5) circle (5pt);
    \draw[thick] (-0.5,0.5) circle (5pt);
    \fill[white] (0.5,-0.5) circle (5pt);
    \draw[thick] (0.5,-0.5) circle (5pt);
    \fill[black] (0.5,0.5) circle (5pt);
	\draw [ultra thick, blue!75] (0, -1) to (0, 1);
        \draw [ultra thick, blue!75] (-1, 0) to (1, 0);
    \end{tikzpicture}
    \longrightarrow
    \begin{tikzpicture}[baseline={([yshift=-.5ex] current bounding box.center)}, scale=0.65]
    \draw[dashed] (1,1) -- (-1,1) -- (-1,-1) -- (1,-1) -- (1, 1);
    \fill[black] (-0.5,-0.5) circle (5pt);
    \fill[white] (-0.5,0.5) circle (5pt);
    \draw[thick] (-0.5,0.5) circle (5pt);
    \fill[white] (0.5,-0.5) circle (5pt);
    \draw[thick] (0.5,-0.5) circle (5pt);
    \fill[black] (0.5,0.5) circle (5pt);
		\node [] (0) at (0, 1) {};
		\node [] (1) at (-1, 0) {};
		\node [] (2) at (0, -1) {};
		\node [] (3) at (1, 0) {};
    \draw [gray!50, ultra thick] (0, -1) to (0, 1);
    \draw [gray!50, ultra thick] (-1, 0) to (1, 0);
		\draw [bend right=45, looseness=1.75, ultra thick, blue!75] (0.center) to (3.center);
		\draw [bend left=45, looseness=1.75, ultra thick, blue!75] (1.center) to (2.center);
    \end{tikzpicture}
    \rightarrow
    \begin{tikzpicture}[baseline={([yshift=-.5ex] current bounding box.center)}, scale=0.65]
    \node [] (0) at (0, 1) {};
    \node [] (1) at (-1, 0) {};
    \node [] (2) at (0, -1) {};
    \node [] (3) at (1, 0) {};
    \draw [gray!50, ultra thick] (0, -1) to (0, 1);
    \draw [gray!50, ultra thick] (-1, 0) to (1, 0);
    \draw [bend right=45, looseness=1.75, ultra thick, blue!75] (0.center) to (3.center);
    \draw [bend left=45, looseness=1.75, ultra thick, blue!75] (1.center) to (2.center);
    \draw[dashed] (1,1) -- (-1,1) -- (-1,-1) -- (1,-1) -- (1, 1);
    \end{tikzpicture}.
\end{equation}

We will refer to the restricted walk model specified by bulk vertex configurations generated by iterating through all possible configurations of boolean variables, then \textit{removing} the boolean variables, as osculating domain wall walks (ODWs). Again, truncating edges from these will generate the boundary vertex configurations. ODWs are subsets of SOWs.

For the square lattice, this procedure produces all vertex configurations of $\mathrm{SOW}_\square$. However, this is not true for the triangular lattice. Here, some vertex configurations allowed by crossings modified to be osculating are not generated by the above procedure. For example,
\begin{align}
    \begin{tikzpicture}[baseline={([yshift=-.5ex] current bounding box.center)}, scale=0.75]
    		\node [] (0) at (-0.5, {sqrt(3)/2}) {};
    		\node [] (1) at (0.5, {sqrt(3)/2}) {};
    		\node [] (2) at (1, 0) {};
    		\node [] (3) at (0.5, -{sqrt(3)/2}) {};
    		\node [] (4) at (-0.5, -{sqrt(3)/2}) {};
    		\node [] (5) at (-1, 0) {};
                \draw[ultra thick, gray!50] (1.center) to (4.center);
                \draw[ultra thick, gray!50] (2.center) to (5.center);
                \draw[ultra thick, gray!50] (0.center) to (3.center);
    		\draw[ultra thick, blue!75] (1.center) to (4.center);
    		\draw [bend left=300, looseness=2.50, ultra thick, blue!75] (2.center) to (3.center);
    		\draw [bend left=60, looseness=2.50, ultra thick, blue!75] (0.center) to (5.center);
            \draw[dashed] (0) -- (1) -- (2) -- (3) -- (4) -- (5) -- (0);
    \end{tikzpicture}
    \; \text{is not allowed by ODW, while} \;
    \begin{tikzpicture}[baseline={([yshift=-.5ex] current bounding box.center)}, scale=0.75]
		\node [] (0) at (-0.5, {sqrt(3)/2}) {};
		\node [] (1) at (0.5, {sqrt(3)/2}) {};
		\node [] (2) at (1, 0) {};
		\node [] (3) at (0.5, -{sqrt(3)/2}) {};
		\node [] (4) at (-0.5, -{sqrt(3)/2}) {};
		\node [] (5) at (-1, 0) {};
                \draw[ultra thick, gray!50] (1.center) to (4.center);
                \draw[ultra thick, gray!50] (2.center) to (5.center);
                \draw[ultra thick, gray!50] (0.center) to (3.center);
		\draw [bend left=300, looseness=2.50, ultra thick, blue!75] (2.center) to (3.center);
		\draw [bend right=60, looseness=2.50, ultra thick, blue!75] (4.center) to (5.center);
		\draw [bend right=60, looseness=2.50, ultra thick, blue!75] (0.center) to (1.center);
        \draw[dashed] (0) -- (1) -- (2) -- (3) -- (4) -- (5) -- (0);
                \node[circle, inner sep=0, draw=black, fill=black, minimum size=5, thick] at (0, {0.7*sqrt(3)/2}) {};
                \node[circle, inner sep=0, draw=black, fill=white, minimum size=5, thick] at (0, {-0.7*sqrt(3)/2}) {};
                \node[circle, inner sep=0, draw=black, fill=white, minimum size=5, thick] at ({0.7*sqrt(3)/2*sqrt(3)/2}, {0.7*sqrt(3)/2*1/2})  {};
                \node[circle, inner sep=0, draw=black, fill=black, minimum size=5, thick] at ({0.7*sqrt(3)/2*sqrt(3)/2}, {-0.7*sqrt(3)/2*1/2}){};
                \node[circle, inner sep=0, draw=black, fill=white, minimum size=5, thick] at ({-0.7*sqrt(3)/2*sqrt(3)/2}, {0.7*sqrt(3)/2*1/2}) {};
                \node[circle, inner sep=0, draw=black, fill=black, minimum size=5, thick] at ({-0.7*sqrt(3)/2*sqrt(3)/2}, {-0.7*sqrt(3)/2*1/2}) {};
    \end{tikzpicture}
    \longrightarrow
    \begin{tikzpicture}[baseline={([yshift=-.5ex] current bounding box.center)}, scale=0.75]
		\node [] (0) at (-0.5, {sqrt(3)/2}) {};
		\node [] (1) at (0.5, {sqrt(3)/2}) {};
		\node [] (2) at (1, 0) {};
		\node [] (3) at (0.5, -{sqrt(3)/2}) {};
		\node [] (4) at (-0.5, -{sqrt(3)/2}) {};
		\node [] (5) at (-1, 0) {};
                \draw[ultra thick, gray!50] (1.center) to (4.center);
                \draw[ultra thick, gray!50] (2.center) to (5.center);
                \draw[ultra thick, gray!50] (0.center) to (3.center);
		\draw [bend left=300, looseness=2.50, ultra thick, blue!75] (2.center) to (3.center);
		\draw [bend right=60, looseness=2.50, ultra thick, blue!75] (4.center) to (5.center);
		\draw [bend right=60, looseness=2.50, ultra thick, blue!75] (0.center) to (1.center);
        \draw[dashed] (0) -- (1) -- (2) -- (3) -- (4) -- (5) -- (0);
    \end{tikzpicture}
    \; \text{is.}
\end{align}
Therefore, $\rm{ODW}_\triangle \subset \rm{SOW}_\triangle$. However, they are equal for the square and hexagonal lattice, $\rm{ODW}_\square = \rm{SOW}_\square$, $\rm{ODW}_{\hexagon} = \rm{SOW}_{\hexagon}$. In physics, SOWs and ODWs have appeared recently in the context of two dimensional quantum error correcting codes \cite{pkim2025rigorous}.

Some of the natural questions about restricted walks are those involving the \emph{connective constant} $\mu$, which characterises how the number of such walks grow as the length of the walk $n$ goes to infinity. It is defined as $\mu := \lim_{n \rightarrow \infty} \mu_n$, where $\mu_n := c_n^{1/n}$ and $c_n$ is the number of such walks with length $n$. Because any walk of length $(n+m)$ can be decomposed into a walk of length $n$ and length $m$ but not all joining of walks of length $n$ and $m$ is a valid restricted walk of length $(n+m)$, $c_n c_m \geq c_{n+m}$ and the sequence $(\ln c_n)_{n}$ is subadditive. Therefore Fekete's lemma proves that $\mu$ exists.

Current state-of-the-art for rigorous bounds for connective constants for SAWs is $2.62002 \leq \mu^\rm{SAW}_\square \leq 2.66235$ on the square lattice \cite{couronne2022new}, and $4.57213 \leq \mu^\rm{SAW}_\cube \leq 4.73871$ for the cubic lattice \cite{hara1993new,ponitz2000improved}. As SOWs are supersets of SAWs, the lower bounds immediately apply to SOWs. Due to subadditivity of $\ln c_n$, $\mu \leq \mu_n \forall n$ for walks (which also holds for SOWs), a simple counting of numbers of SOWs can upper bound its connective constant. The first 18 are shown in \Cref{table:num-of-sow-square}. However, this is very inefficient. In \Cref{sec:automata}, we will adapt the `automata' method of \cite{ponitz2000improved}, which allows to throw away all walks with `loops' up to a desired perimeter. This allows us to obtain improved upper bounds for SOWs, as well as other restricted walks.

\begin{table}
    \centering
\begin{tabular}{r r r}
    $n$ & $c_n$ & Upper bound for $\mu_\square^\rm{SOW}$ \\
    \hline
    1 & 4 & 4.00000 \\
    2 & 12 & 3.46411 \\
    3 & 36 & 3.30193 \\
    4 & 108 & 3.22371 \\
    5 & 300 & 3.12914 \\
    6 & 860 & 3.08378 \\
    7 & 2404 & 3.04082 \\
    8 & 6772 & 3.01190 \\
    9 & 18772 & 2.98425 \\
    10 & 52268 & 2.96363 \\
    11 & 144180 & 2.94437 \\
    12 & 398756 & 2.92905 \\
    13 & 1095164 & 2.91458 \\
    14 & 3014244 & 2.90268 \\
    15 & 8252748 & 2.89139 \\
    16 & 22631804 & 2.88185 \\
    17 & 61811108 & 2.87276 \\
    18 & 169034836 & 2.86491
    \end{tabular}
    \caption{Enumeration of SOWs on the square lattice. Each count immediately gives an upper bound to the connective constant $\mu^\rm{SOW}_\square$.}
    \label{table:num-of-sow-square}
\end{table}

There is a hierarchy among these restricted walks. One can consider directed paths/walks (DPs), where the coordinates of subsequent vertex in the Cartesian representation in at least one dimension must always increase. Every DP is a NAW. Every NAW is a SAW. Every SAW is a SOW. Every SOW is a non-reversing walk (NRW), where vertices may be visited more than once but not immediately after. Every NRW is a random walk (RW). Therefore we have
\begin{align}
    \rm{DP}_n \subset \rm{NAW}_n \subset \rm{SAW}_n \subseteq \rm{ODW}_n \subseteq \rm{SOW}_n \subset \rm{EAW}_n \subset \rm{NRW}_n \subset \rm{RW}_n,
\end{align}
where $n$ is the length of the walk. Therefore, we can also bound their connective constants as
\begin{align}
    \mu^\rm{DP} \leq \mu^\rm{NAW} \leq \mu^\rm{SAW} \leq \mu^\rm{ODW} \leq \mu^\rm{SOW} \leq \mu^\rm{EAW} \leq \mu^\rm{NRW} < \mu^\rm{RW}.
\end{align}
On the $d$-dim hypercubic lattice, $\mu^\rm{DP} = d$, $\mu^\rm{NRW} = (2d-1)$, and $\mu^\rm{RW} = 2d$.

\subsection{Restricted surfaces and generalisations} \label{sec:restricted-manifolds-intro}

\begin{table}
    \centering
\begin{tabular}{|| r | r | r | r ||}
    \hline
    Acronym & Full name & Signature $(d, k)$ & Also known as \\
    \hline
    \hline
    SAW & Self-avoiding walks & $(d, 1)$ & 
    \begin{tabular}{r@{}}
         Self-avoiding polygons \\ for $h=0$
    \end{tabular} \\
    \hline
    SOW & Self-osculating walks & $(d, 1)$ & 
    \begin{tabular}{r@{}}
         Self-osculating polygons \\ for $h=0$
    \end{tabular}
    \\
    \hline
    SAS & Self-avoiding surfaces & $(d, 2)$ & \begin{tabular}{r@{}}
         Fixed polyominoes for $(2, 2)$ \\
         (Simply connected if $h=1$)
    \end{tabular} \\
    \hline
    SOS & Self-osculating surfaces & $(d, 2)$ & \\
    \hline
    XD & Fixed polyominoids & $(3, 2)$ & \\
    \hline
    SAM & Self-avoiding manifolds & $(d, k)$ & 
    \begin{tabular}{r@{}}
         Lattice animals or \\
         $d$-dim polycubes for $(d, d)$
    \end{tabular} \\
    \hline
    SOM & Self-osculating manifolds & $(d, k)$ & \\
    \hline
    $(d,k)$-XD & $(d,k)$-Fixed polyominoids & $(d, k)$ & \begin{tabular}{r@{}}
         Fixed polysticks if $(2, 1)$ \\
         Fixed polycubes if $(3, 3)$
    \end{tabular} \\
    \hline
    ODW & \begin{tabular}{r@{}}
         Osculating domain wall \\
         hypersurfaces
    \end{tabular} & $(d, d-1)$ & \\
    \hline
    \end{tabular}
    \caption{Table of restricted manifolds and their acronyms discussed in this work. $h$ is the number of boundary components; $h=0$ denotes closed walks / surfaces / manifolds, and $h=1$ denotes a single boundary component.}
    \label{table:acronyms}
\end{table} 

Similarly to restricted walks, we can also consider restricted surfaces. An example is self-avoiding surfaces (SASs). Here, a surface is a set of faces, each of which has unit area. For SASs, two faces are considered to be connected if they neighbour the same edge. Any surface in $\rm{SAS}$ must have a single connected component. The self-avoiding restriction is that an edge cannot be neighboured by more than two faces. Surfaces are identified up to translation. For the square lattice, they are equivalent to fixed polyominoes\footnote{The term `fixed' distinguishes from `free' polyominoes, where two polyominoes are identified also up to rotations and reflections.}. We can put additional restrictions on SASs. For example, we can consider self-avoiding surfaces with $h$ boundary components, $\rm{SAS}(h)$. $\rm{SAS}_\square(h=1)$, which has only one boundary component and therefore no holes, is also referred to as simply connected fixed polynominoes. On the other hand, a fixed polyominoid (XD)\footnote{i.e. ``fiXed polyominoiDs''. This is to avoid confusion with the zoo of `poly-' objects discussed in this paper.} is constructed out of square faces in the cubic lattice still with the restriction of having one connected component, but edges can be neighboured by more than two faces. Configurations are still identified up to translation.

Again, the natural question for restricted surfaces is to determine the existence and value of a growth constant $\mu$ where $c_n$ is defined with area $n$. The following statements were proved\footnote{With one minor error, which we correct in this work to rectify the overall proof.} by \cite{van1989self}: for SASs in $d$-dim hypercubic lattices, the growth constant exists. The growth constant $\mu^\rm{SAS}_{(d)}(h) \; \forall \; h$ exists for $d \geq 2$ (except for $h=0$ in $d=2$, which is an empty set). The growth constant for $h=0$ is strictly upper bounded by that of $h \geq 1$, $\mu^\rm{SAS}_{(d)}(0) < \mu^\rm{SAS}_{(d)}(1)$, and all those with $h \geq 1$ have equal growth constants $\mu_{(d)}^\rm{SAS}(h) = \mu_{(d)}^\rm{SAS}(1)$, $\forall h \geq 1$. This is because \cite{durhuus1983self} showed a simple counting argument which can be combined with arguments of \cite{van1989self} to prove that $\mu^\rm{SAS}_{(d)}(0) \leq (2d-3)$. Meanwhile, by considering `directed walk' configurations of SASs and arguments of \cite{van1989self}, one can show that $(2d-2) \leq \mu^\rm{SAS}_{(d)}(h) \; \forall \; h \geq 1$.

One can also define self-osculating surfaces (SOSs), where more than two faces can neighbour the same edge, as long as the edges are connected in such a way that the faces do not cross each other and only `osculate', which we will define in \Cref{sec:osculating}. When only two faces neighbour an edge, they are deemed to be connected. However, when more than two faces neighbour the same edge, the connections between the faces must be specified. Clearly, SOSs are a superset of SASs. They can be generated from XDs of (say) area $n$ as the following. First, one generates all possible XDs of area $n$. Then, whenever there are more than two faces neighbouring an edge, we replace it with a set of all possible osculating edges, which specify how the faces are connected. Then, we keep the configurations that have only one connected component. Two distinct SOSs may arise from one XD, if two osculating edges both give one connected component. An example is illustrated in \Cref{fig:xd-to-sos}.

\begin{figure}[H]
    \centering
\begin{align*}
\begin{tikzpicture}[baseline={([yshift=-.5ex] current bounding box.center)}, scale=0.75]
\draw[fill=SkyBlue, opacity=0.95]
(0,0,0) -- (1,0,0) -- (1,0,-1) -- (0,0,-1) -- cycle;
\draw[fill=SkyBlue, opacity=0.95]
(0,0,0) -- (0,0,-1) -- (0,1,-1) -- (0,1,0) -- cycle;
\draw[fill=SkyBlue, opacity=0.95]
(0,-1,0) -- (0,-1,-1) -- (0,0,-1) -- (0,0,0) -- cycle;
\draw[fill=SkyBlue, opacity=0.95]
(-1,0,0) -- (0,0,0) -- (0,0,-1) -- (-1,0,-1) -- cycle;

\draw[fill=SkyBlue, opacity=0.95]
(-1,1,0) -- (0,1,0) -- (0,1,-1) -- (-1,1,-1) -- cycle;
\draw[fill=SkyBlue, opacity=0.95]
(-2,1,0) -- (-1,1,0) -- (-1,1,-1) -- (-2,1,-1) -- cycle;
\draw[fill=SkyBlue, opacity=0.95]
(-1,-1,0) -- (0,-1,0) -- (0,-1,-1) -- (-1,-1,-1) -- cycle;
\draw[fill=SkyBlue, opacity=0.95]
(-2,-1,0) -- (-1,-1,0) -- (-1,-1,-1) -- (-2,-1,-1) -- cycle;
\draw[fill=SkyBlue, opacity=0.95]
(-2,0,0) -- (-2,0,-1) -- (-2,1,-1) -- (-2,1,0) -- cycle;
\draw[fill=SkyBlue, opacity=0.95]
(-2,0,0) -- (-2,0,-1) -- (-2,-1, -1) -- (-2, -1, 0) -- cycle;
\end{tikzpicture}
\quad \longrightarrow \quad
\begin{tikzpicture}[baseline={([yshift=-.5ex] current bounding box.center)}, scale=0.75]
\draw [rounded corners, fill=SkyBlue, opacity=0.95]
(0,1,0) -- (0,0,0) -- (1,0,0) -- (1,0,-1) -- (0,0,-1) -- (0,1,-1);
\draw[opacity=0.05] (0+0.05,0+0.05,0) -- (0+0.05,0+0.05,-1);

\draw [rounded corners, fill=SkyBlue, opacity=0.95]
(0,-1,0) -- (0,0,0) -- (-1,0,0) -- (-1,0,-1) -- (0,0,-1) -- (0,-1,-1);
\draw (0-0.05,0-0.05,0) -- (0-0.05,0-0.05,-1);

\draw[fill=SkyBlue, opacity=0.8]
(-1,1,0) -- (0,1,0) -- (0,1,-1) -- (-1,1,-1) -- cycle;
\draw[fill=SkyBlue, opacity=0.95]
(-2,1,0) -- (-1,1,0) -- (-1,1,-1) -- (-2,1,-1) -- cycle;
\draw[fill=SkyBlue, opacity=0.95]
(-1,-1,0) -- (0,-1,0) -- (0,-1,-1) -- (-1,-1,-1) -- cycle;
\draw[fill=SkyBlue, opacity=0.95]
(-2,-1,0) -- (-1,-1,0) -- (-1,-1,-1) -- (-2,-1,-1) -- cycle;
\draw[fill=SkyBlue, opacity=0.95]
(-2,0,0) -- (-2,0,-1) -- (-2,1,-1) -- (-2,1,0) -- cycle;
\draw[fill=SkyBlue, opacity=0.95]
(-2,0,0) -- (-2,0,-1) -- (-2,-1, -1) -- (-2, -1, 0) -- cycle;
\end{tikzpicture}
\quad\quad
\begin{tikzpicture}[baseline={([yshift=-.5ex] current bounding box.center)}, scale=0.75]
\draw [rounded corners, fill=SkyBlue, opacity=0.95]
(0,1,0) -- (0,0,0) -- (-1,0,0) -- (-1,0,-1) -- (0,0,-1) -- (0,1,-1);
\draw[opacity=0.95] (0+0.05,0-0.05,0) -- (0+0.05,0-0.05,-1);

\draw [rounded corners, fill=SkyBlue, opacity=0.95]
(0,-1,0) -- (0,0,0) -- (1,0,0) -- (1,0,-1) -- (0,0,-1) -- (0,-1,-1);
\draw[opacity=0.95] (0-0.05,0+0.05,0) -- (0-0.05,0+0.05,-1);

\draw[rounded corners, opacity=0.95, fill=SkyBlue] (0,0.1,0) -- (0,0,0) -- (-1,0,0) -- (-1,0,-1) -- (0,0,-1) -- (0,0.1,-1);

\draw[rounded corners, opacity=0.95, fill=SkyBlue, draw=none] (0.1,0,0)--(0,0,0)--(0,-0.9,0)--(0,-0.9,-1)--(0,0,-1)--(0.1,0,-1);

\draw[fill=SkyBlue, opacity=0.8]
(-1,1,0) -- (0,1,0) -- (0,1,-1) -- (-1,1,-1) -- cycle;
\draw[fill=SkyBlue, opacity=0.95]
(-2,1,0) -- (-1,1,0) -- (-1,1,-1) -- (-2,1,-1) -- cycle;
\draw[fill=SkyBlue, opacity=0.95]
(-1,-1,0) -- (0,-1,0) -- (0,-1,-1) -- (-1,-1,-1) -- cycle;
\draw[fill=SkyBlue, opacity=0.95]
(-2,-1,0) -- (-1,-1,0) -- (-1,-1,-1) -- (-2,-1,-1) -- cycle;
\draw[fill=SkyBlue, opacity=0.95]
(-2,0,0) -- (-2,0,-1) -- (-2,1,-1) -- (-2,1,0) -- cycle;
\draw[fill=SkyBlue, opacity=0.95]
(-2,0,0) -- (-2,0,-1) -- (-2,-1, -1) -- (-2, -1, 0) -- cycle;
\end{tikzpicture}
\end{align*}
    
\caption{Example of generating a self-osculating surface from a fixed polyominoid. Choosing different connections may result in distinct self-osculating surfaces.}
\label{fig:xd-to-sos}
\end{figure}
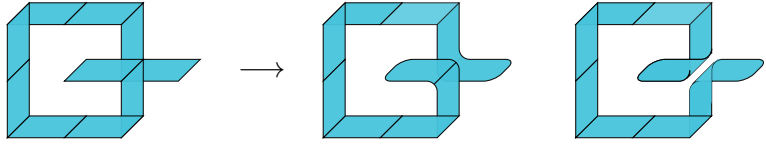

The above models can be considered to be `edge models', in the sense that apart from the single connected component condition, the other conditions are defined on edges. Other kinds of restricted surface vertex models can be created by considering vertex configurations. For example, similarly to ODWs in $d=2$, in the cubic lattice, consider a vertex. This vertex is neighboured by 8 cubes, 12 faces, and 6 edges. Place boolean variables on each of the cubes, and consider a configuration of booleans. If the boolean variables on two neighbouring cubes differ, `turn on' (or occupy) the face that separates them. To decide on the osculating structure, consider each edge neighbouring the vertex, which neighbours 4 faces (each of which neighbours the vertex). These 4 faces are boundaries of 4 of the neighbouring cubes of the vertex. This is the same as the $d=2$ case, where the 4 edges were boundaries of the 4 of the neighbouring faces of the vertex. Therefore, we can use the same osculation rules. Removing the boolean variables, we retrieve a restricted surface model defined by its bulk vertex configurations, which we will similarly refer to as $\mathrm{ODW}_\cube$. Its bulk vertices are shown in \Cref{table:bulk-sos-vertices-cube}.

\begin{table}[]
    \centering
\begin{tabular}{c c c c}
        Index & Vertex & Area & In SAS? \\
        \hline
        1 & \raisebox{-.45\height}{\includegraphics{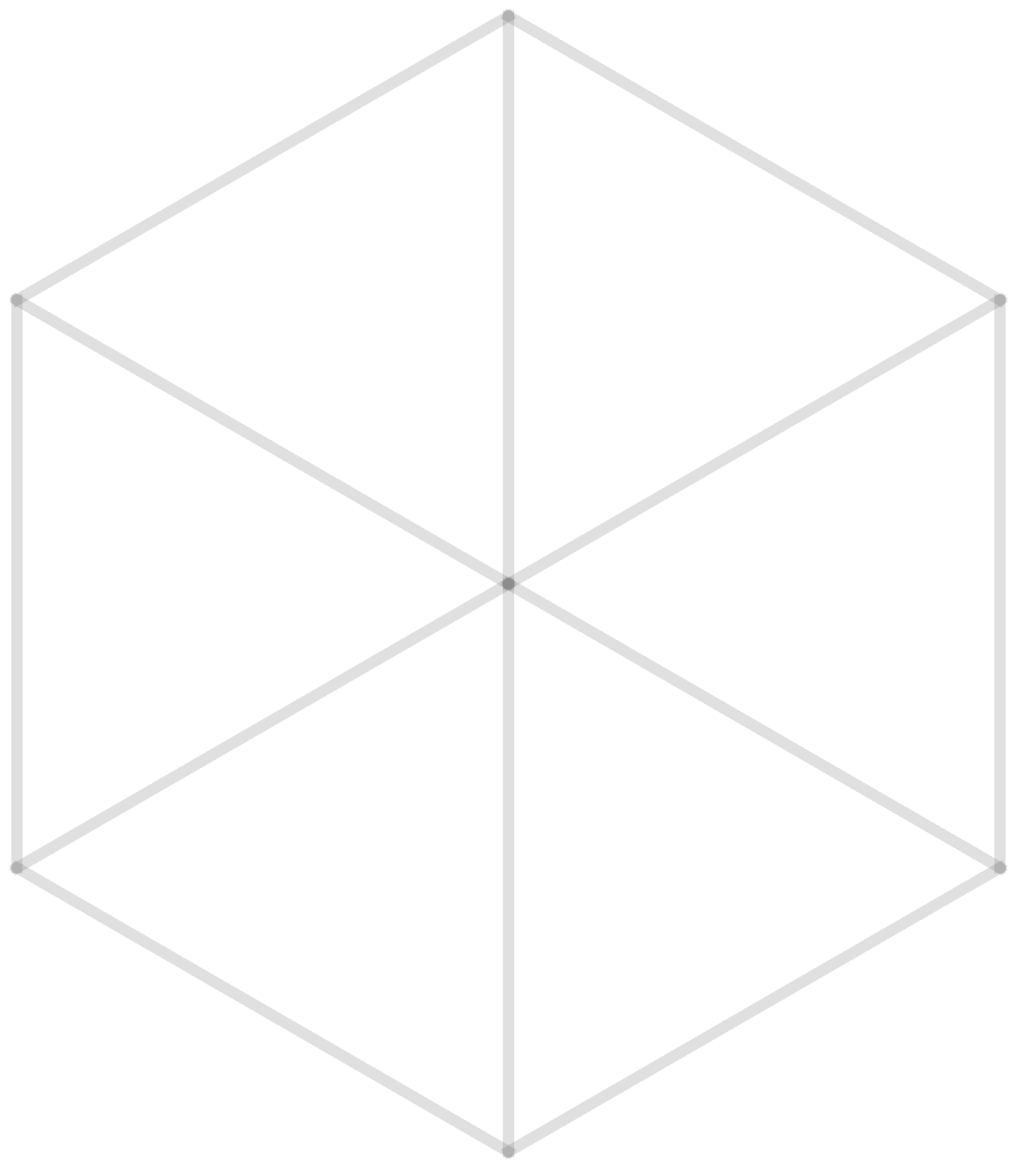}} & $0$ & Yes \\
        2 & \raisebox{-.45\height}{\includegraphics{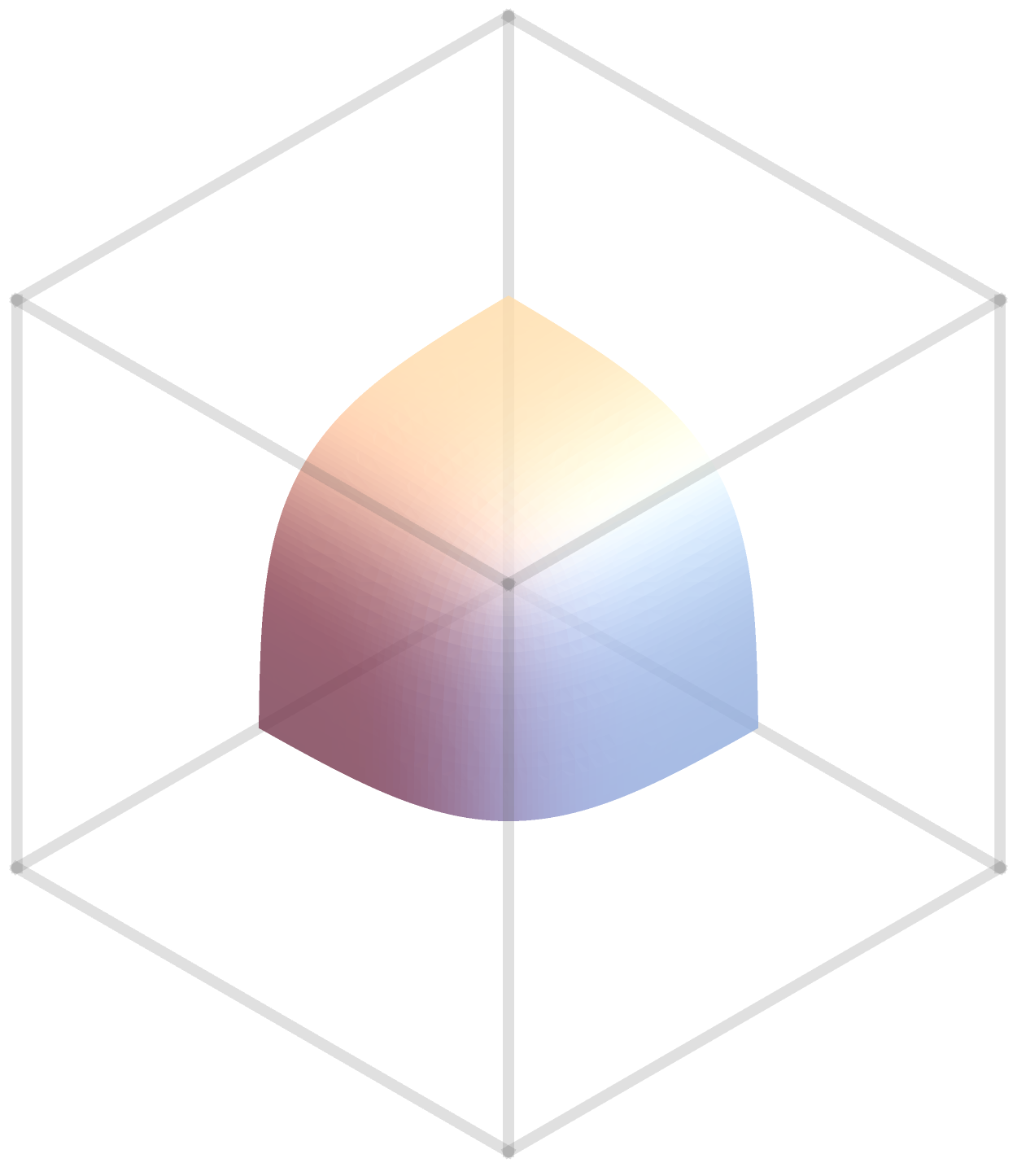}} & $3/4$ & Yes \\
        3 & \raisebox{-.45\height}{\includegraphics{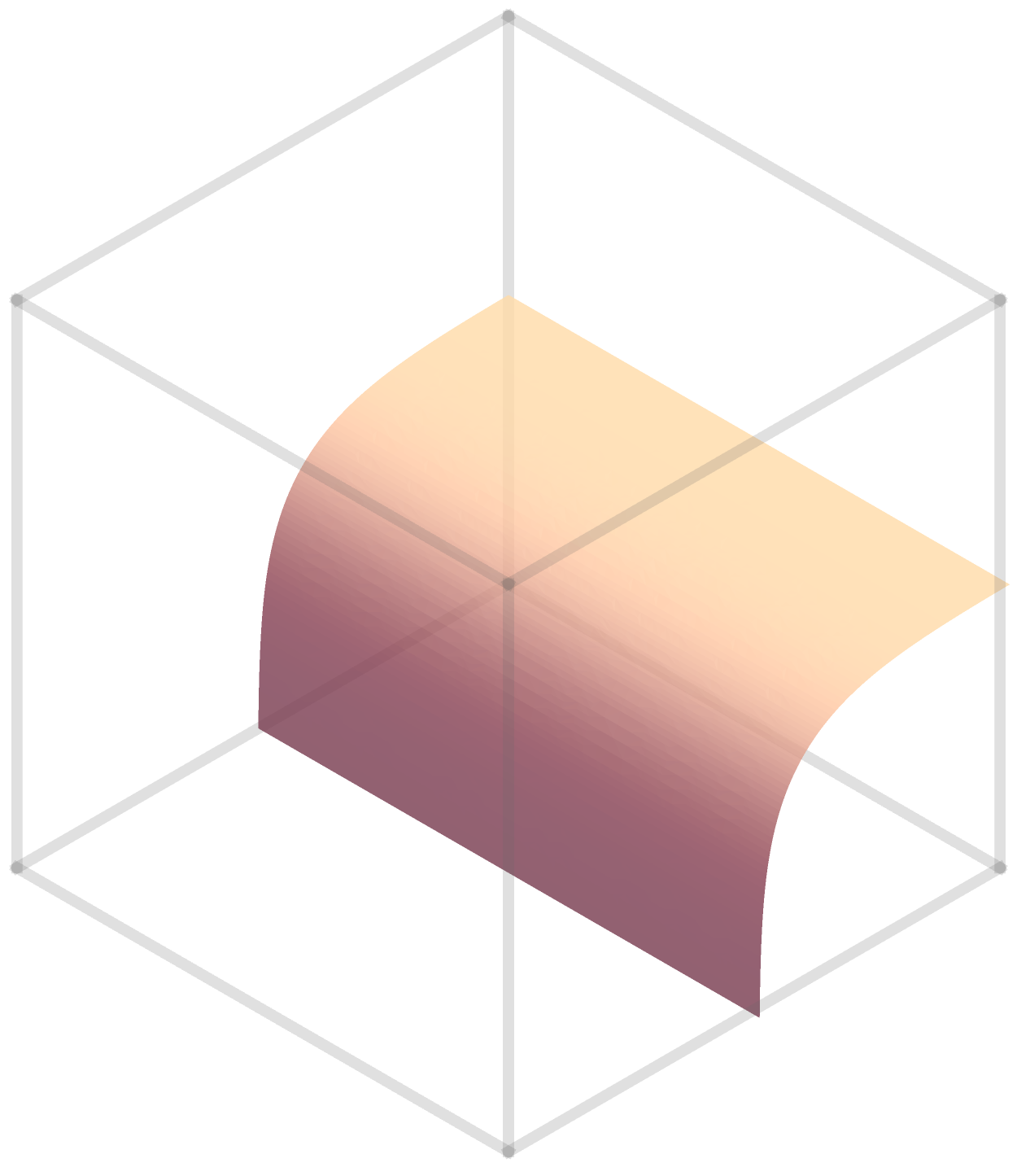}} & $1$ & Yes \\
        4 & \raisebox{-.45\height}{\includegraphics{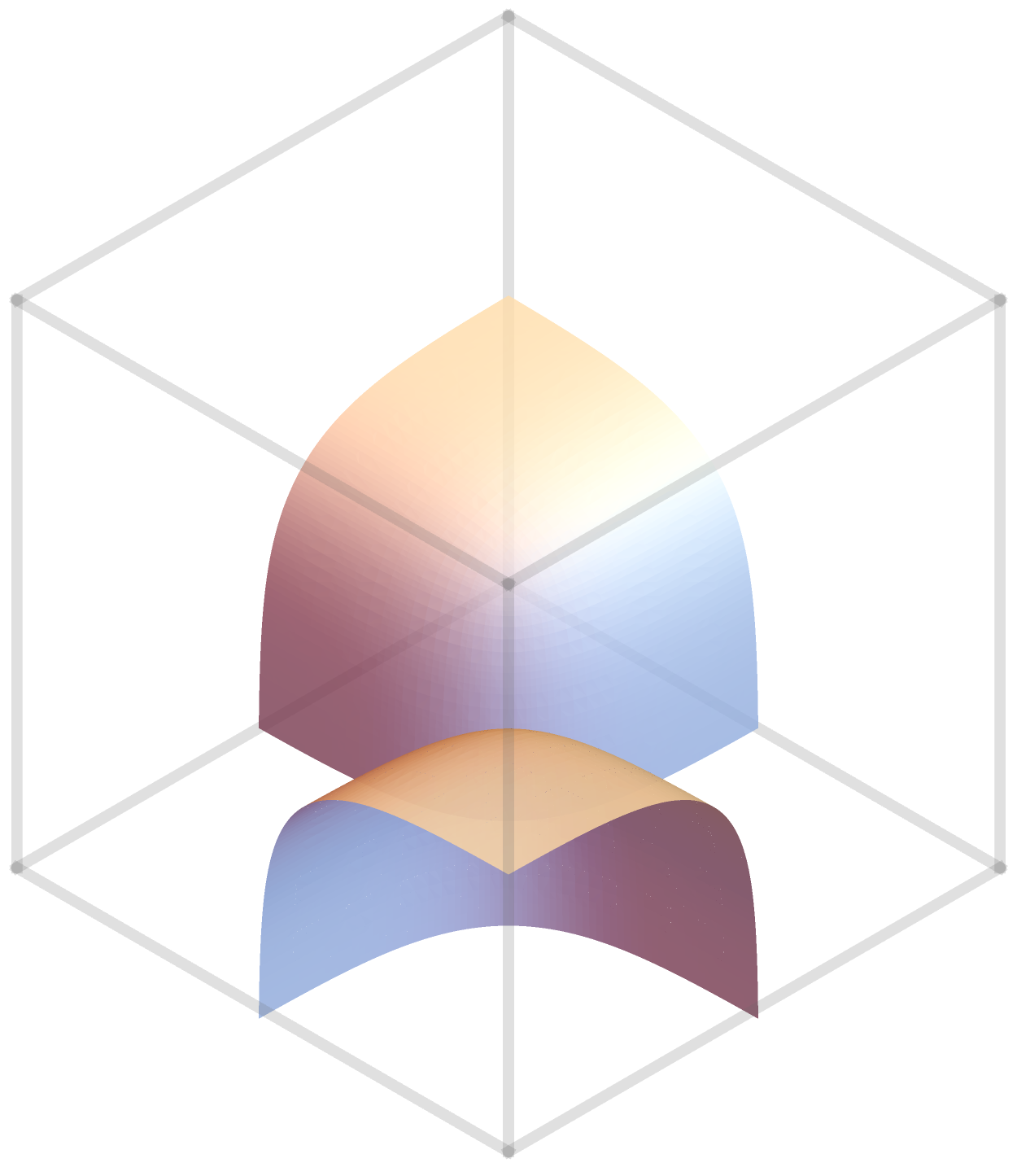}} & $6/4$ & No \\
        5 & \raisebox{-.45\height}{\includegraphics{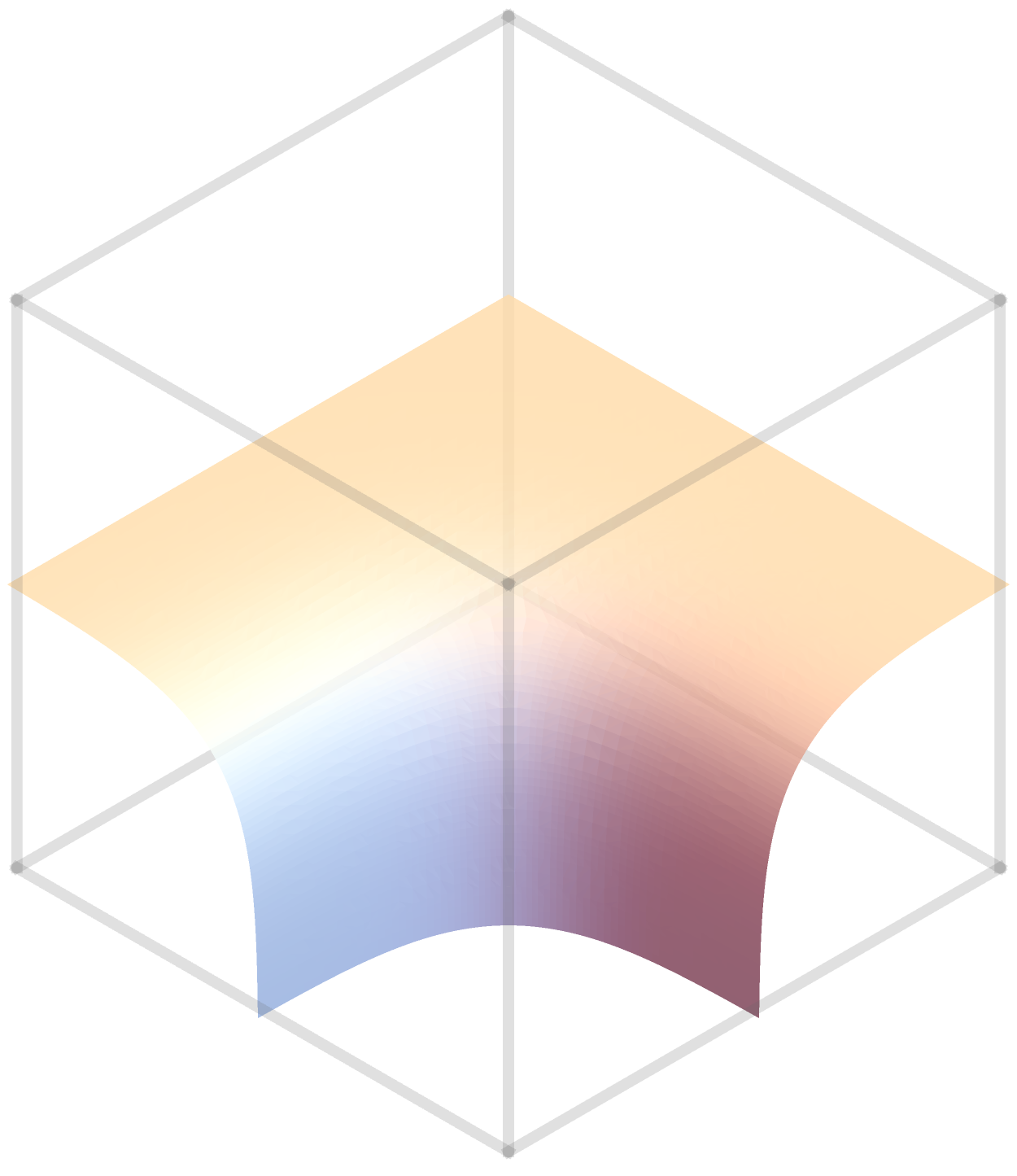}} & $5/4$ & Yes \\
        6 & \raisebox{-.45\height}{\includegraphics{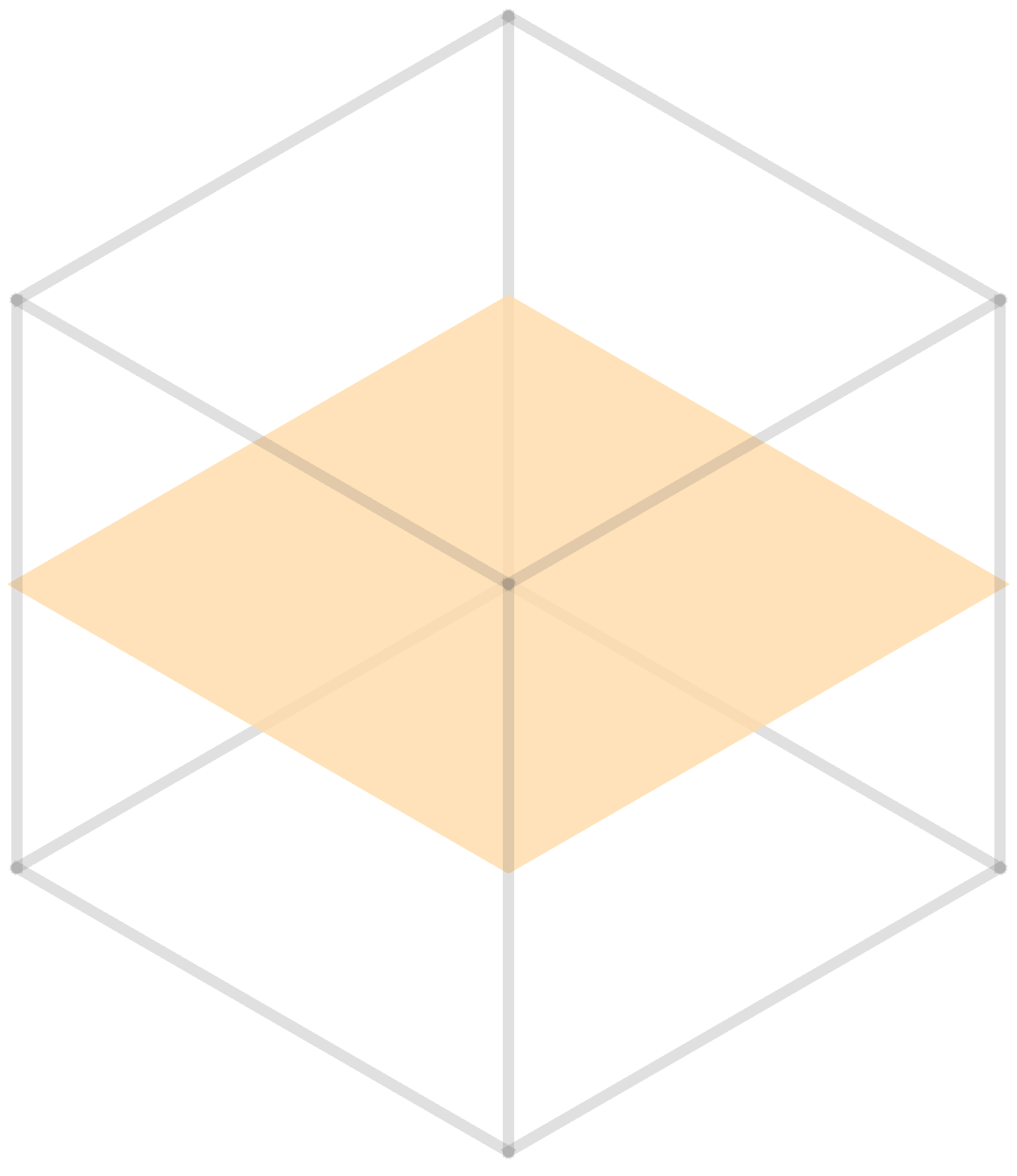}} & $1$ & Yes \\
        7 & \raisebox{-.45\height}{\includegraphics{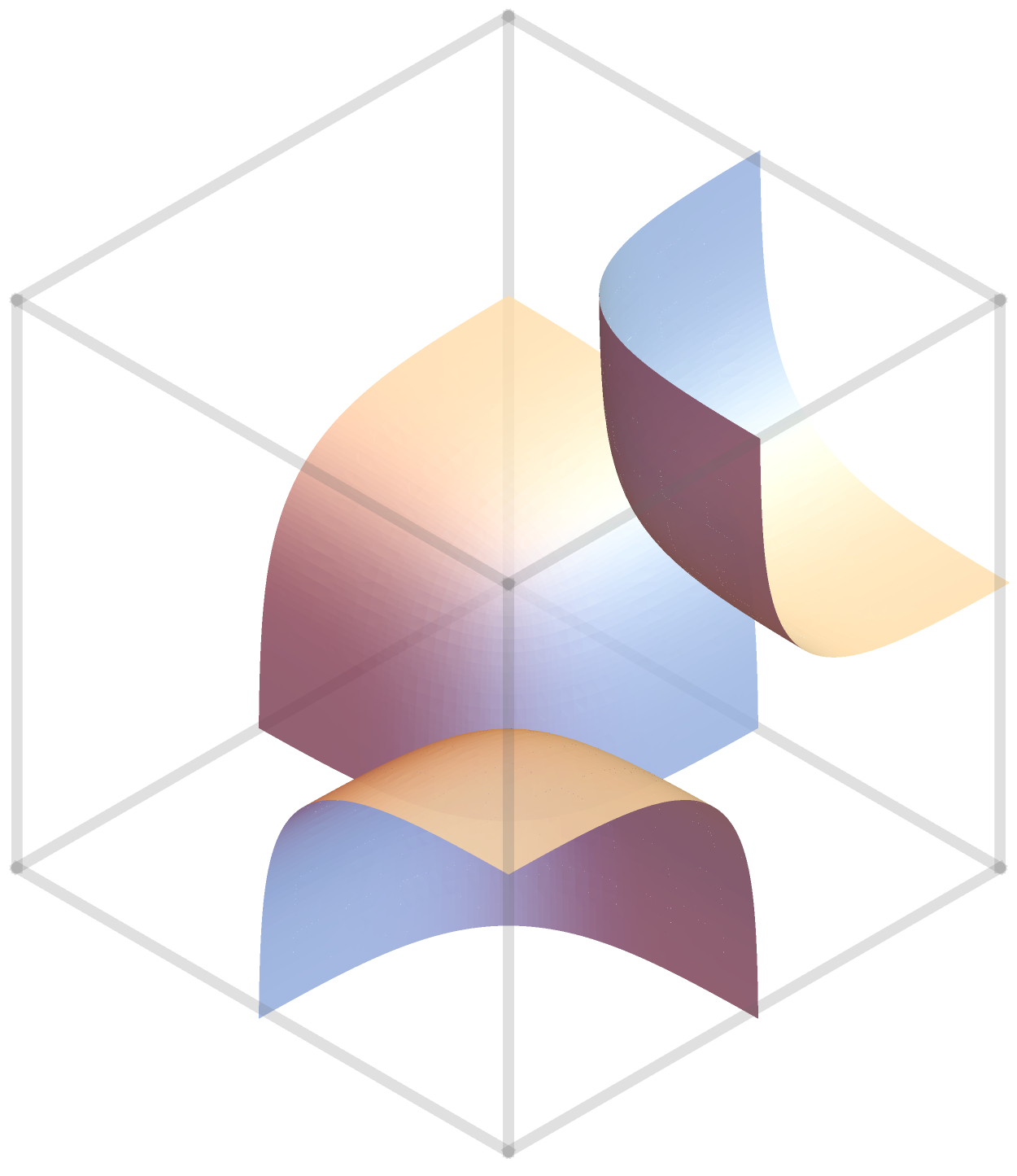}} & $9/4$ & No \\
        8 & \raisebox{-.45\height}{\includegraphics{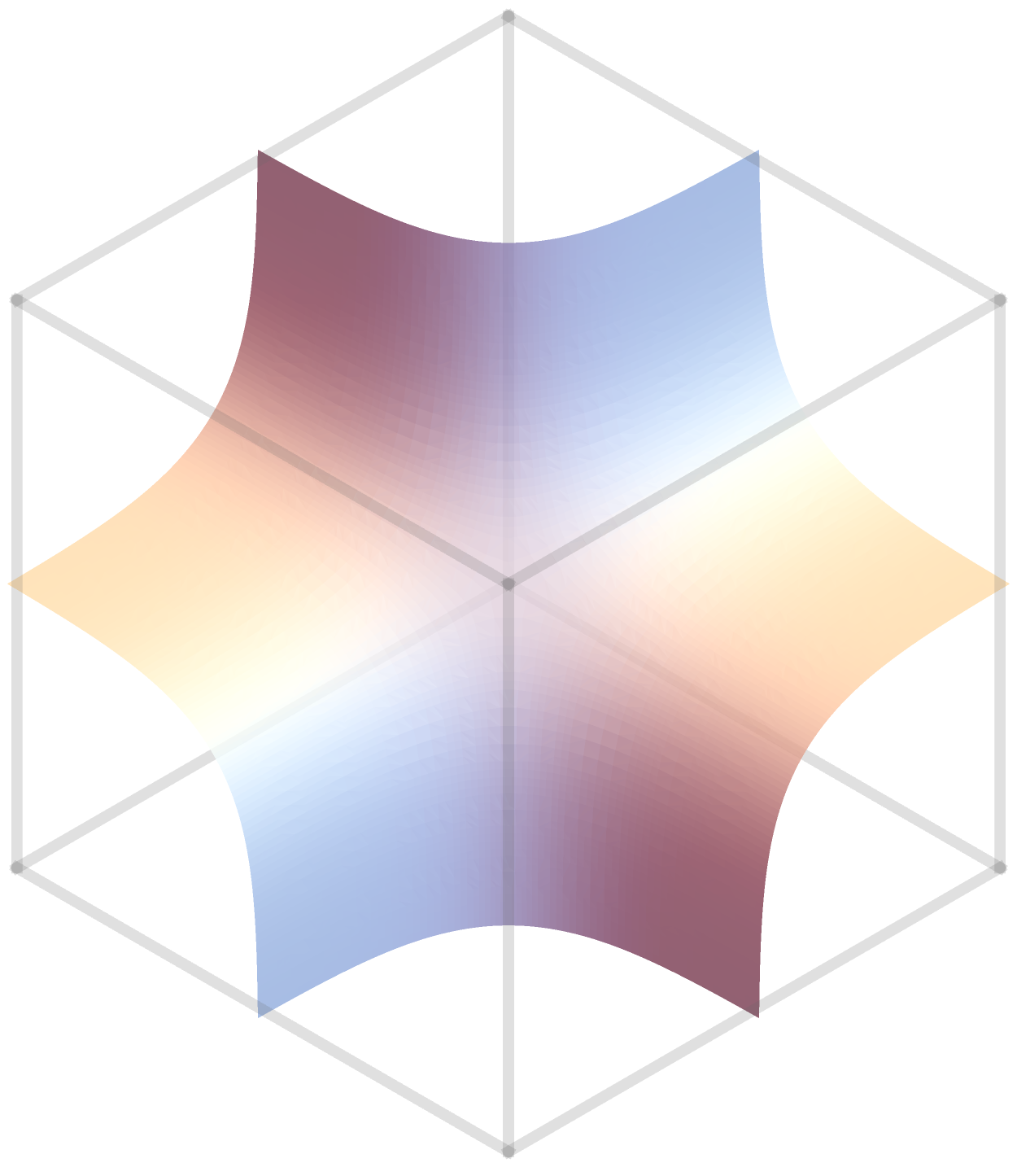}} & $6/4$ & Yes \\
        9 & \raisebox{-.45\height}{\includegraphics{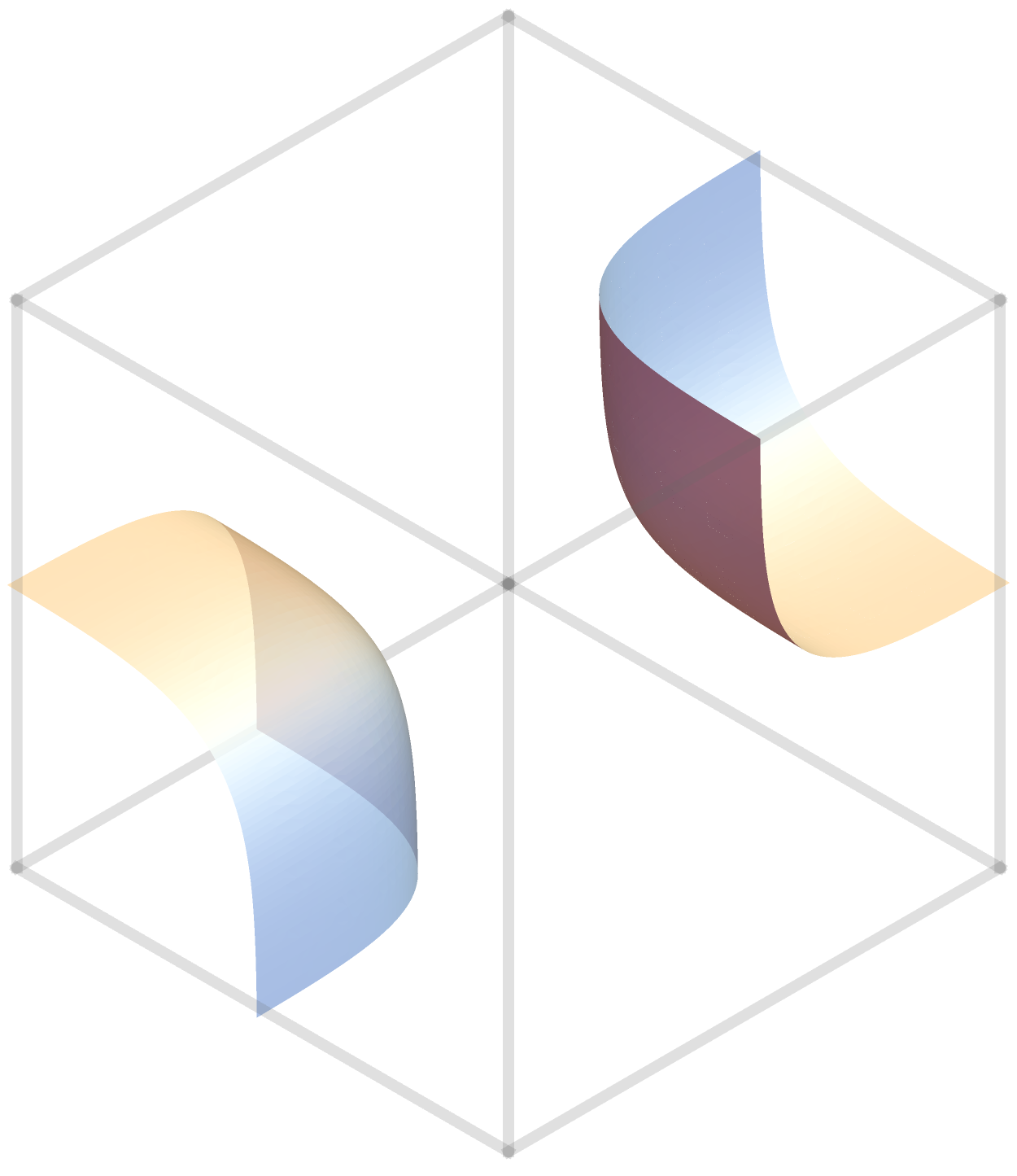}} & $6/4$ & No \\
    \end{tabular} \quad
    \begin{tabular}{c c c c}
        Index & Vertex & Area & In SAS? \\
        \hline
        10 & \raisebox{-.45\height}{\includegraphics{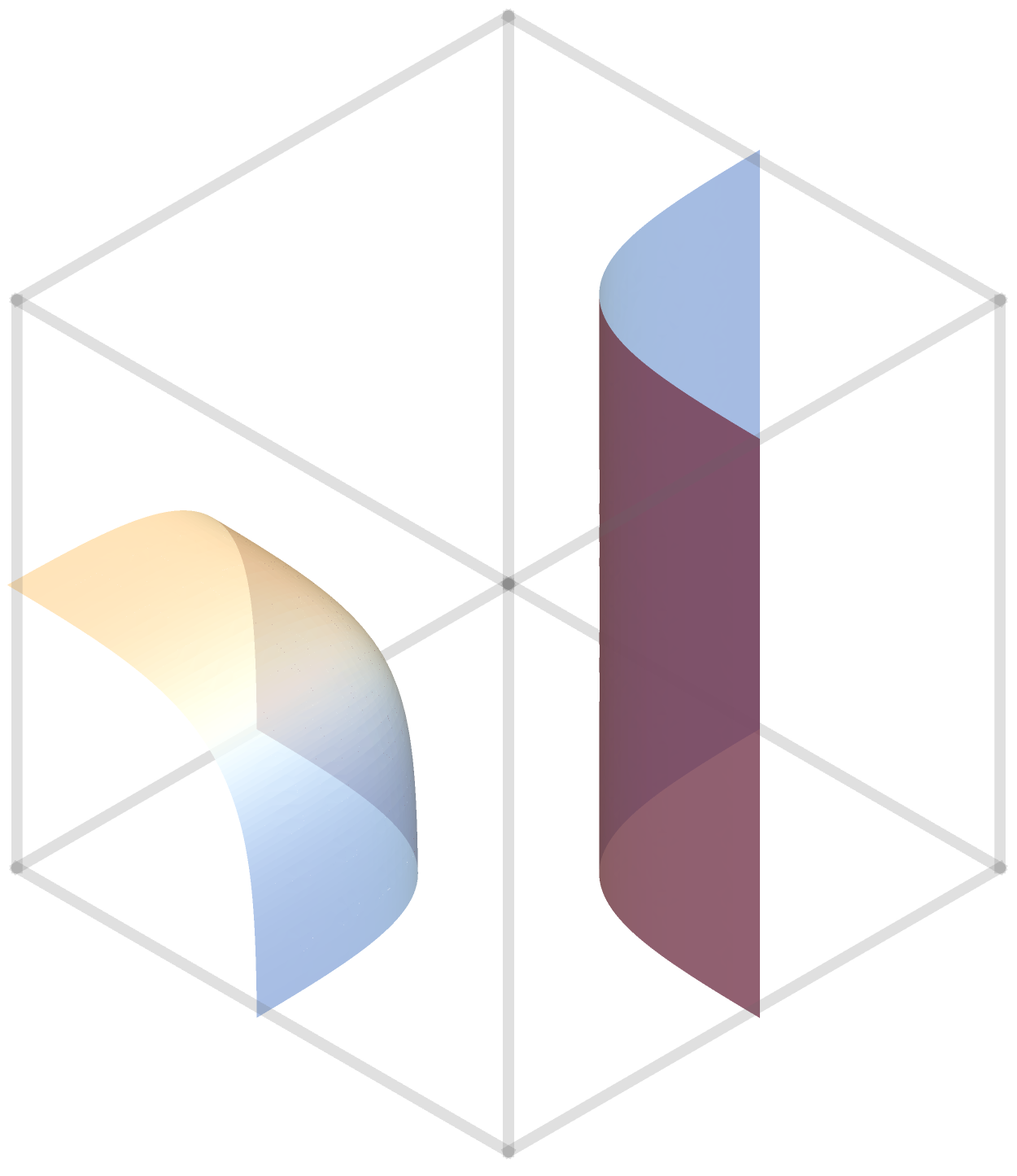}} & $7/4$ & No \\
        11 & \raisebox{-.45\height}{\includegraphics{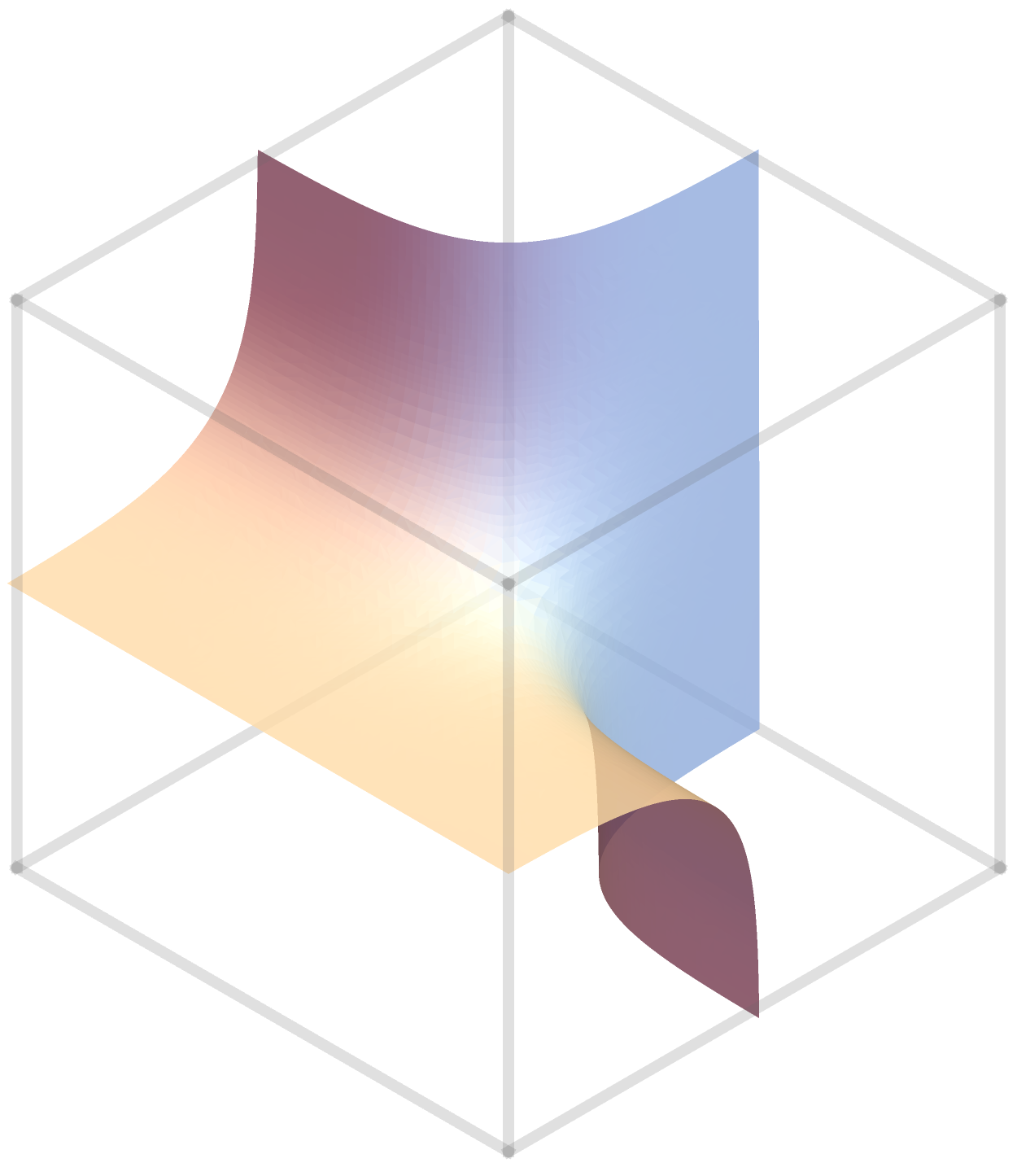}} & $6/4$ & Yes \\
        12 & \raisebox{-.45\height}{\includegraphics{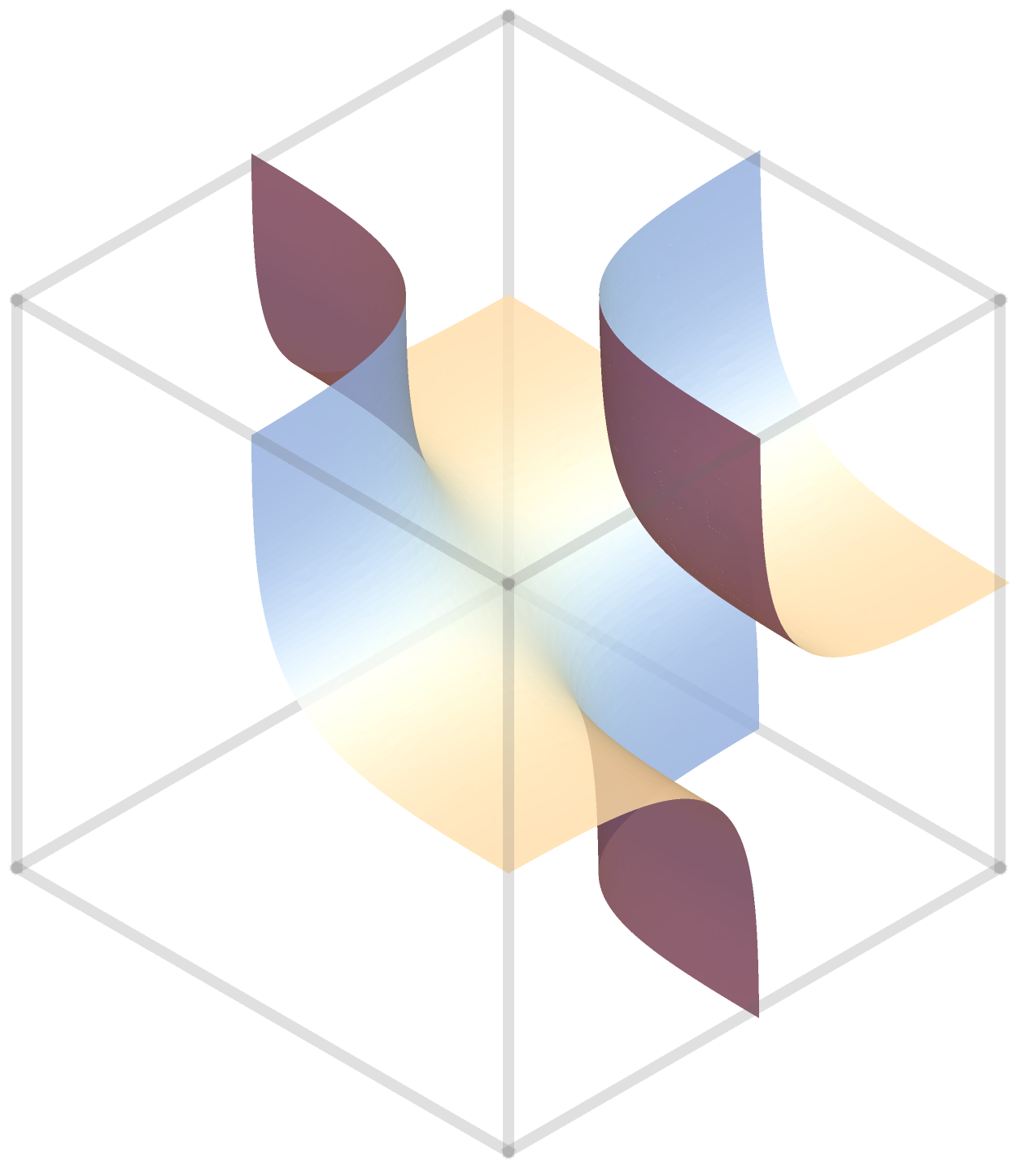}} & $9/4$ & No \\
        13 & \raisebox{-.45\height}{\includegraphics{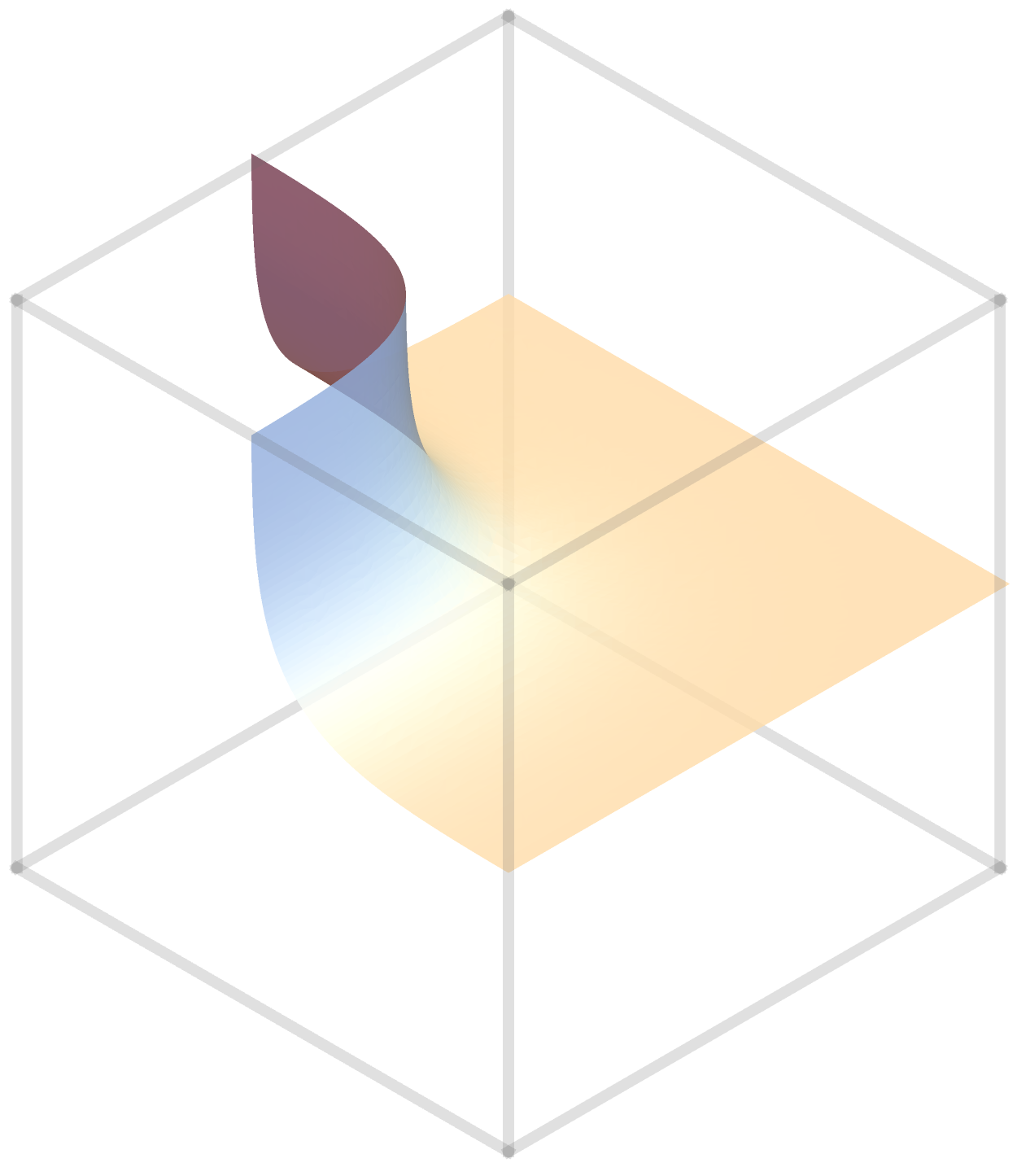}} & $5/4$ & Yes \\
        14 & \raisebox{-.45\height}{\includegraphics{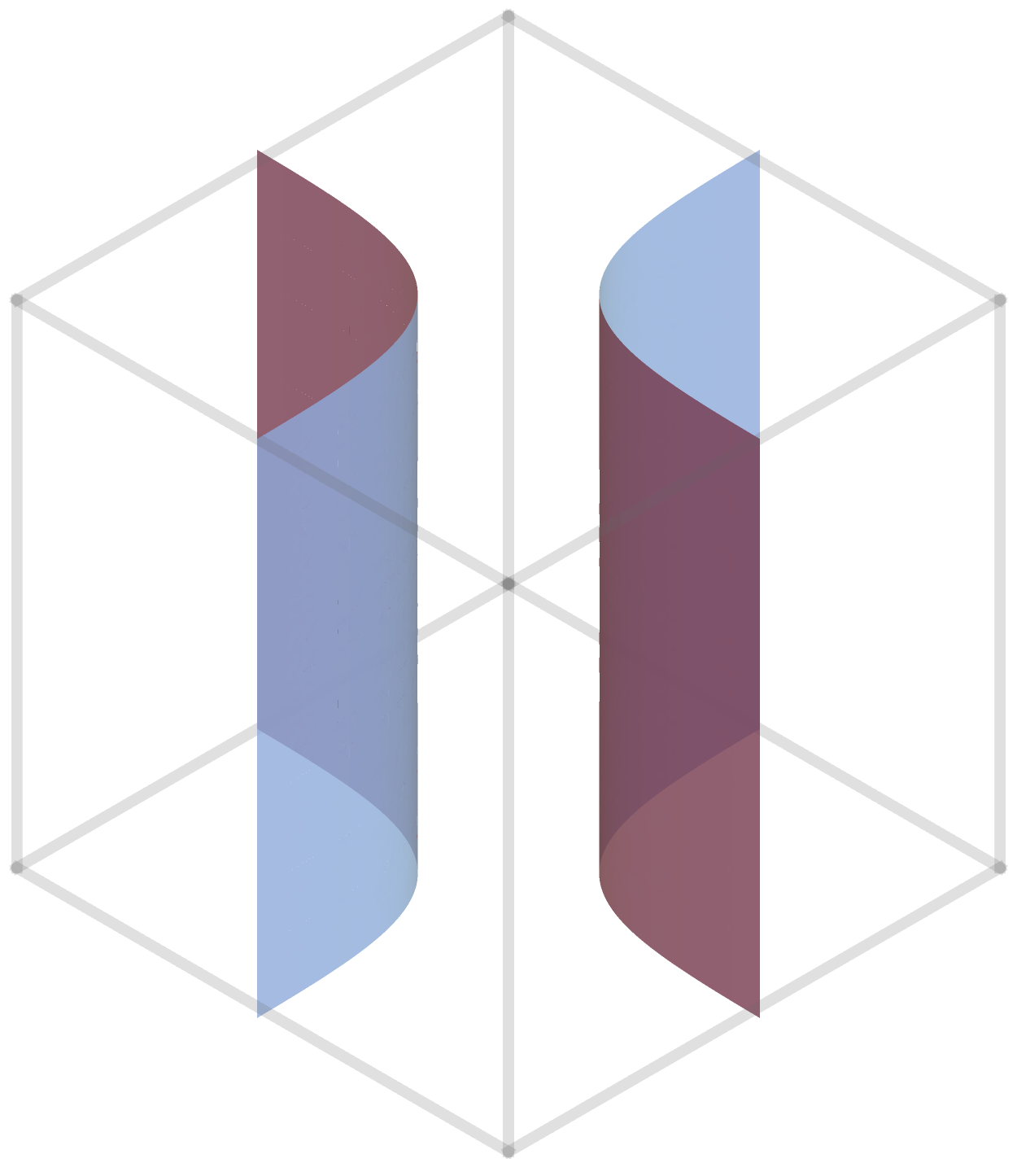}} & $2$ & No \\
        15 & \raisebox{-.45\height}{\includegraphics{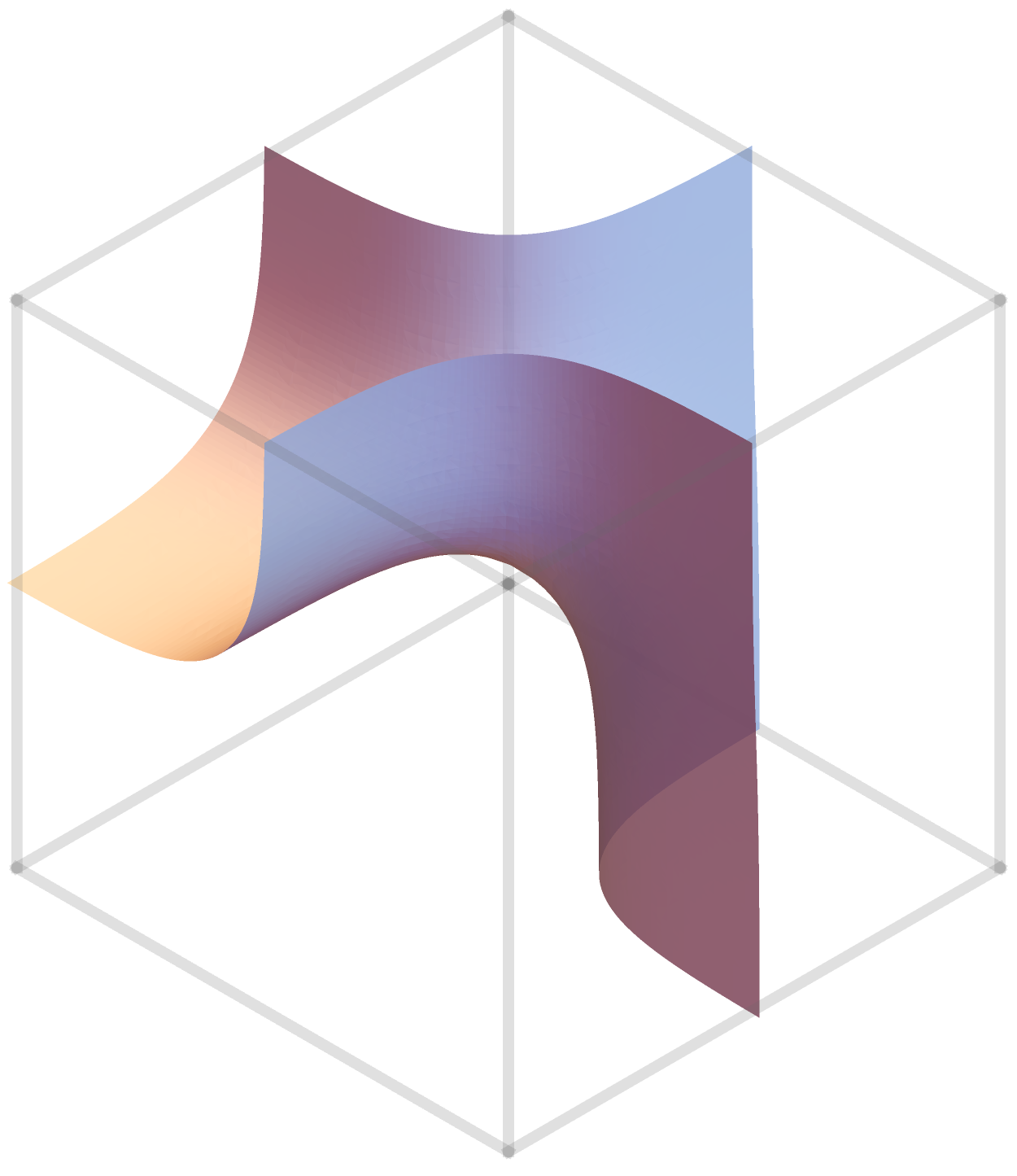}} & $7/4$ & No \\
        16 & \raisebox{-.45\height}{\includegraphics{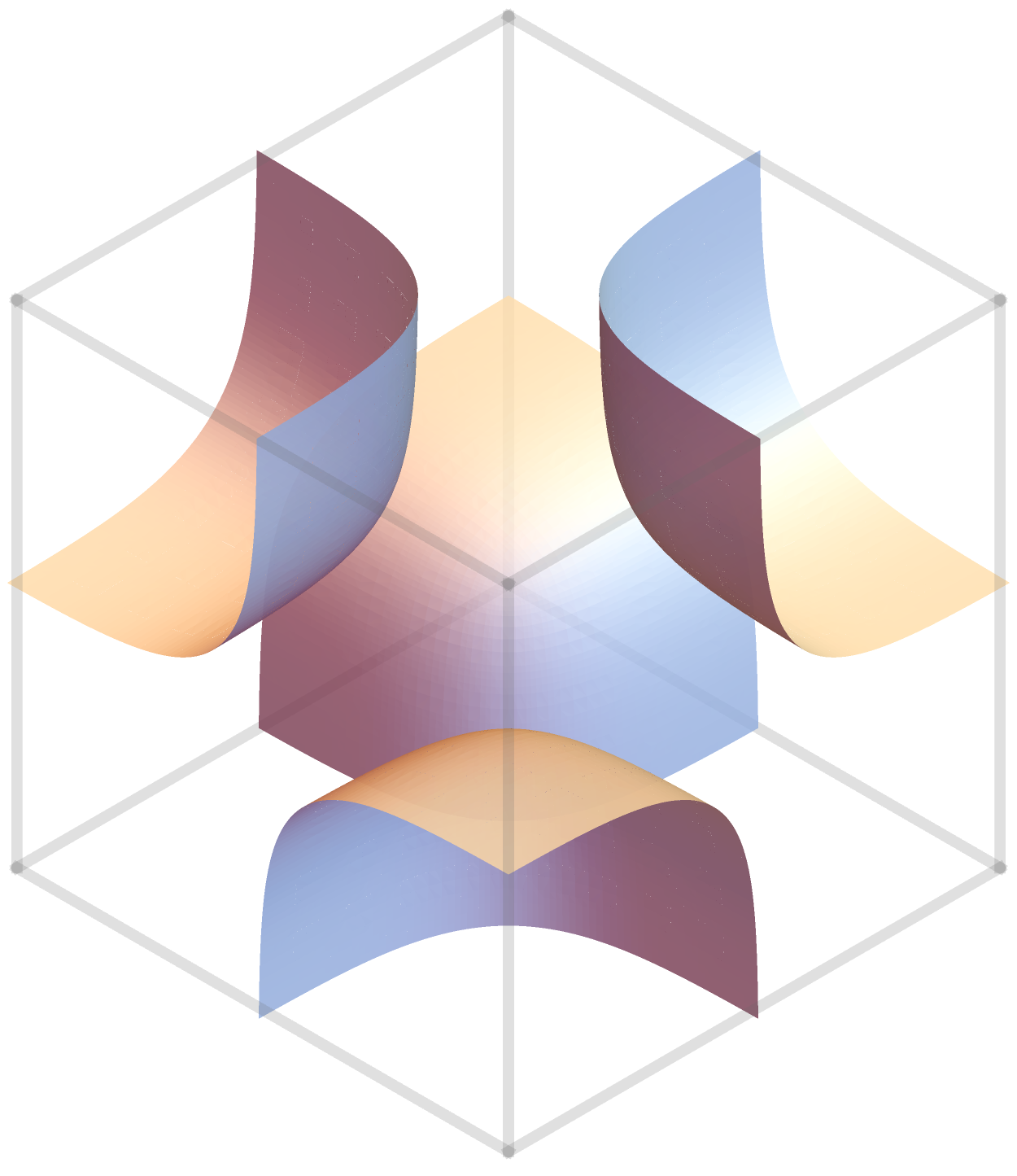}} & $3$ & No \\
        17 & \raisebox{-.45\height}{\includegraphics{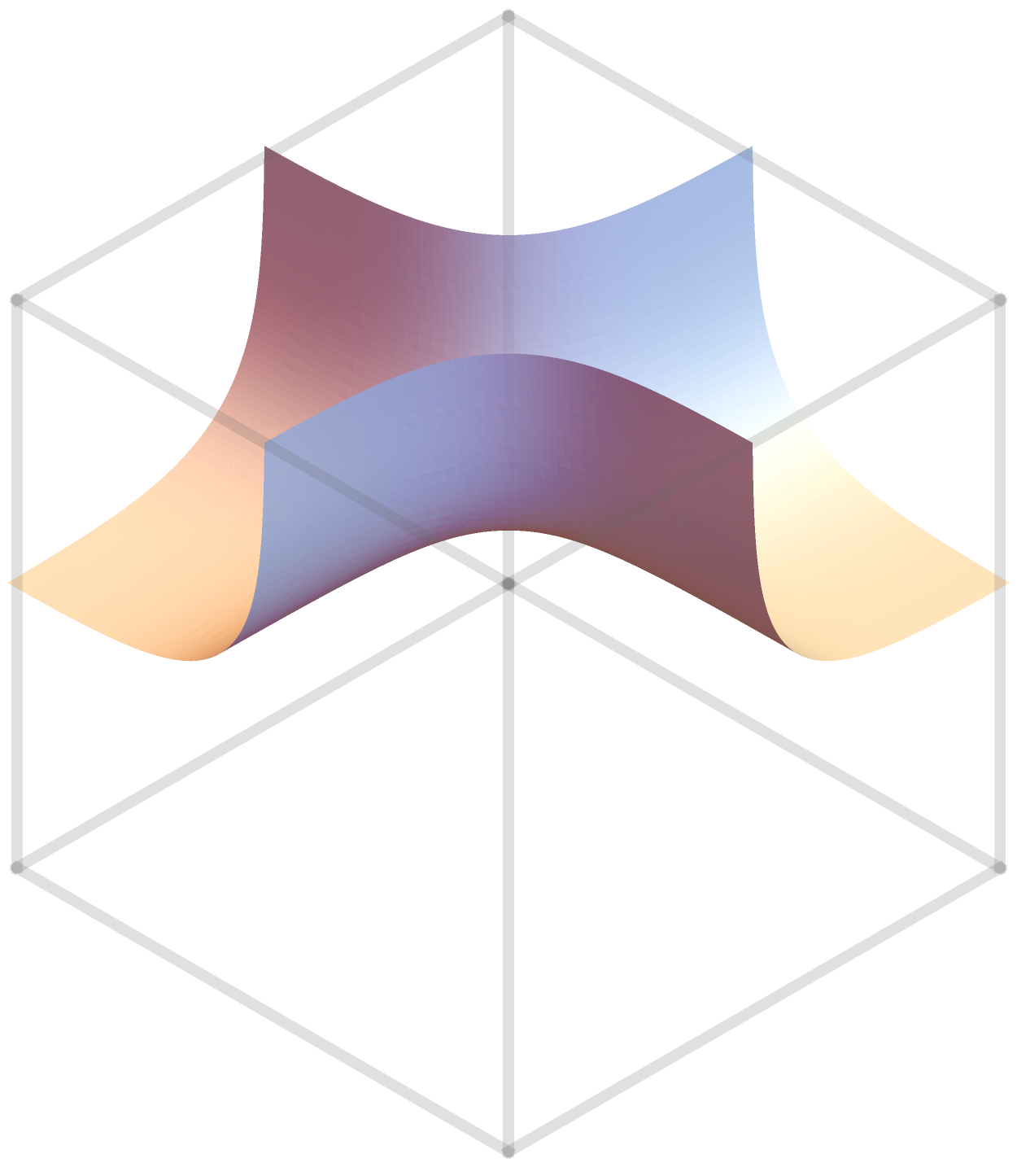}} & $6/4$ & No \\
        18 & \raisebox{-.45\height}{\includegraphics{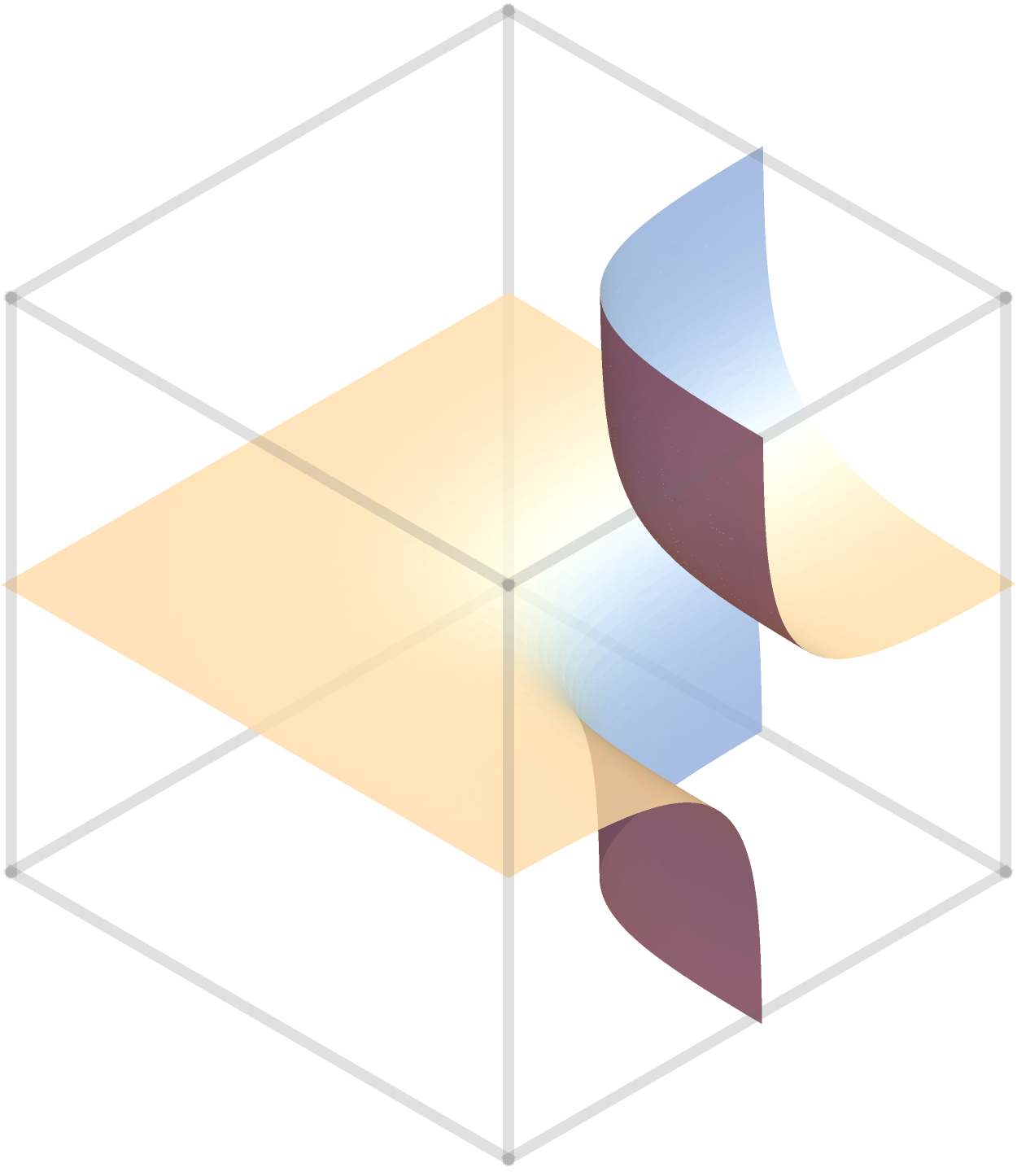}} & $2$ & No \\
    \end{tabular}
    \caption{Bulk vertex `blocks' of $\mathrm{ODW}_\cube$ up to rotations and reflections. Boundary vertex configurations can be obtained by truncating parts of faces from them.}
    \label{table:bulk-sos-vertices-cube}
\end{table}

Next, one can study further generalisations of restricted $k$-dim `manifolds' in $d$-dim hypercubic lattices, where $k \leq d$.

We can define self-avoiding $k$-dim manifolds (SAMs) on $d$-dim hypercube lattice. A $k$-manifold is defined as a set of $k$-faces. Each $k$-face has unit $k$-area. Two $k$-faces are considered to be connected if they neighbour the same $(k-1)$-edge. Any $k$-manifold in $\rm{SAM}$ must have a single connected component. The self-avoiding restriction is that a $(k-1)$-edge cannot be neighboured by more than two $k$-faces. $k$-manifolds are identified up to translation. As with SASs, we can put additional restrictions on SAMs. For example, we can consider SAMs with $h$ boundary components. 
For example, any $(k-1)$-edge neighboured by a $k$-face in a $k$-manifold in $\mathrm{SAM}_{(d,k)}(h=0)$ has two $k$-faces neighbouring it. $\rm{SAM}_{(d,k)}(1) \cong \rm{SAM}_{(d,k-1)}(0)$ since an element in the former can be uniquely identified with an element in the latter by applying the boundary map. However, since the $(k-1)$-surface areas of a $k$-volumes at fixed volume can vary, there is no isomorphism at fixed volume.

For self-osculating $k$-manifolds (SOMs), a $(k-1)$-edge can be neighboured by more than two $k$-faces as long as their specified connections osculate. SOMs supersets of SAMs. These are different from $(d, k)$-XDs, which are generalisation of XDs. In this case, more than two $k$-faces can neighbour the same $(k-1)$-edge, as long as configuration has a single connected component. Here, all $k$-faces neighbouring the same $(k-1)$-edge are deemed to be connected. SOMs can be generated from $(d, k)$-XDs in the same was as SOSs can be generated from XDs. Note that for $(d, k=d)$, $\mathrm{SAM}_{(d,d),n} = \mathrm{SOM}_{(d,d),n} = \mathrm{XD}_{(d,d),n}$. This is because the $(d-1)$-edge of two neighbouring $d$-faces only neighbour those two $d$-faces. Therefore there there can only be at most two $d$-faces neighbouring a $(d-1)$-edge.

Therefore we have a hierachy of restricted manifolds as
\begin{align}
    \rm{SAM}_{(d, k), n}(0) \subset \rm{SAM}_{(d, k), n} \subseteq \begin{cases}
        \rm{SOM}_{(d, k), n}, \\
        \rm{XD}_{(d, k), n},
    \end{cases}
\end{align}
for each $k$-area $n$. The equality holds for $(d,k=d)$. Therefore, assuming that growth constants exist, which we prove in \Cref{sec:existence-and-upper-bound}, we have
\begin{align}
    \mu^\rm{SAM}_{(d, k)}(0) \leq \mu^\rm{SAM}_{(d, k)} \leq \begin{cases}
        \mu^\rm{SOM}_{(d, k)}, \\
        \mu^\rm{XD}_{(d, k)}.
    \end{cases}
\end{align}
In the same section, we will also prove that the leftmost inequality is strictly less than ($<$) for $k>1$.

Similarly to the $(d=3, k=2)$ case, these models can be thought of as `hyperedge models', as the restriction conditions are defined on the level of $(k-1)$-edges. For $(d, k=d-1)$, we can define osculating domain wall hypersurfaces, a vertex model. We do this by similarly to osculating domain wall surfaces in $(d=3, k=2)$: boolean variables live on $d$-dimensional hypercubes and and their boundaries may osculate. They are defined in \Cref{sec:osculating}.

In this scheme, restricted walks can be related to restricted edges, i.e. $(d, k)=(d, 1)$. The important difference is that for self-avoiding edges, configurations are related by translation. Therefore walks and the reverse walks are identified. On the other hand, in the case of restricted walks, they are ordered and start from the origin. Therefore the two are related as $c_n^\rm{walks} = 2c_n^\rm{edges}$. Restricted surfaces are $(d, k) = (d, 2)$, and restricted volumes $(d, k)=(d, 3)$. Since hypersurfaces are defined to have dimension $k=d-1$, they are $(d, k=d-1)$.

Polysticks are connected edges on the square lattice, where similar to polyominoids, more than 2 edges can neighbour the same vertex; they are $(d, k)=(2,1)$. Polyominoes are connected squares on the square lattice, therefore $(d, k) = (2, 2)$. XDs (without a prefix) are $(d, k)=(3, 2)$ and are supersets of SASs in $d=3$, as multiple faces can be connected to one edge.  Polycubes are connected cubes on the cubic lattice, therefore $(d, k)=(3, 3)$. Lattice animals, sometimes referred to as ``$d$-dimensional polycubes'', are connected $d$-dimensional hypercubes in the $d$-dim hypercubic lattice; they are therefore $(d, k=d)$.

The various restricted manifold models, their acronyms, and their other names are listed in \Cref{table:acronyms}.

\subsection{Outline}
The outline of the paper is as follows. 

In \Cref{sec:definitions}, we define some restricted manifold models. First, we define self-osculating walks and osculating domain wall walks on the square and triangular lattices. We introduce the concept of center coordinates for hypercubic lattices, then define $(d,k)$-self-osculating manifolds and $d$-dim osculating domain wall hypersurfaces.

In \Cref{sec:automata}, we generalise the algorithm of \cite{ponitz2000improved}, known as the \say{automata method}. This is a generic method to obtain progressively sharper upper bounds on the connective constants of restricted walks (i.e., walks obeying a given rule). We apply this method to obtain upper bounds of connective constants for various restricted walk models such as the self-osculating walks on the square and triangular lattices.

In \Cref{sec:existence-and-upper-bound}, we prove the existence and bounds for the growth constants for restricted manifolds. The upper bounds for the number of restricted manifolds are constructed by a consistent way of labelling them. This, combined with a `concatenation' procedure, where two restricted manifolds are joined together, results in the proof of the existence of a growth constant. Lower bounds are found by considering `directed walk' configurations of manifolds.

In Section~\ref{sec:twig-method}, we discuss the twig method. This is a method originally used in \cite{eden1961two,klarner1973procedure} to systematically improve upper bounds of growth constants for polyomines, which we adapt for self-avoiding surfaces on the cubic lattice to obtain an improved upper bound for its growth constant.

Finally, in \Cref{sec:conclusion}, we conclude and give an outlook on remaining problems.

\section{Definitions of restricted manifold models} \label{sec:definitions}
\subsection{Self-osculating walks and osculating domain wall walks in $d=2$} \label{sec:osculating-vertex-definitions}
As stated in the introduction, restricted walk models can be thought of as vertex models. For clarity, we define their vertex configuration generators here.

For the square lattice, consider the bulk vertex configurations given by \eqref{eq:bulk-vertices-square}. Note that the first and the third configurations can be obtained by removing edges from the second and fourth configurations. Therefore SOWs on the square lattice can be compactly defined as the following.
\begin{definition}
    The SOW on the square lattice is defined by the following vertex configuration generators.
    \begin{equation} \label{eq:bulk-vertices-square-generators}
\begin{tikzpicture}[baseline={([yshift=-.5ex] current bounding box.center)}, scale=0.65]
    \draw [gray!50, ultra thick] (0, -1) to (0, 1);
    \draw [gray!50, ultra thick] (-1, 0) to (1, 0);
		\draw [ultra thick, blue!75] (1, 0) to (-1, 0);
        \draw[dashed] (1,1) -- (-1,1) -- (-1,-1) -- (1,-1) -- (1, 1);
\end{tikzpicture} \quad \quad
\begin{tikzpicture}[baseline={([yshift=-.5ex] current bounding box.center)}, scale=0.65]
    \draw [gray!50, ultra thick] (0, -1) to (0, 1);
    \draw [gray!50, ultra thick] (-1, 0) to (1, 0);
		\draw [bend left=45, looseness=1.75, ultra thick, blue!75] (0, 1) to (-1, 0);
		\draw [bend left=315, looseness=1.75, ultra thick, blue!75] (1, 0) to (0, -1);
        \draw[dashed] (1,1) -- (-1,1) -- (-1,-1) -- (1,-1) -- (1, 1);
\end{tikzpicture}
\end{equation}
\end{definition}

SOWs can also be considered on the triangular lattice, where they are defined in terms of a set of allowed vertex configurations. All vertex configurations can be generated from vertex configuration generators, which are fully occupied vertex configurations up to rotations and reflections. These correspond to all possible ways in which paths may traverse the central vertex such that no additional path can be introduced at that vertex. The set of all allowed vertices is then given by the collection of all subsets of these fully occupied vertex configurations.

\begin{definition}
    The SOW on the triangular lattice are defined by the following vertex configuration generators.
    \begin{align}
    \begin{tikzpicture}[baseline={([yshift=-.5ex] current bounding box.center)}, scale=0.75]

    \coordinate (O) at (0, 0);
    \coordinate (A) at (1, 0);
    \coordinate (B) at (0.5, 0.866);   
    \coordinate (C) at (-0.5, 0.866);
    \coordinate (D) at (-1, 0);
    \coordinate (E) at (-0.5, -0.866);
    \coordinate (F) at (0.5, -0.866);
 
    \draw[dashed] (A) -- (B);
    \draw[dashed] (B) -- (C);
    \draw[dashed] (C) -- (D);
    \draw[dashed] (D) -- (E);
    \draw[dashed] (E) -- (F);
    \draw[dashed] (F) -- (A);

    \draw[gray!50, ultra thick] (A) -- (D);
    \draw[gray!50, ultra thick] (B) -- (E);
    \draw[gray!50, ultra thick] (C) -- (F);
    
    \draw[bend right=60, looseness=3, ultra thick, blue!75] (B) to (A);
    \draw[bend left=60, looseness=3, ultra thick, blue!75] (C) to (D);
    \draw[bend left=60, looseness=3, ultra thick, blue!75] (E) to (F);

    \begin{scope}[shift={(3,0)}]
    \coordinate (O) at (0, 0);
    \coordinate (A) at (1, 0);
    \coordinate (B) at (0.5, 0.866);   
    \coordinate (C) at (-0.5, 0.866);
    \coordinate (D) at (-1, 0);
    \coordinate (E) at (-0.5, -0.866);
    \coordinate (F) at (0.5, -0.866);

    \draw[dashed] (A) -- (B);
    \draw[dashed] (B) -- (C);
    \draw[dashed] (C) -- (D);
    \draw[dashed] (D) -- (E);
    \draw[dashed] (E) -- (F);
    \draw[dashed] (F) -- (A);

    \draw[gray!50, ultra thick] (A) -- (D);
    \draw[gray!50, ultra thick] (B) -- (E);
    \draw[gray!50, ultra thick] (C) -- (F);
    
    \draw[bend right=30, looseness=1.7, ultra thick, blue!75] (B) to (F);
    \draw[bend left=30, looseness=1.7, ultra thick, blue!75] (C) to (E);
    \end{scope}

    \begin{scope}[shift={(6,0)}]
    \coordinate (O) at (0, 0);
    \coordinate (A) at (1, 0);
    \coordinate (B) at (0.5, 0.866);   
    \coordinate (C) at (-0.5, 0.866);
    \coordinate (D) at (-1, 0);
    \coordinate (E) at (-0.5, -0.866);
    \coordinate (F) at (0.5, -0.866);

    \draw[dashed] (A) -- (B);
    \draw[dashed] (B) -- (C);
    \draw[dashed] (C) -- (D);
    \draw[dashed] (D) -- (E);
    \draw[dashed] (E) -- (F);
    \draw[dashed] (F) -- (A);

    \draw[gray!50, ultra thick] (A) -- (D);
    \draw[gray!50, ultra thick] (B) -- (E);
    \draw[gray!50, ultra thick] (C) -- (F);
    
    \draw[bend right=30, looseness=1.7, ultra thick, blue!75] (B) to (F);
    \draw[bend left=60, looseness=3, ultra thick, blue!75] (C) to (D);
    \end{scope}

    \begin{scope}[shift={(9
    ,0)}]
    \coordinate (O) at (0, 0);
    \coordinate (A) at (1, 0);
    \coordinate (B) at (0.5, 0.866);   
    \coordinate (C) at (-0.5, 0.866);
    \coordinate (D) at (-1, 0);
    \coordinate (E) at (-0.5, -0.866);
    \coordinate (F) at (0.5, -0.866);
 
    \draw[dashed] (A) -- (B);
    \draw[dashed] (B) -- (C);
    \draw[dashed] (C) -- (D);
    \draw[dashed] (D) -- (E);
    \draw[dashed] (E) -- (F);
    \draw[dashed] (F) -- (A);

    \draw[gray!50, ultra thick] (A) -- (D);
    \draw[gray!50, ultra thick] (B) -- (E);
    \draw[gray!50, ultra thick] (C) -- (F);
    
    \draw[ultra thick, blue!75] (A) to (D);
    \draw[bend right=60, looseness=3, ultra thick, blue!75] (C) to (B);
    \draw[bend right=60, looseness=3, ultra thick, blue!75] (F) to (E);
    \end{scope}
\end{tikzpicture}
    \end{align}
\end{definition}

As mentioned in the introductions, ODWs on the triangular lattice are a subset of SOWs. Considering the various boolean variable configurations, one obtains the following osculating configurations, which will be sufficient to generate the other configurations once the boolean variables are removed and paths truncated.
\begin{align}
    \begin{tikzpicture}[baseline={([yshift=-.5ex] current bounding box.center)}, scale=0.75]

    \coordinate (O) at (0, 0);
    \coordinate (A) at (1, 0);
    \coordinate (B) at (0.5, 0.866);   
    \coordinate (C) at (-0.5, 0.866);
    \coordinate (D) at (-1, 0);
    \coordinate (E) at (-0.5, -0.866);
    \coordinate (F) at (0.5, -0.866);

    \draw[dashed] (A) -- (B);
    \draw[dashed] (B) -- (C);
    \draw[dashed] (C) -- (D);
    \draw[dashed] (D) -- (E);
    \draw[dashed] (E) -- (F);
    \draw[dashed] (F) -- (A);

    \draw[gray!50, ultra thick] (A) -- (D);
    \draw[gray!50, ultra thick] (B) -- (E);
    \draw[gray!50, ultra thick] (C) -- (F);

    \draw[bend right=60, looseness=3, ultra thick, blue!75] (B) to (A);
    \draw[bend left=60, looseness=3, ultra thick, blue!75] (C) to (D);
    \draw[bend left=60, looseness=3, ultra thick, blue!75] (E) to (F);

    \coordinate (G1) at (0, {2/3 * sqrt(3)/2});
    \draw[thick] (G1) circle (3pt);
    \coordinate (G2) at (0, {-2/3 * sqrt(3)/2});
    \filldraw[black] (G2) circle (3pt);
    \coordinate (G3) at (-1/2, {-1/3 * sqrt(3)/2});
    \draw[thick] (G3) circle (3pt);
    \coordinate (G4) at (1/2, {-1/3 * sqrt(3)/2});
    \draw[thick] (G4) circle (3pt);
    \coordinate (G5) at (1/2, {1/3 * sqrt(3)/2});
    \filldraw[black] (G5) circle (3pt);
    \coordinate (G6) at (-1/2, {1/3 * sqrt(3)/2});
    \filldraw[black] (G6) circle (3pt);
    \begin{scope}[shift={(3,0)}]
    \coordinate (O) at (0, 0);
    \coordinate (A) at (1, 0);
    \coordinate (B) at (0.5, 0.866);   
    \coordinate (C) at (-0.5, 0.866);
    \coordinate (D) at (-1, 0);
    \coordinate (E) at (-0.5, -0.866);
    \coordinate (F) at (0.5, -0.866);

    \draw[dashed] (A) -- (B);
    \draw[dashed] (B) -- (C);
    \draw[dashed] (C) -- (D);
    \draw[dashed] (D) -- (E);
    \draw[dashed] (E) -- (F);
    \draw[dashed] (F) -- (A);

    \draw[gray!50, ultra thick] (A) -- (D);
    \draw[gray!50, ultra thick] (B) -- (E);
    \draw[gray!50, ultra thick] (C) -- (F);

    \draw[bend right=30, looseness=1.7, ultra thick, blue!75] (B) to (F);
    \draw[bend left=30, looseness=1.7, ultra thick, blue!75] (C) to (E);

    \coordinate (G1) at (0, {2/3 * sqrt(3)/2});
    \draw[thick] (G1) circle (3pt);
    \coordinate (G2) at (0, {-2/3 * sqrt(3)/2});
    \draw[thick] (G2) circle (3pt);
    \coordinate (G3) at (-1/2, {-1/3 * sqrt(3)/2});
    \filldraw[black] (G3) circle (3pt);
    \coordinate (G4) at (1/2, {-1/3 * sqrt(3)/2});
    \filldraw[black] (G4) circle (3pt);
    \coordinate (G5) at (1/2, {1/3 * sqrt(3)/2});
    \filldraw[black] (G5) circle (3pt);
    \coordinate (G6) at (-1/2, {1/3 * sqrt(3)/2});
    \filldraw[black] (G6) circle (3pt);
    
    \end{scope}

    \begin{scope}[shift={(6,0)}]
    \coordinate (O) at (0, 0);
    \coordinate (A) at (1, 0);
    \coordinate (B) at (0.5, 0.866);   
    \coordinate (C) at (-0.5, 0.866);
    \coordinate (D) at (-1, 0);
    \coordinate (E) at (-0.5, -0.866);
    \coordinate (F) at (0.5, -0.866);

    \draw[dashed] (A) -- (B);
    \draw[dashed] (B) -- (C);
    \draw[dashed] (C) -- (D);
    \draw[dashed] (D) -- (E);
    \draw[dashed] (E) -- (F);
    \draw[dashed] (F) -- (A);

    \draw[gray!50, ultra thick] (A) -- (D);
    \draw[gray!50, ultra thick] (B) -- (E);
    \draw[gray!50, ultra thick] (C) -- (F);

    \draw[bend right=30, looseness=1.7, ultra thick, blue!75] (B) to (F);
    \draw[bend left=60, looseness=3, ultra thick, blue!75] (C) to (D);

    \coordinate (G1) at (0, {2/3 * sqrt(3)/2});
    \draw[thick] (G1) circle (3pt);
    \coordinate (G2) at (0, {-2/3 * sqrt(3)/2});
    \draw[thick] (G2) circle (3pt);
    \coordinate (G3) at (-1/2, {-1/3 * sqrt(3)/2});
    \draw[thick] (G3) circle (3pt);
    \coordinate (G4) at (1/2, {-1/3 * sqrt(3)/2});
    \filldraw[black] (G4) circle (3pt);
    \coordinate (G5) at (1/2, {1/3 * sqrt(3)/2});
    \filldraw[black] (G5) circle (3pt);
    \coordinate (G6) at (-1/2, {1/3 * sqrt(3)/2});
    \filldraw[black] (G6) circle (3pt);
    \end{scope}

    \begin{scope}[shift={(9
    ,0)}]
    \coordinate (O) at (0, 0);
    \coordinate (A) at (1, 0);
    \coordinate (B) at (0.5, 0.866);   
    \coordinate (C) at (-0.5, 0.866);
    \coordinate (D) at (-1, 0);
    \coordinate (E) at (-0.5, -0.866);
    \coordinate (F) at (0.5, -0.866);

    \draw[dashed] (A) -- (B);
    \draw[dashed] (B) -- (C);
    \draw[dashed] (C) -- (D);
    \draw[dashed] (D) -- (E);
    \draw[dashed] (E) -- (F);
    \draw[dashed] (F) -- (A);

    \draw[gray!50, ultra thick] (A) -- (D);
    \draw[gray!50, ultra thick] (B) -- (E);
    \draw[gray!50, ultra thick] (C) -- (F);

    \draw[bend right=60, looseness=3, ultra thick, blue!75] (C) to (B);
    \draw[bend right=60, looseness=3, ultra thick, blue!75] (F) to (E);

    \coordinate (G1) at (0, {2/3 * sqrt(3)/2});
    \filldraw[black] (G1) circle (3pt);
    \coordinate (G2) at (0, {-2/3 * sqrt(3)/2});
    \filldraw[black] (G2) circle (3pt);
    \coordinate (G3) at (-1/2, {-1/3 * sqrt(3)/2});
    \draw[thick] (G3) circle (3pt);
    \coordinate (G4) at (1/2, {-1/3 * sqrt(3)/2});
    \draw[thick] (G4) circle (3pt);
    \coordinate (G5) at (1/2, {1/3 * sqrt(3)/2});
    \draw[thick] (G5) circle (3pt);
    \coordinate (G6) at (-1/2, {1/3 * sqrt(3)/2});
    \draw[thick] (G6) circle (3pt);
    \end{scope}

    \begin{scope}[shift={(12
    ,0)}]
    \coordinate (O) at (0, 0);
    \coordinate (A) at (1, 0);
    \coordinate (B) at (0.5, 0.866);   
    \coordinate (C) at (-0.5, 0.866);
    \coordinate (D) at (-1, 0);
    \coordinate (E) at (-0.5, -0.866);
    \coordinate (F) at (0.5, -0.866);

    \draw[dashed] (A) -- (B);
    \draw[dashed] (B) -- (C);
    \draw[dashed] (C) -- (D);
    \draw[dashed] (D) -- (E);
    \draw[dashed] (E) -- (F);
    \draw[dashed] (F) -- (A);

    \draw[gray!50, ultra thick] (A) -- (D);
    \draw[gray!50, ultra thick] (B) -- (E);
    \draw[gray!50, ultra thick] (C) -- (F);

    \draw[bend right=60, looseness=3, ultra thick, blue!75] (C) to (B);
    \draw[ultra thick, blue!75] (A) to (D);
    
    \coordinate (G1) at (0, {2/3 * sqrt(3)/2});
    \filldraw[black] (G1) circle (3pt);
    \coordinate (G2) at (0, {-2/3 * sqrt(3)/2});
    \filldraw[black] (G2) circle (3pt);
    \coordinate (G3) at (-1/2, {-1/3 * sqrt(3)/2});
    \filldraw[black] (G3) circle (3pt);
    \coordinate (G4) at (1/2, {-1/3 * sqrt(3)/2});
    \filldraw[black] (G4) circle (3pt);
    
    \coordinate (G5) at (1/2, {1/3 * sqrt(3)/2});
    \draw[thick] (G5) circle (3pt);
    \coordinate (G6) at (-1/2, {1/3 * sqrt(3)/2});
    \draw[thick] (G6) circle (3pt);
    \end{scope}
\end{tikzpicture}
\end{align}

Removing the boolean variables, ODWs on the triangular lattice can be defined as the following.
\begin{definition}
    ODWs on the triangular lattice are defined by the following vertex configuration generators.
    \begin{align}
    \begin{tikzpicture}[baseline={([yshift=-.5ex] current bounding box.center)}, scale=0.75]

    \coordinate (O) at (0, 0);
    \coordinate (A) at (1, 0);
    \coordinate (B) at (0.5, 0.866);   
    \coordinate (C) at (-0.5, 0.866);
    \coordinate (D) at (-1, 0);
    \coordinate (E) at (-0.5, -0.866);
    \coordinate (F) at (0.5, -0.866);

    \draw[dashed] (A) -- (B);
    \draw[dashed] (B) -- (C);
    \draw[dashed] (C) -- (D);
    \draw[dashed] (D) -- (E);
    \draw[dashed] (E) -- (F);
    \draw[dashed] (F) -- (A);

    \draw[gray!50, ultra thick] (A) -- (D);
    \draw[gray!50, ultra thick] (B) -- (E);
    \draw[gray!50, ultra thick] (C) -- (F);
    
    \draw[bend right=60, looseness=3, ultra thick, blue!75] (B) to (A);
    \draw[bend left=60, looseness=3, ultra thick, blue!75] (C) to (D);
    \draw[bend left=60, looseness=3, ultra thick, blue!75] (E) to (F);

    \coordinate (G1) at (0, {2/3 * sqrt(3)/2});
    \coordinate (G2) at (0, {-2/3 * sqrt(3)/2});
    \coordinate (G3) at (-1/2, {-1/3 * sqrt(3)/2});
    \coordinate (G4) at (1/2, {-1/3 * sqrt(3)/2});
    \coordinate (G5) at (1/2, {1/3 * sqrt(3)/2});
    \coordinate (G6) at (-1/2, {1/3 * sqrt(3)/2});
    \begin{scope}[shift={(3,0)}]
    \coordinate (O) at (0, 0);
    \coordinate (A) at (1, 0);
    \coordinate (B) at (0.5, 0.866);   
    \coordinate (C) at (-0.5, 0.866);
    \coordinate (D) at (-1, 0);
    \coordinate (E) at (-0.5, -0.866);
    \coordinate (F) at (0.5, -0.866);

    \draw[dashed] (A) -- (B);
    \draw[dashed] (B) -- (C);
    \draw[dashed] (C) -- (D);
    \draw[dashed] (D) -- (E);
    \draw[dashed] (E) -- (F);
    \draw[dashed] (F) -- (A);

    \draw[gray!50, ultra thick] (A) -- (D);
    \draw[gray!50, ultra thick] (B) -- (E);
    \draw[gray!50, ultra thick] (C) -- (F);

    \draw[bend right=30, looseness=1.7, ultra thick, blue!75] (B) to (F);
    \draw[bend left=30, looseness=1.7, ultra thick, blue!75] (C) to (E);

    \coordinate (G1) at (0, {2/3 * sqrt(3)/2});
    \coordinate (G2) at (0, {-2/3 * sqrt(3)/2});
    \coordinate (G3) at (-1/2, {-1/3 * sqrt(3)/2});
    \coordinate (G4) at (1/2, {-1/3 * sqrt(3)/2});
    \coordinate (G5) at (1/2, {1/3 * sqrt(3)/2});
    \coordinate (G6) at (-1/2, {1/3 * sqrt(3)/2});
    
    \end{scope}

    \begin{scope}[shift={(6,0)}]
    \coordinate (O) at (0, 0);
    \coordinate (A) at (1, 0);
    \coordinate (B) at (0.5, 0.866);   
    \coordinate (C) at (-0.5, 0.866);
    \coordinate (D) at (-1, 0);
    \coordinate (E) at (-0.5, -0.866);
    \coordinate (F) at (0.5, -0.866);

    \draw[dashed] (A) -- (B);
    \draw[dashed] (B) -- (C);
    \draw[dashed] (C) -- (D);
    \draw[dashed] (D) -- (E);
    \draw[dashed] (E) -- (F);
    \draw[dashed] (F) -- (A);

    \draw[gray!50, ultra thick] (A) -- (D);
    \draw[gray!50, ultra thick] (B) -- (E);
    \draw[gray!50, ultra thick] (C) -- (F);

    \draw[bend right=30, looseness=1.7, ultra thick, blue!75] (B) to (F);
    \draw[bend left=60, looseness=3, ultra thick, blue!75] (C) to (D);

    \coordinate (G1) at (0, {2/3 * sqrt(3)/2});
    \coordinate (G2) at (0, {-2/3 * sqrt(3)/2});
    \coordinate (G3) at (-1/2, {-1/3 * sqrt(3)/2});
    \coordinate (G4) at (1/2, {-1/3 * sqrt(3)/2});
    \coordinate (G5) at (1/2, {1/3 * sqrt(3)/2});
    \coordinate (G6) at (-1/2, {1/3 * sqrt(3)/2});
    \end{scope}

    \begin{scope}[shift={(9
    ,0)}]
    \coordinate (O) at (0, 0);
    \coordinate (A) at (1, 0);
    \coordinate (B) at (0.5, 0.866);   
    \coordinate (C) at (-0.5, 0.866);
    \coordinate (D) at (-1, 0);
    \coordinate (E) at (-0.5, -0.866);
    \coordinate (F) at (0.5, -0.866);

    \draw[dashed] (A) -- (B);
    \draw[dashed] (B) -- (C);
    \draw[dashed] (C) -- (D);
    \draw[dashed] (D) -- (E);
    \draw[dashed] (E) -- (F);
    \draw[dashed] (F) -- (A);

    \draw[gray!50, ultra thick] (A) -- (D);
    \draw[gray!50, ultra thick] (B) -- (E);
    \draw[gray!50, ultra thick] (C) -- (F);

    \draw[bend right=60, looseness=3, ultra thick, blue!75] (C) to (B);
    \draw[bend right=60, looseness=3, ultra thick, blue!75] (F) to (E);

    \coordinate (G1) at (0, {2/3 * sqrt(3)/2});
    \coordinate (G2) at (0, {-2/3 * sqrt(3)/2});
    \coordinate (G3) at (-1/2, {-1/3 * sqrt(3)/2});
    \coordinate (G4) at (1/2, {-1/3 * sqrt(3)/2});
    \coordinate (G5) at (1/2, {1/3 * sqrt(3)/2});
    \coordinate (G6) at (-1/2, {1/3 * sqrt(3)/2});
    \end{scope}

    \begin{scope}[shift={(12,0)}]
    \coordinate (O) at (0, 0);
    \coordinate (A) at (1, 0);
    \coordinate (B) at (0.5, 0.866);   
    \coordinate (C) at (-0.5, 0.866);
    \coordinate (D) at (-1, 0);
    \coordinate (E) at (-0.5, -0.866);
    \coordinate (F) at (0.5, -0.866);
 
    \draw[dashed] (A) -- (B);
    \draw[dashed] (B) -- (C);
    \draw[dashed] (C) -- (D);
    \draw[dashed] (D) -- (E);
    \draw[dashed] (E) -- (F);
    \draw[dashed] (F) -- (A);

    \draw[gray!50, ultra thick] (A) -- (D);
    \draw[gray!50, ultra thick] (B) -- (E);
    \draw[gray!50, ultra thick] (C) -- (F);

    \draw[bend right=60, looseness=3, ultra thick, blue!75] (C) to (B);
    \draw[ultra thick, blue!75] (A) to (D);
    
    \coordinate (G1) at (0, {2/3 * sqrt(3)/2});
    \coordinate (G2) at (0, {-2/3 * sqrt(3)/2});
    \coordinate (G3) at (-1/2, {-1/3 * sqrt(3)/2});
    \coordinate (G4) at (1/2, {-1/3 * sqrt(3)/2});
    \coordinate (G5) at (1/2, {1/3 * sqrt(3)/2});
    \coordinate (G6) at (-1/2, {1/3 * sqrt(3)/2});
    \end{scope}
\end{tikzpicture}
\end{align}
\end{definition}

\subsection{Center coordinates in the hypercubic lattice} \label{sec:center-coordinates}
In the next subsection, we will define SOMs and ODWs generally for the hypercubic lattice. Before we do so, it will be useful to introduce notion of center coordinates. This will also be useful for proofs regarding the existence and bounds for their growth constants in \Cref{sec:existence-and-upper-bound}.

We will consider the vertices of a $d$-hypercubic lattice to at integer coordinates. So, any vertex $v$ can be described by its coordinates $\bm{v} = (v_i)_{i=1}^d$, where $v_i \in \mathbb{Z} \; \forall i$. Higher-dimensional objects such as edges, faces, and cubes can also be located via their center coordinates. An edge $l$ links two vertices with one dimension differing by 1. So, it can be described by the midpoint between the two vertices and therefore its center coordinates $\bm{l} = (l_i)_{i=1}^d$ has one half-integer and $(d-1)$ integers. A face $f$ with center coordinates $\bm{f} = (f_i)_{i=1}^d$ has two half-integers and $(d-2)$ integers. Its two half-integers give the orientation of the face.

Similarly, a $k$-cube $b$ (which will sometimes be referred to as a $k$-edge or $k$-face depending on what role it is playing), has center coordinates $\bm{b}$ with $k$ half-integers and $(d-k)$ integers. We will use the notation $\mathrm{Int}(\bm{b})$ as the integer dimensions of $b$ and $\mathrm{HalfInt}(\bm{b})$ as the half-integer dimensions of $b$. The half-integer dimensions will sometimes be referred to as orientations when convenient.

Consider a face $f$ with orientations or half-integer dimensions $\{j, j'\} = \mathrm{HalfInt}(\bm{f})$. Its boundaries are edges centered at $\bm{f} \pm \frac12 \bm{j}$ and $\bm{f} \pm \frac12 \bm{j}'$. Here, $\bm{j}$ is the unit vector in the $j$\textsuperscript{th} dimension. We will refer these edges as neighbouring edges of $f$. $f$ may neighbour other faces through its neighbouring edges. For example, on the edge $\bm{f} + \frac12 \bm{j}$, $f$ may neighbour another face at $\bm{f} + \bm{j}$, or at $\bm{f} + \frac12 \bm{j} \pm \frac12 \bm{i}$ $\forall i \in \mathrm{Int}(\bm{f})$. A face can also be part of the boundary of a cube $b$, centered at $\bm{b} = \bm{f} \pm \frac12 \bm{i}$ $\forall i \in \mathrm{Int}(\bm{f})$. We will refer to such cubes as neighbouring cubes of $f$.

Generalising, a $k$-face $f$ 
\begin{itemize}
    \item neighbours $(k-1)$-edges at $\bm{f} \pm \frac12 \bm{j} \forall j \in \mathrm{HalfInt}(\bm{f})$. This can be thought of as ``selling'' one of its half-integers.
    \item Via these edges, it neighbours other $k$-faces with same orientations/half-integer dimensions at $\bm{f} \pm \bm{j}$ for any $j \in \mathrm{HalfInt}(\bm{f})$, or at orientations that differ by one, at $\bm{f} \pm \frac12 \bm{j} \pm \frac12 \bm{i}$ for any $j \in \mathrm{HalfInt}(\bm{f}), \forall i \in \mathrm{Int}(\bm{f})$. This can be thought of as ``trading'' one of its half-integer dimensions with a new one.
    \item Finally, it neighbours $(k+1)$-cubes at $\bm{f} \pm \frac12 \bm{i}$ for any $i \in \mathrm{HalfInt}(\bm{f})$. This can be thought of as ``buying'' a new half-integer dimension.
\end{itemize}

\subsection{Definitions of $(d,k)$-self-osculating manifolds and osculating domain wall hypersurfaces on the hypercubic lattice} \label{sec:osculating}

For general $(d, k)$, we cannot rely on explicit drawings of vertex or edge configurations. So, in this subsection, we define $(d,k)$-SOMs and ODW on the $d$-dim hypercubic lattice without relying on them.

\begin{definition}
    The set of $(d,k)$-SOMs with hyperarea ($k$-area) $n$, $\mathrm{SOM}_{(d,k), n}$ is a set of SOMs identified up to translation. A SOM of hyperarea $n$ is a collection of hyperfaces ($k$-faces) and connections, such that
    \begin{itemize}
        \item There is one connected component.
        \item When two hyperfaces neighbour the same hyperedge (a $(k-1)$-edge), they are deemed to be connected. 
        \item When more than two hyperfaces neighbour a hyperedge, their connections are specified and must obey the osculating condition. Connections are defined such that
        \begin{itemize}
            \item if there is an even number of hyperfaces neighbouring the hyperedge, they are pairings of such hyperfaces.
            \item If there is an odd number of hyperfaces neighbouring the hyperedge, it is one lone such hyperface and pairings between the rest of such hyperfaces.
        \end{itemize}
    \end{itemize}
\end{definition}

Now we define the criterion to determine whether a configuration of more than two $k$-faces is allowed in SOMs on the hypercubic lattice, which we will call the osculating condition.

\begin{definition}[Osculating condition]
    Consider a hyperedge (a $(k-1)$-edge) centered at $\bm{g}$ and $M > 2$ hyperfaces ($k$-faces) neighbouring it with some connection structure.

    If $M$ is even such that there are $m$ pairs of hyperfaces such that the two hyperfaces in a pair are connected, we compare all possible two pairs in the following.
    
    Consider two pairs of hyperfaces, with center coordinates $\bm{f}_{1}$, $\bm{f}_{2}$, $\bm{f}_{3}$, $\bm{f}_{4}$, neighbouring the same $(k-1)$-edge, such that $f_{1}$ and $f_{2}$ are connected, and $f_{3}$ and $f_{4}$ are connected. Let $e_{i_{a}} = (f_{a}- g)/\lVert f_a - g \rVert$ be the normalised vectors to the centers from the center of the $(k-1)$-edge. Here, $\lVert \cdot \rVert$ is the Euclidean norm. Then, the following applies. First, we require that all $\{e_{i_{a}}\}_{a=1}^4$ are different. Second, if the unit vectors are parallel, $e_{i_1} = - e_{i_2}$ (i.e. $f_1$ and $f_2$ have the same orientation), then we only allow the connections if $e_{i_3} \perp e_{i_4}$.

    The osculating condition is fulfilled the above holds for all two pairs of hyperfaces.
    
    If $M$ is odd, the osculating condition is fulfilled if any hyperface neighbouring the hyperedge can be connected to the lone neighbouring hyperface such that the osculating condition is fulfilled.
\end{definition}
The osculating condition is illustrated for $(d, k)=(3,1)$ in \Cref{fig:osculating-hyperedge}. 

\begin{figure}[H]
    \centering
    \begin{align*}
    \begin{tikzpicture}[baseline={([yshift=-.5ex] current bounding box.center)}, scale=0.5]
    \draw [gray!50, ultra thick, ->] (0, -1, 0) -> (0, 1, 0) node[above, scale = 0.75, black] {$\bm{1}$};
    \draw [gray!50, ultra thick, ->] (-1, 0, 0) -> (1, 0, 0) node[right, scale = 0.75, black] {$\bm{2}$};
    \draw [gray!50, ultra thick, ->] (0, 0, 1) -> (0, 0, -1) node[above right, scale = 0.75, black] {$\bm{3}$};
    \end{tikzpicture}
    \quad \quad
    \begin{tikzpicture}[baseline={([yshift=-.5ex] current bounding box.center)}, scale=0.5]
    \draw [gray!50, ultra thick] (0, -1, 0) to (0, 1, 0);
    \draw [gray!50, ultra thick] (-1, 0, 0) to (1, 0, 0);
    \draw [gray!50, ultra thick] (0, 0, -1) to (0, 0, 1);
    \draw [bend left=45, looseness=1.75, ultra thick, blue!75] (0, 1, 0) to (-1, 0, 0);
    \draw [ultra thick, blue!75] (0, 0, 1) to (0, 0, -1);
    \draw[dashed] (1, 1, 1) -- (-1, 1, 1) -- (-1, -1, 1) -- (1, -1, 1) -- (1, 1, 1);
    \draw[dashed] (1, 1, 1) -- (1, 1, -1);
    \draw[dashed] (-1, 1, 1) -- (-1, 1, -1);
    \draw[dashed] (-1, -1, 1) -- (-1, -1, -1);
    \draw[dashed] (1, -1, 1) -- (1, -1, -1);
    \draw[dashed] (1, 1, -1) -- (-1, 1, -1) -- (-1, -1, -1) -- (1, -1, -1) -- (1, 1, -1);
    \end{tikzpicture}
    \quad \quad
    \begin{tikzpicture}[baseline={([yshift=-.5ex] current bounding box.center)}, scale=0.5]
    \draw [gray!50, ultra thick] (0, -1, 0) to (0, 1, 0);
    \draw [gray!50, ultra thick] (-1, 0, 0) to (1, 0, 0);
    \draw [gray!50, ultra thick] (0, 0, -1) to (0, 0, 1);
    \draw [ultra thick, blue!75] (0, 0, 1) to (0, 0, -1);
    \draw [ultra thick, red!75] (1, 0, 0) to (-1, 0, 0);
    \draw[dashed] (1, 1, 1) -- (-1, 1, 1) -- (-1, -1, 1) -- (1, -1, 1) -- (1, 1, 1);
    \draw[dashed] (1, 1, 1) -- (1, 1, -1);
    \draw[dashed] (-1, 1, 1) -- (-1, 1, -1);
    \draw[dashed] (-1, -1, 1) -- (-1, -1, -1);
    \draw[dashed] (1, -1, 1) -- (1, -1, -1);
    \draw[dashed] (1, 1, -1) -- (-1, 1, -1) -- (-1, -1, -1) -- (1, -1, -1) -- (1, 1, -1);
    \end{tikzpicture} 
\end{align*}
\caption{Illustration of osculating condition for $(d, k)=(3, 1)$. Here, the hyperfaces correspond to edges. In both cases, there exists a connected pair $(f_1, f_2)$ with $e_{i_1} = \bm{3}$, $e_{i_2} = -\bm{3}$, and therefore they are parallel. On the left, the other connected pair $(f_3, f_4)$ do not share any unit vectors with $(f_1, f_2)$ and their corresponding unit vectors are perpendicular, $e_{i_3} = \bm{2} \perp e_{i_4} = \bm{1}$. Therefore, the configuration $(f_1, f_2)$ and $(f_3, f_4)$ are allowed. Geometrically, this constitutes an allowed osculating configuration. On the right, although the other connected pair $(f_3', f_4')$ does not share any of the unit vectors, they are not perpendicular $e_{i_3'} = \bm{2} \perp e_{i_4'} = - \bm{2}$ and so the configurations $(f_1, f_2)$ and $(f_3', f_4')$ are not allowed. Geometrically, this constitutes a crossing.}
\label{fig:osculating-hyperedge}
\end{figure}
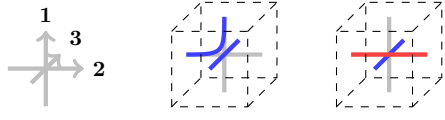

Next, we consider generalising ODWs for general to $(d,k=d-1)$.

First, notice that a $k=(d-1)$-face $f$ only neighbours $2(d-k)=2$ $d$-cubes at $\bm{b}_\pm = \bm{f} \pm \frac12 \bm{i}$, where $i$ is the only element of $\mathrm{Int}(\bm{f})$. This is because it already has $(d-1)$ orientations, i.e. $\abs{\mathrm{HalfInt}(\bm{f})} = d-1$. Therefore, $f$ can be seen as a boundary between $\bm{b}_+$ and $\bm{b}_-$ and the concept of domain wall can be applied. This is different to a general $k$-face with $k<d-1$, which neighbours more than 2 $(k+1)$-cubes.

Next, a $(d-2)$-edge $l$ neighbours four $(d-1)$-faces at $\bm{f}_{\pm, \bullet} = \bm{l} \pm \frac12 \bm{i}_1$ or $\bm{f}_{\bullet, \pm} = \bm{l} \pm \frac12 \bm{i}_2$, where $\{i_1, i_2\} = \mathrm{Int}(\bm{l})$. These $(d-1)$-faces neighbour four $d$-cubes at $\bm{b}_{\pm, \pm'} = \bm{l} \pm \frac12 \bm{i}_1 \pm' \frac12 \bm{i}_2$. Therefore, at each $(d-2)$-edge, the situation reduces to that of SOW in 2D. We are now ready define ODWs.

\begin{definition}
    The set of $d$-dim ODWs with $(d-1)$-area $n$, $\mathrm{ODW}_{(d), n}$, is a set of ODWs identified up to translation. A ODW of hyperarea $n$ is a collection of hyperfaces ($k$-faces) and connections, such that it is a SOW and that every vertex is included in the set of osculating domain wall vertex configurations.
\end{definition}

\begin{definition}[Osculating domain wall vertex configurations]
    Osculating domain wall vertex configurations are generated as following.

    By convention, set the location of the vertex to be at the origin. We will refer to $d$-cubes as hypercubes, $(d-1)$-faces as hyperfaces, and $(d-2)$-edges as hyperedges. First, consider all possible configurations of $2^d$ boolean variables $\sigma(\bm{b})$ at hypercubes $b$ which have the origin as one of their corners. These hypercubes are at $\bm{b}_{\{B_i\}_{i=1}^d}\sum_{i=1}^d \frac{B_i}{2} \bm{i}$, $\forall B_i \in \{-1, +1\} \forall i\in\{1,2,\dots,d\}$. Second, for all hyperfaces that touch the origin~\footnote{They are located at $\bm{f}_{\{F_{i_m}\}_{m=1}^{d-1}} = \sum_{m=1}^{d-1} \frac{F_{i_m}}{2} \bm{i}_m$ for all $F_{i_m} \in \{-1, +1\}$ for all $i_m \in C$ for all $C$ in $(d-1)$-combinations of $\{1,2,\dots,d\}$.}, let a hyperface be present if the boolean variables at neighbouring hypercubes differ. Third, consider all hyperedges touching the origin. For each hyperedge $l$ with $\mathrm{Int}(\bm{l}) = \{i_1, i_2\}$, consider the cases where all four of its neighbouring hyperfaces are present, i.e.
    \begin{align}
        \begin{pmatrix}
            \sigma(\bm{b}_{-+}) & \sigma(\bm{b}_{++}) \\
            \sigma(\bm{b}_{--}) & \sigma(\bm{b}_{+-})
        \end{pmatrix}
        =
        \begin{pmatrix}
            0 & 1 \\ 1 & 0
        \end{pmatrix}
        \quad \text{or} \quad
        \begin{pmatrix}
            1 & 0 \\ 0 & 1
        \end{pmatrix}.
    \end{align}
    In the former case, connect the hyperfaces as $(\bm{f}_{\bullet, +}, \bm{f}_{+, \bullet})$ and $(\bm{f}_{-, \bullet}, \bm{f}_{\bullet, -})$. In the latter case, connect the hyperfaces as $(\bm{f}_{\bullet, +}, \bm{f}_{-, \bullet})$ and $(\bm{f}_{\bullet,-}, \bm{f}_{+,\bullet})$. This is illustrated in \Cref{fig:osculating-domain-wall} for $d=3$.

    The set of all possible hyperfaces and their connections disregarding the boolean variable configurations, define the bulk osculating domain wall vertex configurations.

    The boundary osculating domain wall vertex configurations are defined by removing faces from the bulk osculating domain wall vertex configurations. Together, they define the osculating domain wall vertex configurations.
\end{definition}

\begin{corollary}
    ODWs in $d$-dim hypercubic lattice with $(d-1)$ area $n$ are a subset of $\mathrm{SOM}_{(d,d-1),n}$,
    \begin{align}
        \mathrm{ODW}_{\mathbb{Z}^d, n} \subset \mathrm{SOM}_{(d,d-1), n}.
    \end{align}
\end{corollary}

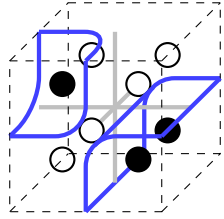
\begin{figure}[H]
\centering
\begin{tikzpicture}[baseline={([yshift=-.5ex] current bounding box.center)}, scale=1]

\draw [rounded corners=4mm, ultra thick, blue!75] (0,-1,-1) -- (0,0,-1) -- (1,0,-1);
\draw [rounded corners=4mm, ultra thick, blue!75] (-1,1,0) -- (-1,0,0) -- (-1,0,1);
\draw [rounded corners=4mm, ultra thick, blue!75] (0,-1,1) -- (0,-1,-1);

\draw [gray!50, ultra thick] (0, -1, 0) to (0, 1, 0);
\draw [gray!50, ultra thick] (-1, 0, 0) to (1, 0, 0);
\draw [gray!50, ultra thick] (0, 0, -1) to (0, 0, 1);

\node[circle, inner sep=0, draw=black, fill=white, minimum size=9, thick] at (-0.5, -0.5, -0.5) {};
\node[circle, inner sep=0, draw=black, fill=black, minimum size=9, thick] at (0.5, -0.5, -0.5) {};
\node[circle, inner sep=0, draw=black, fill=white, minimum size=9, thick] at (-0.5, 0.5, -0.5) {};
\node[circle, inner sep=0, draw=black, fill=white, minimum size=9, thick] at (0.5, 0.5, -0.5) {};

\node[circle, inner sep=0, draw=black, fill=white, minimum size=9, thick] at (-0.5, -0.5, 0.5) {};
\node[circle, inner sep=0, draw=black, fill=black, minimum size=9, thick] at (0.5, -0.5, 0.5) {};
\node[circle, inner sep=0, draw=black, fill=black, minimum size=9, thick] at (-0.5, 0.5, 0.5) {};
\node[circle, inner sep=0, draw=black, fill=white, minimum size=9, thick] at (0.5, 0.5, 0.5) {};

\draw[dashed] (1, 1, 1) -- (-1, 1, 1) -- (-1, -1, 1) -- (1, -1, 1) -- (1, 1, 1);
\draw[dashed] (1, 1, 1) -- (1, 1, -1);
\draw[dashed] (-1, 1, 1) -- (-1, 1, -1);
\draw[dashed] (-1, -1, 1) -- (-1, -1, -1);
\draw[dashed] (1, -1, 1) -- (1, -1, -1);
\draw[dashed] (1, 1, -1) -- (-1, 1, -1) -- (-1, -1, -1) -- (1, -1, -1) -- (1, 1, -1);

\draw [rounded corners=4mm, ultra thick, blue!75] (0,-1,1) -- (0,0,1) -- (1,0,1);
\draw [rounded corners=4mm, ultra thick, blue!75] (-1,0,1) -- (0,0,1) -- (0,1,1);
\draw [rounded corners=4mm, ultra thick, blue!75] (0,1,1) -- (0,1,0) -- (-1,1,0);
\draw [rounded corners=4mm, ultra thick, blue!75] (1,0,1) -- (1,0,-1);
\draw [rounded corners=4mm, ultra thick, blue!75] (0+0.12,0-0.12,1) -- (0+0.12,0-0.12,-1);
\end{tikzpicture}
\caption{Illustration of the generation of an osculating domain-wall vertex configuration for $d=3$. 
Boolean variables are placed at the cubes (i.e.\, $3$-cubes) and are represented by filled or unfilled circles. 
The filled (black) circles are separated from the unfilled (white) circles by a domain wall, a surface whose outline is given in blue. The osculating connections can be specified by looking at each edge, which is neighboured by four faces. The same procedure applies for general $d$, where boolean variables at $d$-cubes are separated by a a domain wall. The osculating connections can be specified by looking at each $(d-1)$-edge.}
\label{fig:osculating-domain-wall}
\end{figure}

\section{Automata method for upper bounds of connective constants of restricted walks} \label{sec:automata}
In this section, we review the automata method of \cite{ponitz2000improved} for SAWs and discuss how it must be modified for SOWs and for generic restricted walks. Using this adapted method, we find the following:
\begin{theorem}
    The connective constant for SOWs on the square lattice verifies the upper bound
    \begin{align}
        \mu^\rm{SOW}_\square \leq 2.73911.
    \end{align}
\end{theorem}

\begin{theorem}
    The connective constant for SOWs on the triangular lattice verifies the upper bound
    \begin{align}
        \mu^\rm{SOW}_\triangle \leq 4.44931.
    \end{align}
\end{theorem}

\begin{theorem}
    The connective constant for ODWs on the triangular lattice verifies the upper bound
    \begin{align}
        \mu^\rm{SOW}_\triangle \leq 4.44867.
    \end{align}
\end{theorem}

\subsection{Outline of the method}
We briefly describe a method to obtain an upper bound on the connective constant for restricted walks, which are walks where certain paths are disallowed by a given rule. This method is essentially a generalization of the approach presented in \cite{ponitz2000improved} for self-avoiding walks (SAWs).

The overall idea of the method is to count, for a given rule of a restricted walk, the number of paths that do not contain `loops' of size less than or equal to $k$, as defined below. Such loops form a subset of the disallowed walks for the given rule. Consequently, this yields an upper bound on the number of paths for the restricted walk. By increasing the maximum loop size $k$ that is taken into account, we systematically obtain tighter bounds.

\begin{definition}[Loop of size $k$]
\label{Loops}
Consider a restricted walk with a rule $R$ and denote $\mathbf{1}_R(\gamma)$ a boolean function that takes as input a path and returns $0$ or $1$ depending on whether the path is allowed.
A loop of size $k$, with respect to a restricted walk with a rule $R$,
is a disallowed path $\gamma = (\gamma_1, \gamma_2, \dots, \gamma_k)$
of length $k$, $\mathbf{1}_R(\gamma) = 0$, such that every subpaths
$\gamma_{[n,k]} := (\gamma_n, \gamma_{n+1}, \dots, \gamma_k) \quad \text{for all } n \in \{2, \dots, k-1\}
$ is an allowed path $\mathbf{1}_R(\gamma_{[n,k]})=1$.
\end{definition}


As discussed, we would like to obtain an upper bound by counting paths that do not contain loops of size at most $k$.
In order to achieve this, we construct a transfer matrix $M_{\leq k}$ that counts these paths. In particular, consider a path $\gamma_0$ which is an initial path of unit length connecting the origin with a neighboring vertex. We can identify this path with the basis vector $\gamma_0 \mapsto v_1 = (1,0,0,0,...)^{\intercal}$. Note that this is \textit{not} a basis vector in the lattice we consider. Then the number of walks of size $n$ that do not contain loops of size $\leq k$ is given by
\begin{equation}
    c_{n,k} = \kappa \vec{1}^{\intercal}\left(M_{\leq k}\right)^n v_1 \hspace{1pt},
\end{equation}
where $\kappa$ is the number of nearest neighbors to the origin and $\vec{1} = (1, 1, 1, \dots)^{\intercal}$ is a vector full of ones. The matrix $M_{\leq k}$ essentially evolves a given path one step further imposing that no loops of size $\leq k$ are present.

\subsection{Algorithm}
The automata method is essentially a procedure to build the matrix $M_{\leq k}$.

Consider evolving a given a path $\gamma$ of length $n$ one step further. Consider $\gamma$ such that it does not contain loops of size $\leq k$. Let $\gamma'$ be the resulting path, of length $n+1$, obtained by evolving the path $\gamma$. The task is to check whether $\gamma'$ contains or not a loop of size $\leq k$.
We can write $\gamma' = (\gamma'_1, ... ,\gamma'_{n+1})$ with $\gamma = (\gamma'_1, ..., \gamma'_{n})$. Since we know that $\gamma$ is not a loop of size $\leq k$, we only need to check whether the subpath $(\gamma'_{n+2-k},...,\gamma'_{n+1})$ contains loops of size $\leq k$. If it does, we must discard it otherwise we keep it in the count.
Therefore, we only need to keep track of the last $k$ steps of any given walk and check if this contains any loop of size $\leq k$. Therefore, we list all possible paths of length at most $k$ that do not contain loops of size $\leq k$ (up to translation). Each of these paths correspond then to a basis vector in $\{v_j\}_{j=1}^{|\mathcal{C}_k|}$, where $\mathcal{C}_k$ is the set of all such paths.
 
The process consists of three steps:
\begin{enumerate}
    \item Find all loops of size up at most $k$ up to rotations and translation.
    \item Determine the configuration space $\mathcal{C}_k$ and the corresponding basis state $\{v_{\alpha}\}_{\alpha=1}^{|\mathcal{C}_k|}$.
    \item Construct the matrix $M_{\leq k}$ and extract the largest eigenvalue.
\end{enumerate}

\subsubsection{Finding the loops}
The process begins by identifying all loops of a given size, according to \Cref{Loops}. 
Consider restricted walks that obey a given rule $R$ and let $\mathbf{1}_R(\gamma)$ be a boolean function that takes as input a path and return $0$ if the path does not conform with the rule or $1$ otherwise.
The aim is to find the minimal number of loops of size \( k \) (up to translation). To optimize computational efficiency, we design an algorithm that identifies only those loops which are distinct up to rotations and reflections; mathematically, this corresponds to finding representatives of the equivalence classes of loops.
These loops of size \( k \) are found as follows:
\begin{itemize}
    \item Choose the origin from which to start the path, and force the first step of the walk to always be in a fixed direction, (say eastward).
    \item Eliminate all paths that are equivalent under reflection. This is achieved by discarding any path that, after proceeding eastward for \( n < k \) steps, turns downward. These paths have a reflected counterpart, obtained by mirroring along the east direction, that turns upward (rather than downward) and is thus considered equivalent.
    \item Among the remaining paths, find all paths \(\{\gamma\}\) of length \( k \) such that every subpath \(\gamma_{[n,k]}\) for all \( n \in [2, k-1] \) is allowed, i.e., \(\mathbf{1}_R(\gamma_{[n,k]}) = 1\), but the full path is not allowed, \(\mathbf{1}_R(\gamma) = 0\) (see \Cref{Loops}). 
    These selected paths correspond to the loops of size \( k \) up to rotations and reflections.
\end{itemize}
After this process, we have determined the set of loops $\Gamma_k$ of length $k$.
\subsubsection{Determining the basis}
Using the previous process, we can determine the set of loops of size up to $k$ which we denote $\Gamma_{\leq k}$. We now need to find the basis in which to construct the matrix $M_{\leq k}$.
Let $\gamma \in \Gamma_{\leq k}$, then we consider the set of all possible subpaths with the last step removed: $\gamma \to S(\gamma) = \{\gamma_{[1,m]}| \forall m \in [2,k-1] \}$. Paths in $S(\gamma)$ are by construction allowed paths: $\mathbf{1}_R(\gamma') = 1$ $\forall \gamma' \in S(\gamma)$.
The configuration space, $\mathcal{C}$, is nothing but the union of the the sets $S(\gamma)$:
\begin{equation}
    \mathcal{C}_{k} = \bigcup_{\gamma \in \Gamma_{\leq k}}S(\gamma) = \{\gamma_{[1,m]}| \; \forall m \in [2,\mathrm{len}(\gamma)-1] \; \forall \gamma \in \Gamma_{\leq k}\}.
\end{equation}
Here, $\mathrm{len}(\gamma)$ is the length of path $\gamma$. We then assign a standard basis vector to each element of $\mathcal{C}_k$, this is the basis in which the matrix $M_{\leq k}$ is expressed. Consequenlty, $M_{\leq k}$ is a matrix of dimension $|\mathcal{C}_k|$.
\subsubsection{Constructing the matrix}
The previous step of the algorithm gives us a suitable basis $\{v_{\alpha}\}_{\alpha=1}^{|\mathcal{C}_k|}$ to represent the matrix $M_{\leq k}$. To determine the matrix elements we evolve each path, which corresponds to a given basis vector $v_{\alpha}$, by one step and obtain a collection of augmented valid paths: $v_{\alpha} \to \{v_{\beta,\alpha}\}_{\beta}$. We then check to which path in $\{v_{\alpha}\}$, $v_{\beta,\alpha}$ is equivalent to. This equivalence is determined recursively as follows:
\begin{enumerate}
    \item Consider the subpath \( v_{\beta,\alpha,1} \), where the extra subscript indicates that the first step of the walk has been removed. If \( v_{\beta,\alpha,1} \) coincides, up to rotations, reflections and translation, to a path \( v_i \in \{ v_{\alpha} \}_{\alpha = 1}^{|\mathcal{C}_{k}|} \), then we say that \( v_{\beta,\alpha} \) is equivalent to \( v_i \). We denote this equivalence by
    \[
    v_{\beta,\alpha} \sim v_i.
    \]
    \item If $v_{\beta,\alpha,1}$ is not equivalent to any path in the set, 
    then we consider $v_{\beta,\alpha,2}$, i.e., the subpath obtained 
    by removing the first two steps of the walk, and repeat Step~1 
    with $v_{\beta,\alpha,2}$ in place of $v_{\beta,\alpha,1}$.
\end{enumerate}

We then assign the matrix elements as follows: if \( v_{\beta,\alpha} \sim v_{i} \), we increment by one the entry labeled \((i, \alpha)\). This procedure is repeated for all \(\beta, \alpha\). 
We then need to extract the largest eigenvalue. To simplify numerical computations, we observe that as \( k \) increases, the matrix dimension grows correspondingly; however, many elements of the matrix are zero. This sparsity arises because, given a graph with degree \( d \) (i.e., each vertex has \( d \) nearest neighbors), the mapping induced by extending the path one step further,
\[
v_{\alpha} \to \{ v_{\beta,\alpha} \}_{\beta},
\]
produces at most \( d - 1 \) new paths, independent of \( k \). Consequently, as \( k \) increases, the matrix becomes progressively sparser, allowing efficient storage and computations using sparse matrix packages.

    

The steps of the automata method are summarised below.
\begin{enumerate}
    \item Find all loops of size $\leq k$ (identified up to translation, rotations and reflections):
\[
\Gamma_{\leq k} = \{\gamma \mid \gamma \; \text{ is a loop of size } \leq k\}.
\]
\item Find all subpaths $\mathcal{C}_k$ and assign a standard basis vector to each element of $\mathcal{C}_k$.

\item Evolve each subpath $\lambda \in \mathcal{C}_{k}$ one step further discarding loops of size $\leq k$. This gives a set of new paths $\{\lambda_{j}\}$. Identify each path $\lambda_j$ with a suitable element in $\mathcal{C}_k: \lambda_j \simeq \gamma$. 
Add a $1$ in the corresponding matrix element.
\end{enumerate}

This method provides a general approach for obtaining an upper bound on a restricted walk 
characterized by a rule $R$ that specifies which vertex configurations are permitted. 
In the next section, we review the original method applied to self-avoiding walks (SAWs) 
and consider the simple case $k=4$, where the matrix $M_{\leq 4}$ can be written explicitly 
and its largest eigenvalue determined. 
We then analyze self-osculating walks (SOWs) on the square and triangular lattices, 
deriving progressively sharper upper bounds using the automata method.

\subsection{Self-avoiding walks on the square lattice}
It is instructive to consider a minimal example to understand how the automata method works. We consider SAWs on the square lattice with $k=4$.
From the previous discussion, it follows that we need to keep track of paths with loops of size $2$ and loops of size $4$ only. These two loops are depicted in Fig~\ref{fig:Loops k=4}, where we fix the first step to be to the right.
\begin{figure}
    \centering
\begin{center}
\begin{tikzpicture}
    
    \draw[rounded corners=0.5mm] (0,0) -- (1, 0) -- (1,0.1) -- (0, 0.1);
    \filldraw[black] (0,0) circle (2pt);
    
    \draw (5.5,-0.5) rectangle (6.5,0.5);

    \filldraw[black] (5.5,-0.5) circle (2pt);
\end{tikzpicture}
\end{center}
    \caption{Loops of size less than or equal to $4$ for SAWs starting from the origin (black dot). On the left a loop of size $2$ and on the right a loop of size $4$. We fix the orientation of the first step to be to the right.}
    \label{fig:Loops k=4}
\end{figure}
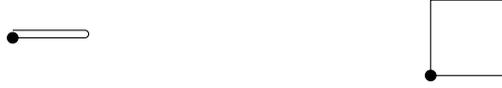

The next step of the algorithm is to determine the configurations space $\mathcal{C}_{4}$. In this case, the paths in $\mathcal{C}_{4}$ are the following (up to rotations and reflections)
\begin{equation}
v_1 = \begin{tikzpicture}[baseline={([yshift=-.5ex]current bounding box.center)}]
    \draw (-1,0) -- (0,0);
    \filldraw[black] (-1,0) circle (2pt);
\end{tikzpicture} 
\hspace{2cm}
v_2 = \begin{tikzpicture}[baseline={([yshift=-.5]current bounding box.center)}]
    \draw (1,0) -- (2,0);
    \draw (2,0) -- (2,1);
    \filldraw[black] (1,0) circle (2pt);
\end{tikzpicture} 
\hspace{2cm}
v_3 = \begin{tikzpicture}[baseline={([yshift=-.5ex]current bounding box.center)}]
    \draw (3,0) -- (4,0);
    \draw (4,0) -- (4,1);
    \draw (4,1) -- (3,1);
    \filldraw[black] (3,0) circle (2pt);
\end{tikzpicture}
\end{equation}
The next step is to evolve each paths $\{v_{\alpha}\}_{\alpha = 1}^{3}$ one step further and identify them with suitable element in $\mathcal{C}_{4}$.

It can be verified that $M_{\leq 4}v_{1} = 2v_{2} + v_{1}$,  $M_{\leq 4}v_{2} = v_{3} + v_2+v_{1}$ and $M_{\leq 4}v_{3} = v_2 +v_{1}$. Hence, the matrix $M_{\leq 4}$ in this basis is given by
\begin{equation}
 M_{\leq 4} = \begin{bmatrix}
     1 & 1 & 1 \\
     2 & 1 & 1 \\
     0 & 1 & 0     
 \end{bmatrix}.
\end{equation}
In terms of this matrix, we can express the number of walks of length $n$ such that there are no loops of size $2$ or $4$ is
\begin{equation}
    c_{n,4} = 4\mathbf{1}^{\intercal}(M_{\leq 4})^n v_1.
\end{equation}
The large $n$ asymptotic of this expression gives us an upper bound on $\mu$ the connective constant of the square lattice defined as
\begin{equation}
    \mu = \lim_{n \to \infty}c_n^{1/n} \hspace{1pt}.
\end{equation}

The upper bound on $\mu$ is given by $\lambda_1 \approx 2.83118$, the eigenvalue with the largest absolute value of the $3 \times 3$ matrix. This procedure can be iterated to obtain the following upper bound on the connective  constant for SAW, $\mu^{\mathrm{SAW}}_\square\leq 2.6792$ \cite{ponitz2000improved}.

This method is quite general. In the following sections, we adapt it to SOWs, and in \Cref{apdx:modified-walk-square}, we discuss a modified self-avoiding walk on the square lattice.

\subsection{Self-osculating walks on the square lattice}
We can generalise the previous methods to SOWs on the square lattice. As before, we consider an example with a low value of $k$ first: $k=5$. The first thing to do is to determine the loops. To do so recall that the allowed vertices for SOWs are given by Eq.~\eqref{eq:bulk-vertices-square} and Eq.~\eqref{eq:boundary-vertices-square}.
The only loops of size less that or equal to $5$ are depicted in Fig.~\ref{fig:Loops k=5}.
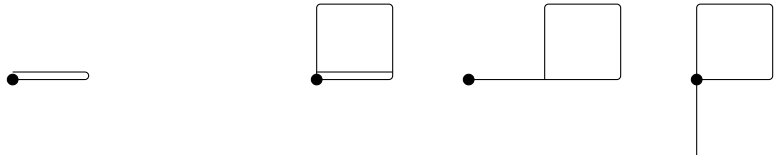
\begin{figure}[H]
    \centering
\begin{center}
\begin{tikzpicture}
    
    \draw[rounded corners=0.5mm] (0,0) -- (1, 0) -- (1,0.1) -- (0, 0.1);
    \filldraw[black] (0,0) circle (2pt);

    \begin{scope}[shift={(4,0)}]
    \draw[rounded corners=0.5mm] 
        (0,0) -- 
        (1,0) -- 
        (1,1) -- 
        (0,1) -- 
        (0,0);
        
    \draw[rounded corners=0.5mm, black] 
        (0,0.1) -- (1,0.1);
    
    \filldraw[black] (0,0) circle (2pt);  
    \end{scope}

    \begin{scope}[shift={(6,0)}]
    \draw[rounded corners=0.5mm] 
        (0,0) -- 
        (1,0) -- 
        (2,0) -- 
        (2,1) -- 
        (1,1) --
        (1,0) ;
    \filldraw[black] (0,0) circle (2pt);  
    \end{scope}

    \begin{scope}[shift={(9,0)}]
    \draw[rounded corners=0.5mm] 
        (0,0) -- 
        (1,0) -- 
        (1,1) -- 
        (0,1) -- 
        (0,0) --
        (0,-1) ;
    \filldraw[black] (0,0) circle (2pt);  
    \end{scope}
    
\end{tikzpicture}
\end{center}
    \caption{Loops of size less than or equal to $5$ for SOWs starting from the origin (black dot). On the left a loop of size $2$ and on the right the three loops of size $5$. We fix the orientation of the first step to be to the right.}
    \label{fig:Loops k=5}
\end{figure}
As before, we can then construct the basis ${v_{\alpha}}$ which is obtained by considering the subpaths of these loops where the last step is removed. In this case, we find a $7 \times 7$ matrix with the largest eigenvalue given by $\lambda_{1}(5) = 2.86055$.
Iterating this procedure for different values of \( k \) yields progressively refined upper bounds, which are presented in Table~\ref{table:upper-bounds-sow}.

\begin{table}
    \centering
\begin{tabular}{r r}
    Maximum accounted loop size & Upper bound for $\mu_\square^\rm{SOW}$ \\
    \hline
    5 & 2.86055 \\
    7 & 2.82042 \\
    9 & 2.79208 \\
    11 & 2.77524 \\
    13 & 2.76333 \\
    15 & 2.75475 \\
    17 & 2.74824 \\
    19 & 2.74316 \\
    21 & 2.73911
    \end{tabular}
    \hspace{3em}
    \caption{Rigorous upper bounds for SOWs on the square lattice, obtained using the automata method.}
    \label{table:upper-bounds-sow}
\end{table}

\subsection{Self-osculating walks on the triangular lattice}
Consider the SOW on the triangular lattice, defined in \Cref{sec:osculating-vertex-definitions}. Unlike SOWs on the square lattice, where only loops of odd length starting from $5$ are present (excluding the trivial back-and-forth loop of length $2$), the triangular lattice admits loops of all lengths starting from $4$. 
As before, we consider an example with a small value of $k$. The trivial case $k=2$ yields $\mu^{\mathrm{SOW}}_{\triangle} \leq 5$, since out of the $6$ possible nearest neighbors of a vertex we can only select at most $5$ of them. A refined upper bound arises from considering $k=4$, in this case the loops are given in Fig.~\ref{fig:Loops k=4 triangle}.
\begin{figure}[H]
    \centering
\begin{center}
\begin{tikzpicture}
    
    \draw[rounded corners=0.5mm] (0,0) -- (1, 0) -- (1,0.1) -- (0, 0.1);
    \filldraw[black] (0,0) circle (2pt);

    \begin{scope}[shift={(4,0)}]
    \draw[rounded corners=0.5mm] 
        (0,0) -- 
        (1,0) -- 
        ({1/2}, {sqrt(3)/2})--
        (0,0);
        
    \draw[rounded corners=0.5mm, black] 
        (0,0.07) -- (1,0.07);
    
    \filldraw[black] (0,0) circle (2pt);  
    \end{scope}

    \begin{scope}[shift={(6,0)}]
    \draw[rounded corners=0.5mm] 
        (0,0) --
        (1,0) -- 
        ({1+1/2}, {sqrt(3)/2})--
        ({1-1/2}, {sqrt(3)/2})--
        ({1+0.2*cos(2*180/3)}, {0.2*sin(2*180/3)});
    \filldraw[black] (0,0) circle (2pt);  
    \draw[thick,black] ({1+0.14*cos(2*180/3)}, {0.14*sin(2*180/3)}) circle (2pt);
    \end{scope}

    \begin{scope}[shift={(8,0)}]
    \draw[rounded corners=0.5mm] 
        (0,0) -- 
        (1,0) -- 
        ({1/2}, {sqrt(3)/2})--
        (0,0)--
        ({1/2}, {-sqrt(3)/2});

    \filldraw[black] (0,0) circle (2pt);  
    \end{scope}
    
\end{tikzpicture}
\end{center}
    \caption{Loops of length less than or equal to 4 for self-osculating walks (SOWs) on the triangular lattice, starting from the origin (black dot). The first step is always taken to the right; all other loops of the same length can be generated from these by applying rotations and reflections. In the third loop from the left, an empty circle is used to indicate the endpoint, making the direction of the walk unambiguous.}
    \label{fig:Loops k=4 triangle}
\end{figure}
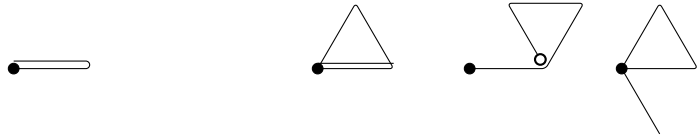
We can then find the basis $\{v_{\alpha}\}$ and construct the matrix $M_{\leq k}$, whose largest eigenvalue $\lambda_1(k)$ yields an upper bound for the connective constant of SOWs on the triangular lattice as presented in Table~\ref{table:upper-bounds-sow-triangle}.
\begin{table}[H]
    \centering
\begin{tabular}{r r}
    Maximum accounted loop size & Upper bound for $\mu_\triangle^\rm{SOW}$ \\
    \hline
    4 & 4.81152 \\
    5 & 4.70066 \\
    6 & 4.63539 \\
    7 & 4.55209  \\
    8 & 4.55209\\
    9 & 4.52473\\
    10 & 4.50327\\
    11 & 4.48587 \\
    12 & 4.47151\\
    13 & 4.45950 \\
    14 & 4.44931

    \end{tabular}
    \caption{Rigorous upper bounds for SOWs on the triangular lattice, obtained using the automata method.}
    \label{table:upper-bounds-sow-triangle}
\end{table}

We also applied the automata method to $\mathrm{ODW}_\triangle$, shown in \Cref{tab:upper-bounds-modified-sow}.

\begin{table}[H]
    \centering
    \begin{tabular}{r r}
        Loop size \(n\) & Upper bounds \(\mu^{\rm{ODW}}_{\triangle}\) \\
        \hline
        4  & 4.81152 \\
        5  & 4.70066 \\
        6  & 4.63518 \\
        7  & 4.58768 \\
        8  & 4.55164 \\
        9  & 4.52423 \\
        10 & 4.50273 \\
        11 & 4.48529 \\
        12 & 4.47092 \\
        13 & 4.45889 \\
        14 & 4.44867
    \end{tabular}
     \caption{Upper bounds for \(\mu_{\triangle}^{\rm{ODW}}\) obtained using the automata method for the osculating domain wall walks on the triangular lattice.}
    \label{tab:upper-bounds-modified-sow}
\end{table}

\section{Existence and bounds of growth constants for restricted $k$-manifolds on the $d$-dim hypercubic lattice} \label{sec:existence-and-upper-bound}
In this section, we prove the existence of growth constants of various restricted manifolds:
\begin{theorem}[Existence of growth constants] \label{thm:existence}
    The growth constants for closed $(d,k)$-SAMs, $(d,k)$-SAMs, $(d,k)$-SOMs, and $(d,k)$-XDs, given by $\mu^\rm{SAM}_{(d,k)}(h=0)$, $\mu^\rm{SAM}_{(d,k)}$, $\mu^\rm{SOM}_{(d,k)}$, and $\mu^\rm{XD}_{(d,k)}$, respectively, exist.
\end{theorem}
We will also prove general and explicit upper bounds:
\begin{theorem} \label{thm:sam-0-upper-bound}
    The growth constants for closed $(d, k)$-SAMs are upper bounded by
    \begin{align}
        \mu^{\mathrm{SAM}}_{(d,k)}(h=0) \leq (2(d-k)+1).
    \end{align}
\end{theorem}
\begin{theorem} \label{thm:sam-som-upper-bound}
    The growth constants for $(d,k)$-SAMs and $(d,k)$-SOMs are upper bounded by
    \begin{align}
        \mu_{(d,k)}^\rm{SAM} \leq \mu_{(d,k)}^\rm{SOM} \leq \frac{(2 k-1)^{2 k-1}}{(2k-2)^{2 k-2} } (2(d-k)+1).
    \end{align}
\end{theorem}
\begin{theorem} \label{thm:xd-upper-bound}
    The growth constant for $(d,k)$-XDs is upper bounded by
    \begin{equation}
        \mu^\rm{XD}_{(d, k)} \leq \frac{w(d,k)^{w(d,k)}}{(w(d,k)-1)^{w(d,k)-1}},
    \end{equation}
    where $w(d,k) = (2k-1)(2(d-k)+1)$.
\end{theorem}

To prove the existence, we will first prove an exponential upper bound for $\mu_n$ by presenting a general strategy to index manifold configurations. This will later be also used to upper bound the growth constant. Then, we will develop a concatenation method, similar to \cite{van1989self}, which can be used in conjunction with Theorem 1 of \cite{wilker1979extension} to prove that the limit $\mu = \lim_{n \rightarrow \infty} \mu_n$ exists. Throughout, we will be using the concept of center coordinates discussed in \Cref{sec:center-coordinates}.

Lower bounds will also be found by considering `directed walk' configurations (\Cref{sec:lower-bound}), leading to the following lower bounds.
\begin{theorem}
    The connective constant of SAMs is lower bounded as
    \begin{align}
        \mu^{\mathrm{SAM}}_{(d,k)} \geq k(d-k+1).
    \end{align}
\end{theorem}
\begin{theorem}
    The connective constant of SAMs with no boundaries ($h=0$) is lower bounded as
    \begin{align}
        \mu^{\mathrm{SAM}}_{(d,k)}(h=0) \geq \left((k+1)(d-k)\right)^{1/2k}.
    \end{align}
\end{theorem}

Therefore, we have that
\begin{theorem}
For $d > k > 1$, The growth constant for $\mathrm{SAM}_{(d, k)}(h=0)$ is strictly smaller than $\mathrm{SAM}_{(d, k)}$. 

(Specifically, $\mu^{\mathrm{SAM}}_{(d,k)}(h=0) \leq (2(d-k)+1) < k(d-k+1) \leq \mu^{\mathrm{SAM}}_{(d,k)}$)
\end{theorem}

We can look at specific cases of \Cref{thm:sam-som-upper-bound}. For example, for $(d=4, k=3)$, we have
\begin{corollary}
    The growth constant of self-osculating volumes (SOVs), ODWs, and self-avoiding volumes (SAVs) on the tesseractic lattice is upper bounded as
    \begin{align}
        \mu_{\tesseract}^\rm{SAV} \leq \mu_{\tesseract}^\rm{ODW} \leq \mu_{\tesseract}^\rm{SOV} \leq \frac{9375}{256} = 36.6211.
    \end{align}
\end{corollary}
For $(d=3,k=2)$, we have the upper bound for SASs as $\mu^\rm{SAS}_\cube \leq 81/4 = 20.25$, which is already improved compared to existing upper bound $\mu^\rm{SAS}_\cube \leq 31.2504$ of \cite{van1989self}\footnote{This is after correcting a small mistake in their proof.}. In \Cref{sec:twig-method}, we will improve these bounds via adapting the `twig' method.  Finally, note that setting $k=1$ in \Cref{thm:sam-som-upper-bound}, the upper bound is equal to the connective constant of non-reversing walks, $2d-1$.

\subsection{General strategy to upper bound $c_n$}
First, we describe a general strategy to upper bound $c_n$. We will first treat the case of SASs as considered by \cite{van1989self}. We refine their argument to achieve an improved bound for any $d$. Next, we will generalise the refined argument to arbitrary $(d, k)$-SAMs, $(d,k)$-SOMs, and $(d,k)$-XDs.

Consider $\Sigma$, a SAS in $\mathbb{Z}^d$ composed of $n$ faces. Each face can be located using a vector of its centre coordinates. This allows lexicographical ordering of the faces. Let $f^{(1)}$ be the face that is the smallest in lexicographical order. As it is lexicographically the smallest, there are no faces below $f^{(1)}$ in the 1st dimension in $\Sigma$.

$f^{(1)}$ neighbours four edges with centre coordinates $\bm{f}^{(1)} \pm \frac12 \bm{j}_{1,2}$ where $j_{1,2} \in \mathrm{HalfInt}(\bm{f}^{(1)})$. The centre coordinates of these edges can again be lexicographically ordered. In lexicographical order of the edges, if there is another face attached to an edge, we label it $\bm{f}^{(2)}$ then $\bm{f}^{(3)}$ and so on, with increasing index. After this, we do the same analysis for $\bm{f}^{(2)}$, labelling any faces that are previously not labelled. In this way, we can unambigously labelled the faces.

Note that this numbering scheme also works for SOSs and XDs with small modifications. For SOSs, we will need to check whether the new unindexed face neighbouring the edges neighbours the original face, which need not be if there are more than two faces neighbouring an edge. Only if the new unindexed face is neighbouring the original face will it be indexed at that step. In the case of XDs, more than two faces can neighbour an edge. Therefore, the (possibly) multiple faces connected via an edge are again indexed in lexicographical order.

Back to SASs, \cite{van1989self} uses the following scheme to specify the surface configurations. At each edge, there are $(2d-3)$ orientations in which another face can be attached to $\bm{f}_1$ via this edge, which can be ordered lexicographically. So, for each of the $4$ edges neighbouring a face, we consider $(2d-3)$ boolean variables, which are all set to $0$ is there is no face with a higher index attached to it, and if there is a face, one of the $(2d-3)$ of the boolean variables is set to $1$, corresponding to the orientation of the attached face with a higher index. The last face cannot have neighbours with a higher index, and therefore does not have the boolean variables. Additionally, the orientation of the first face must be specified, of which there are $d \choose 2$.

It then follows that any SAS of area $n$ can be specified as one $d \choose 2$-nary variable followed by a sequence of $4(2d-3)(n-1)$ boolean variables. That is, there is an injection between SASs of area $n$ and such sequences, since not all such sequences result in a SAS.

One can further constrain the space of these sequences by noting that there should be exactly $(n-1)$ boolean variables that are set to $1$ and the rest to $0$. This is because one boolean on an edge is turned on only when it introduces a face with a new index, which we are introducing $(n-1)$ of. Then the number of SASs with area $n$ can be upper bounded as

\begin{align}
    c^\mathrm{SAS}_{\mathbb{Z}^d,n} \leq {d \choose 2} {4(2d-3)(n-1) \choose n-1} \leq {d \choose 2} K^{n-1}, \quad \text{where} \quad K = \frac{(8 d-12)^{8 d-12}}{(8 d-13)^{8 d-13}}.
\end{align}

Here, standard bounds on the binomial coefficient was used to exponentially bound $c_n$~\footnote{This is a minor correction to Theorem 2.1 of \cite{van1989self}, where it states that $K = (4(2d-3)+1)^{4(2d-3)+1} / (4(2d-3)-1)^{4(2d-3)-1} = (8d-11)^{8d-11}/(8d-13)^{8d-13}$, which is incorrect.}.

Note that if one can prove that the growth constant exists, then it can be upper bounded by $K$.

\subsection{Upper bound for self-avoiding and osculating $k$-manifolds on the $d$-dim hypercubic lattice}
We now provide a more efficient scheme of encoding SASs as well as other restricted and higher dimensional generalisations, giving a better upper bound as following.
\begin{lemma} \label{lem:som-num-upper-bound}
    The number of $(d,k)$-SOMs with $k$-area $n$ can be exponentially upper bounded as
    \begin{align}
        c^\mathrm{SOM}_{(d,k),n} \leq \binom{d}{k} \left(\frac{(2 k-1)^{2 k-1}}{(2k-2)^{2 k-2}}\right)^n (2(d-k)+1)^{n-1}.
    \end{align}
\end{lemma}

\begin{proof}
    The previous indexing scheme can be improved by modifying the two following aspects.

First, we assign one $(2d-2)$-ary number to each edge instead of $(2d-3)$ binary numbers. One each edge of a face, we will assign a variable $q \in \{0, 1, 2, 3, \dots, 2d-3\}$, where
\begin{itemize}
    \item if no other face of a higher index is attached to the edge, then we set $q = 0$.
    \item If another face of a higher index is attached to the edge, then $q \in \{1,2,3...,2d-3\}$, where $2d-3$ corresponds to the possible ways two surfaces can share the given edge.
\end{itemize}

Second, we only assign $q$ to three of the edges instead of all four edges to a face, except the first face $f^{(1)}$. This is because one of the edges always connects back to another face with a lower index, and therefore can be unspecified. This leaves us with the first face $\bm{f}^{(1)}$, which we will equip with four $q$ variables, which will not change the overall limit. Again, the last face does not need to have variables assigned.

We can then upper bound the number of $n$ SOSs in $d$ dimension, $c_{\mathbb{Z}^d, n}^{\mathrm{SOS}}$, by considering all possible sequences of $q$ variables such that, out of the total $3(n-2) + 4 = 3(n-1) + 1$ of $q$ variables, exactly $n-1$ of them are nonzero. This constraint imposes that $n-1$ edges are shared, so that there are $n$ surfaces in total. The $q$ variables assigned to these $n-1$ edges can take any value in $\{1,2,3, \dots,2d-3\}$. Therefore, we get the following improved upper bound:

\begin{equation}
    c_{\mathbb{Z}^d, n}^{\mathrm{SOS}} \leq \binom{d}{2}\binom{3(n-1) + 1}{n-1} (2d-3)^{n-1} \leq {d \choose 2} \left(\frac{27}{4}\right)^n (2d-3)^{n-1},
\end{equation}
The factor $\binom{3(n-1) + 1}{n-1}$ accounts for all possible ways in which we can choose the $n-1$ nonzero $q$ variables out of a total of $3(n-2)+4$. The factor $(2d-3)^{n-1}$ reflects the possible values that these nonzero $q$ variables can take and $\binom{d}{2}$ are the possible orientations of the initial surface. This serves as an upper bound for $c_n(d,2)$ because, although the construction ensures that at most two surfaces can share a given edge, it does not eliminate the possibility of overlapping surfaces.

A similar method can be used to upper bound the number restricted $k$-manifolds in $\mathbb{Z}^{d}$ with $n$ elements, denoted $c_n(d,k)$. In this case, we glue together $n$ $k$-faces at their $(k-1)$-edges such that each $(k-1)$-edge is used at most twice.

The bound is given by:
\begin{equation}
    c_{(d,k), n}^{\rm{SOM}} \leq \binom{d}{k} \binom{(2k-1)(n-1) + 1}{n-1} (2(d-k)+1)^{n-1}.
\end{equation}
The binomial $\binom{d}{k}$ reflects the possible orientations of the first $k$-cube. The factor $\binom{(2k-1)(n-1)+1}{n-1}$ indicates the possible ways in which we can choose the $n-1$ nonzero $q$ variables out of a total of $(2k-1)(n-2) + 2k = (2k-1)(n-1) + 1$. This follows because a $k$-cube has $2k$ $(k-1)$-edges. Therefore, given a $k$-cube we assign $2k-1$ $q$-variables, since at most $2k-1$ of its $(k-1)$-edges can be shared except for the $k$-face with the first index. There are $2(d-k)+1$ ways to attach an additional $k$-cube to a $(k-1)$-face of a given $k$-cube; therefore, $q \in {0,1,2...,2(d-k)+1}$. Again applying standard bounds on the binomial coefficient gives the result of the lemma. 
\end{proof}

\subsection{Upper bound for $(d,k)$-polyominoids}
A $(d,k)$-XD is a sequences of $k$-faces that are obtained by attaching the next $k$-face to one of its $(k-1)$-edge, allowing for a $(k-1)$-edge to be used multiple times, with the constraint that there are no overlapping $k$-faces. Using a similar strategy as before, we will prove the following lemma.
\begin{lemma} \label{lem:xd-num-upper-bound}
    The number of $(d,k)$-XDs of $k$-area $n$ is upper bounded as
    \begin{align}
        c^\mathrm{XD}_{(d,k), n} \leq {d \choose k} \left(\frac{w(d,k)^{w(d,k)}}{(w(d,k)-1)^{w(d,k)-1}}\right)^{n-1},
    \end{align}
    where $w(d, k) = (2k-1)(2(d-k)+1)$.
\end{lemma}

In this case, we can assign to each of the $(k-1)$-edge (minus $1$, except for $\bm{f}^{(1)}$) of a $k$-faces as many binary variables as there are orientations. We do not assign binary variables to one $(k-1)$-edge of the $k$-face because at least one of these $(k-1)$-faces is shared; hence we do not want to have overlapping binary variables located at the shared $(k-1)$-edge.
If the $i$-th binary variable $b_i$ belonging to a $(k-1)$-edge of a given $k$-face is $0$ it means that we do not glue a $k$-cube associated to the $i$-th orientation to that face. On the other hand, if the variable is $1$ then the cube is attached to the face. A $k$-cube has $2k$ $(k-1)$-edges. Given a $(k-1)$-edge belonging to a $k$-face, there are $2(d-k)+1$ possible orientations in which the next a $k$-cube can be glued. Therefore, if we have $n$ $k$-cube, there are a total of $(2k-1)(2(d-k)+1)n$ binary variables. Since we want to upper bound the number of $(d,k)$-polyomionoid with $n$ $k$-faces, $c_{(d,k),n}^\rm{XD}$ we need to force $n-1$ of these binary variable to take the value $1$ and all the other to be $0$. This reasoning gives the following upper-bound:
\begin{equation}
    c_{(d,k),n}^\rm{XD} \leq {d \choose k} \binom{w(d,k)(n-1) + (2(d-k)+1)}{n-1}.
\end{equation}
with $w(d,k) = (2k-1)(2(d-k)+1)$ which is the number of binary variables at each $(k-1)$-face. The binomial coefficient is exponentially upper bounded to retrieve the lemma.

\subsection{Upper bound for closed self-avoiding manifolds}
For closed $(d,k)$-SAMs, we can prove a tighter upper bound.
\begin{lemma}[] \label{lem:closed-som-num-upper-bound}
    The number of closed $(d,k)$-SAMs with hyperarea $n$ is upper bounded by
    \begin{align}
        c^{\mathrm{SAM}}_{(d,k),n}(h=0) \leq {d \choose k} (2(d-k)+1)^{n-1}.
    \end{align}
\end{lemma}
\begin{proof}
    Because the manifold is closed, every hyperedge neighbouring a hyperface in the manifold has another hyperface attached to it. Therefore, we consider the following scheme to build a closed manifold of area $n$.

    Consider the first $k$-face $\bm{f}^{(1)}$, corresponding to the lexicographically smallest $k$-face in the manifold. Consider all of its hyperedges in lexicographical order. Due to the closedness of the manifold, there must be another $k$-face attached to the first $k$-face via this $(k-1)$-edge. Specify the orientation in which it is connected, with a $(2(d-k)+1)$-ary number, and index the new hyperface. Repeat this for all hyperedges bounding $\bm{f}^{(1)}$ in lexicographical order.

    Move to $\bm{f}^{(2)}$. Now, consider all of its neighbouring hyperedges that are not neighbouring other indexed hyperfaces. For each of such hyperedges, specify how a new hyperface is attached to $\bm{f}^{(2)}$ via a $(2(d-k)+1)$-ary number.

    This scheme can be continued until $(n-1)$ hyperfaces have been introduced. Each time a hyperface is introduced, it is specified by a new $(2(d-k)+1)$-anry number. Therefore, combined ${d \choose k}$ possible orientations of the first hyperface, there are ${d \choose k} (2(d-k)+1)^{n-1}$ possible ways of specifying the connections between the hyperfaces. All closed SAMs can be specified in this way, but not all sequences of $(2(d-k)+1)$-anry variables will result in a SAM. Hence, we have an injection and we retrieve the lemma.
\end{proof}

\subsection{Concatenation of manifolds and existence of growth constants}
In this subsection, we will prove \Cref{thm:existence}. To do this, we will develop a method to concatenation two manifolds together, which generalises the procedure introduced in \cite{van1989self}.

\begin{theorem}[Concatenation of manifolds] \label{thm:concatenation}
    For two $(d,k)$-SAMs in $\mathrm{SAM}_n(h)$ and $\mathrm{SAM}_m(g)$, there exists a concatenation procedure that forms an injection,
    \begin{align}
        \mathrm{SAM}_n(h) \times \mathrm{SAM}_m(g) \rightarrowtail \mathrm{SAM}_{n+m+2(k^2+3k-1)}(h+g).
    \end{align}
    The same concatenation procedure can be used for SOMs and $(d,k)$-XDs,
    \begin{align}
        \mathrm{SOM}_n \times \mathrm{SOM}_m & \rightarrowtail \mathrm{SOM}_{n+m+2(k^2+3k-1)}, \\
        \mathrm{XD}_n \times \mathrm{XD}_m & \rightarrowtail \mathrm{XD}_{n+m+2(k^2+3k-1)}.
    \end{align}
\end{theorem}

Because this is an injection, this immediately gives us the following pseudo-superadditivity properties,
\begin{corollary}[Pseudo-superadditivity of restricted manifolds] \label{lem:pseudo-superadditivity}
    For SAM, SOM, and XDs with $1<k<d$,
    \begin{align}
        c_n c_m & \leq c_{n+m+2(k^2+3k-1)}.
    \end{align}
    In particular, for SAMs,
    \begin{align}
        c_n(h) c_m(g) & \leq c_{n+m+2(k^2+3k-1)}(h+g),
    \end{align}
    and therefore
    \begin{align}
        c_n(0) c_m(0) & \leq c_{n+m+2(k^2+3k-1)}(0).
    \end{align}
\end{corollary}
From which the existence of the upper bound follows immediately:
\begin{proof}[Proof of \Nref{thm:existence}]
    By Theorem 1 of \cite{wilker1979extension}, if $c_n$ is exponentially upper bounded and $c_n c_m \leq c_{n+y(m)}$, where $\allowbreak \lim_{m \rightarrow \infty } \allowbreak (1/m)y(m) \allowbreak =1$, and $\mu_n$ is bounded from above, then $\mu$ exists. By \Cref{lem:sam-h-1-lower-bound,lem:xd-num-upper-bound,lem:sam-h-0-lower-bound}, $c_n$ for $\mathrm{SAM}_{(d,k)}(0)$, $\mathrm{SAM}_{(d,k)}$, and $\mathrm{XD}_{(d,k)}$ are exponentially upper bounded. By \Cref{lem:pseudo-superadditivity}, $y(m) = m+2(k^2 + 3k -1)$ for all three restricted surfaces. Therefore $\mu$ exists for $\mathrm{SAM}_{(d,k)}(0)$, $\mathrm{SAM}_{(d,k)}$, $\mathrm{SOM}_{(d,k)}$, and $\mathrm{XD}_{(d,k)}$.
\end{proof}

Now that the existence of the growth constants are proved, the upper bounds for the growth constants, \Cref{thm:sam-som-upper-bound,thm:xd-upper-bound,thm:sam-0-upper-bound} follow immediately from \Cref{lem:som-num-upper-bound,lem:xd-num-upper-bound,lem:closed-som-num-upper-bound}.

Before we start the proof of \Cref{thm:concatenation}, it will be useful to prove the following lemma about the orientation of the topmost (or bottommost) hyperface.
\begin{lemma} \label{lem:f1-not-boundary}
    Consider a $k$-manifold $\Sigma$. If the lexicographically smallest or largest face $\bm{f}^{(1)} \in \Sigma$ is not part of a boundary, then $1 \not \in \mathrm{HalfInt}(\bm{f}^{(1)})$.
\end{lemma}

\begin{proof}
    Suppose that $\bm{f}^{(1)}$ is not part of a boundary but $1 \not \in \mathrm{Int}(\bm{f}^{(1)})$, i.e. $1 \in \mathrm{HalfInt}(\bm{f}^{(1)})$. Then $\bm{f}^{(1)}$ is neighbouring a hyperedge at $\bm{l} = \bm{f}^{(1)} - \frac{1}{2}\bm{1}$, which must neighbour another hyperface in $\Sigma$. Now consider all possible places where the neighbouring hyperface could be:
    \begin{itemize}
        \item $\bm{l} + \frac12 \bm{1}$: this is $\bm{f}^{(1)}$.
        \item $\bm{l} - \frac12 \bm{1}$: this face has lexicographically smaller coordinates than $\bm{f}^{(1)}$.
        \item $\bm{l} \pm \frac12 \bm{i}$ $\forall i \in \mathrm{Int}(\bm{f}_1)$: again, this is face has coordinates lexicographically smaller than $\bm{f}^{(1)}$.
    \end{itemize}
    Therefore, this is a contradiction. Therefore, we must have that $1 \in \mathrm{Int}(\bm{f}^{(1)})$. Similar arguments hold if $\bm{f}^{(1)}$ is the lexicographically highest face.
\end{proof}

\begin{corollary}
    Consider a $k$-manifold $\Sigma$. If the lexicographically smallest or largest face $\bm{f}^{(1)} \in \Sigma$ has $1 \in \mathrm{HalfInt}(\bm{f}^{(1)})$, then $\bm{f}^{(1)}$ is part of a boundary.
\end{corollary}
\begin{proof}
    Suppose that $1 \in \mathrm{HalfInt}(\bm{f}^{(1)})$, but $\bm{f}^{(1)}$ is not part of a boundary. By \Cref{lem:f1-not-boundary}, it must be part of a boundary. Therefore it is a contradiction.
\end{proof}

\begin{corollary}
    If a $k$-manifold $\Sigma$ is closed, then both the lexicographically smallest or largest face $f^{(1)}$ have $1 \in \mathrm{Int}(\bm{f}^{(1)})$.
\end{corollary}

We now prove the concatenation theorem.

\begin{proof}[Proof of \Nref{thm:concatenation}]
    Consider two manifolds $\Sigma_\mathrm{top}$ and $\Sigma_{\mathrm{bot}}$, made up of $n$ and $m$ hyperfaces, respectively. We will describe a method to concatenate these two manifolds. Throughout this proof, a hyperface will refer to a $k$-face, a hyperedge will refer to a $(k-1)$-edge, and a hypercube will refer to a $(k+1)$-cube. When we refer to an object being `above' or `below' another, this means that it is higher or lower in the $1$st dimension, respectively. 

    Our strategy will be to `add' the boundaries of hypercubes on the top of $\Sigma_{\mathrm{bot}}$ and the bottom of $\Sigma_\mathrm{top}$ to join them together. We define the addition operation to removes a hyperface if there is already a hyperface there, and adds one if there isn't~\footnote{This can be thought of as representing the manifold as a vector in a $\mathbb{Z}_2$-valued vector space spanned by orthonormal basis vectors corresponding to the faces whose coefficient is $0$ if the hyperface is not present and $1$ if there is one.}. Since each boundary of a hypercube does not have any boundaries, this operation does not add any new boundary components nor combine existing boundaries, and therefore should result in a new manifold where the number of boundary components add.

    Recall that two $k$-faces $f$ and $f'$ can only be connected if they share at least $k-1$ orientations (or half-integer dimensions), i.e. $\abs{\mathrm{HalfInt}(\bm{f}) \cap \mathrm{HalfInt}(\bm{f}')} = k-1$, or equivalently, if their orientations differ by 1, i.e. $\abs{\mathrm{HalfInt}(\bm{f}) \backslash \mathrm{HalfInt}(\bm{f}')} = \abs{\mathrm{HalfInt}(\bm{f}') \backslash \mathrm{HalfInt}(\bm{f})} = 1$. 
    $(k+1)$-cubes can neighbour each other if they share $k$ orientations.
    
    In general, the orientations of $\bm{f}^{(m)}_\mathrm{bot}$ and $\bm{f}^{(1)}_\mathrm{top}$ will not be compatible to concatenate. Let the new orientations (or half-integer dimensions) that need to be added to connect $f_\mathrm{bot}^{(m)}$ to $f_\mathrm{top}^{(1)}$ be $\mathrm{Buy} := \mathrm{HalfInt}(\bm{f}^{(1)}_{\mathrm{top}}) \backslash \mathrm{HalfInt}(\bm{f}^{(m)}_{\mathrm{bot}})$, and the orientations to be be thrown away be $\mathrm{Sell} := \mathrm{HalfInt}(\bm{f}^{(m)}_{\mathrm{bot}}) \backslash \mathrm{HalfInt}(\bm{f}^{(1)}_{\mathrm{top}})$.

    The worst case scenario is if there is no overlap between the orientations of $f_\mathrm{bot}^{(m)}$ and $f_\mathrm{top}^{(1)}$, i.e. $\mathrm{HalfInt}(f_\mathrm{bot}^{(m)}) \cap \mathrm{HalfInt}(f_\mathrm{top}^{(1)}) = \varnothing$. In this case, $\abs{\mathrm{Buy}} = k$ and $\abs{\mathrm{Sell}}=k$~\footnote{The construction could be improved for the case where $k > \lceil d/2 \rceil$, where even in the worst case, the two faces would share orientations. However, since our procedure still works in this case and is enough to prove the pseudo-subadditivity and therefore existence of the growth constant subsequently, we do not attempt a more `efficient' proof.}.

    In the first step, we add `buffer' faces such that there is no risk that subsequently added faces will interact (i.e. intersect or neighbour) with existing faces of $\Sigma_\mathrm{bot}$ or $\Sigma_\mathrm{top}$ except $\bm{f}^{(m)}_\mathrm{bot}$ or $\bm{f}^{(1)}_\mathrm{top}$, respectively.

    For $\Sigma_\mathrm{bot}$, consider its topmost hyperface, $f^{(m)}_\mathrm{bot}$. 
    \begin{itemize}
        \item If $1 \not \in \mathrm{HalfInt}(\bm{f}^{(1)}_{\mathrm{top}})$, add the boundaries of two hypercube above $\bm{f}^{(m)}_\mathrm{bot}$ it at $\bm{b}'_0 = \bm{f}^{(m)}_{\mathrm{bot}} + \frac12 \bm{1}$ and $\bm{b}_0 = \bm{b}'_0 + \bm{1}$. Recall that the boundary of $(k+1)$-cube has $(2k+2)$ of $k$-faces. Whenever a hyperface is `added' where there is a already one (i.e. they overlap), both the existing hyperface and the new hyperface are removed instead. Therefore whenever the boundary of a neighbouring hypercube (consisting of $2k+2$ hyperfaces) is added, it net contributes $2k+2-1-1=2k$ faces to the manifold. Therefore adding the two hypercubes contribute $4k$ hyperfaces. 
        After this step, it is guaranteed that adding any hypercubes at the same height in the $1$st dimension but neighbouring in others does not intersect with any of the existing hyperfaces in $\Sigma_\mathrm{bot}$, since the boundaries of such hypercubes would have the first coordinate greater than $[\bm{f}^{(m)}_\mathrm{bot}]_1 + 1$.         
        \item If $1 \in \mathrm{HalfInt}(\bm{f}^{(m)}_{\mathrm{bot}})$, first add $2k-1$ hyperfaces at $\bm{f}_{\mathrm{bot}}(\tau) = \bm{f}^{(m)}_{\mathrm{bot}} + \tau \bm{1}$, where $\tau = 1, 2, \dots, 2k-1$. Then,
        \begin{itemize}
            \item if $\mathrm{Buy}, \mathrm{Sell} \neq \varnothing$, add a hyperface to $\bm{f}_\mathrm{bot}(2k) = \bm{f}_{\mathrm{bot}}(2k-1) + \frac12 \bm{1} + \bm{u}$, where $\bm{u}$ is first dimension of $\mathrm{Buy}$ when ordered. Remove $\bm{u}$ from both $\mathrm{Buy}$ and $\mathrm{Sell}$.
            \item If $\mathrm{Buy}, \mathrm{Sell} = \varnothing$, add a hyperface to $\bm{f}_\mathrm{bot}(2k) = \bm{f}_{\mathrm{bot}}(2k-1) + \frac12 \bm{1} - \frac12 \bm{i}_\mathrm{min}$, where $\bm{i}_\mathrm{min}$ is the smallest dimension in $\mathrm{Int}(\bm{f}^{(m)}_{\mathrm{bot}})$. In this case, $1 \in \mathrm{HalfInt}(\bm{f}^{(1)}_\mathrm{top})$ as well.
        \end{itemize}
        Since $\bm{f}_{\mathrm{bot}}(2k)$ has $1$ as an orientation, i.e. $1 \in \mathrm{HalfInt}(\bm{f}_{\mathrm{bot}}(2k))$, we can and do add the boundary of a hypercube at $\bm{b}_0 = \bm{f}_{\mathrm{bot}}(2k) + \bm{1}$. Overall, this procedure adds $4k$ hyperfaces to $\Sigma_\mathrm{bot}$. Since $k \geq 2$ and we added $\geq 3$ hyperfaces above $\bm{f}^{(1)}_\mathrm{bot}$ which has $1$ as an orientation, boundaries of such hypercubes is guaranteed to have the first coordinate greater than $[\bm{f}^{(m)}_\mathrm{bot}]_1 + 3.5$. 

    \end{itemize}
    Similarly, for $\Sigma_\mathrm{top}$, we look at its bottom face $f^{(1)}_\mathrm{top}$. We add two hypercubes below it if $1 \not \in \mathrm{HalfInt}(\bm{f}^{(m)}_\mathrm{bot})$, and if $1 \in \mathrm{HalfInt}(\bm{f}^{(m)}_\mathrm{bot})$, add $4k+2$ hyperfaces below it in a similar manner. Note that for the case where $\mathrm{Buy}, \mathrm{Sell} \neq \varnothing$, $\bm{f}_\mathrm{top}(2k+2) = \bm{f}_\mathrm{top}(2k+1) - \frac12 \bm{1} - \bm s$, where $\bm s$ is the first index of $\mathrm{Sell}$ when ordered, then remove it from both from both $\mathrm{Buy}$ and $\mathrm{Sell}$. Note also that in the case where $\mathrm{Buy}, \mathrm{Sell} = \varnothing$, $\bm{i}_\mathrm{min}$ is the same. We will refer to the final hypercube from as $\bm{b}_{-1}$.

    As $4k$ hyperfaces are added on the top and on the bottom components, $8k$ hyperfaces are added in this step.

    At this point, $\bm{b}_0$ and $\bm{b}_{-1}$ always share $1$ as one of its orientations, but could differ on all other orientations, i.e. $k$ orientations. In the second step, we add hypercubes from the bottom and change orientations until the bottom and top components become compatible to concatenate.
    
    For $t=1,2, \dots, (k-1)$, if $t \leq \abs{\mathrm{Buy}} - 1 = \abs{\mathrm{Sell}} - 1$, we add hypercubes at $\bm{b}_t = \bm{b}_{t-1} + \frac12 \bm{u}_t - \frac12 \bm{s}_t$, where $\bm{u}_t$ and $\bm{s}_t$ is the $t$\textsuperscript{th} element of $\mathrm{Buy}$ and $\mathrm{Sell}$ respectively. If $t > \abs{\mathrm{Buy}} = \abs{\mathrm{Sell}}$, we add hypercubes at $\bm{b}_t = \bm{b}_{t-1} + \bm{1}$.

    The second step adds $(k-1)$ $(k+1)$-cubes, and therefore adds $(k-1) \times 2k$ $k$-faces.

    Finally, we translate the top component such that $\bm{b}_{-1}$ is moved to where $\bm{b}_k$ would be with the rules above. This will remove $2$ $k$-faces.

    This completes the concatenation procedure, which added a net of $8k + (k-1) \times 2k - 2 = 2(k^2+3k-1)$ faces. It is schematically illustrated in \Cref{fig:concatenation}.

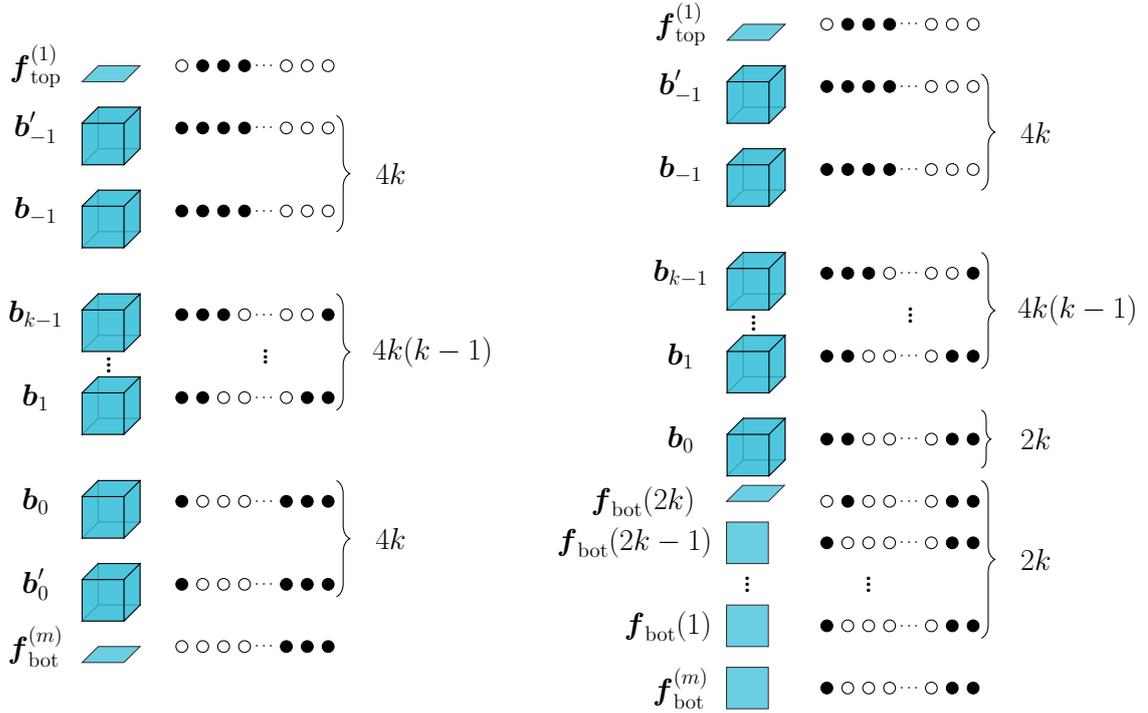
\begin{figure}[h]
    \centering
    
    \begin{subfigure}[c]{0.45\textwidth}
    \centering
    \begin{tikzpicture}[scale = 0.55]
    \begin{scope}[shift={(0,2)}, transform shape]
    \draw[fill=SkyBlue, opacity=0.8] 
    (0,0,0) -- (1,0,0) -- (1,0,1) -- (0,0,1) -- cycle;    
    \end{scope}

    \begin{scope}[shift={(2,2)}, transform shape]
    \node at (-3.5,0) {\huge $\bm{f}^{(1)}_\mathrm{top}$};
    \node[circle, draw, minimum size=8, inner sep=0] (C1) at (0,0) {};
    \node[circle, draw, fill = black, minimum size=8, inner sep=0pt] (C2) at (0.5,0) {};
    \node[circle, draw, fill = black, minimum size=8, inner sep=0pt] (C3) at (1.0,0) {};
    \node[circle, draw, fill = black, minimum size=8, inner sep=0] (C4) at (1.5,0) {};
    \node at (2,0) {$\cdots$};
    \node[circle, draw, minimum size=8, inner sep=0] (C5) at (2.5,0) {};
    \node[circle, draw, minimum size=8, inner sep=0] (C6) at (3,0) {};
    \node[circle, draw, minimum size=8, inner sep=0] (C7) at (3.5,0) {};
    \end{scope}

    \begin{scope}[shift={(0,0)}]
    
    \draw[fill=SkyBlue, opacity=0.8]
    (0,0,0) -- (0,1,0) -- (0,1,1) -- (0,0,1) -- cycle;
    
    \draw[fill=SkyBlue, opacity=0.8]
    (0,0,0) -- (1,0,0) -- (1,1,0) -- (0,1,0) -- cycle;
    \draw[fill=SkyBlue, opacity=0.8]
    (0,0,0) -- (1,0,0) -- (1,0,1) -- (0,0,1) -- cycle;
    
    \draw[fill=SkyBlue, opacity=0.8]
    (1,0,0) -- (1,1,0) -- (1,1,1) -- (1,0,1) -- cycle;
    
    \draw[fill=SkyBlue, opacity=0.8]
    (0,0,1) -- (1,0,1) -- (1,1,1) -- (0,1,1) -- cycle;
    
    \draw[fill=SkyBlue, opacity=0.8]
    (0,1,0) -- (1,1,0) -- (1,1,1) -- (0,1,1) -- cycle;
    \end{scope}

    \begin{scope}[shift={(2,0.5)}, transform shape]
    \node at (-3.5,0) {\huge $\bm{b}'_{-1}$};
    \node[circle, draw,  fill = black, minimum size=8, inner sep=0] (C1) at (0,0) {};
    \node[circle, draw,  fill = black, minimum size=8, inner sep=0pt] (C2) at (0.5,0) {};
    \node[circle, draw, fill = black, minimum size=8, inner sep=0pt] (C3) at (1.0,0) {};
    \node[circle, draw, fill = black, minimum size=8, inner sep=0] (C4) at (1.5,0) {};
    \node at (2,0) {$\cdots$};
    \node[circle, draw, minimum size=8, inner sep=0] (C5) at (2.5,0) {};
    \node[circle, draw, minimum size=8, inner sep=0] (C6) at (3,0) {};
    \node[circle, draw, minimum size=8, inner sep=0] (C7) at (3.5,0) {};
    \draw[decorate, decoration={brace, amplitude=5pt, mirror}] 
    (3.7,-2.5) -- (3.7,0.3) node[midway, right=15pt]{};
    \node at (5,-1.1) {\huge $4k$};
    
    \end{scope}
    
    \begin{scope}[shift={(0,-2)}]
    
    \draw[fill=SkyBlue, opacity=0.8]
    (0,0,0) -- (0,1,0) -- (0,1,1) -- (0,0,1) -- cycle;
    
    \draw[fill=SkyBlue, opacity=0.8]
    (0,0,0) -- (1,0,0) -- (1,1,0) -- (0,1,0) -- cycle;
    \draw[fill=SkyBlue, opacity=0.8]
    (0,0,0) -- (1,0,0) -- (1,0,1) -- (0,0,1) -- cycle;
    
    \draw[fill=SkyBlue, opacity=0.8]
    (1,0,0) -- (1,1,0) -- (1,1,1) -- (1,0,1) -- cycle;
    
    \draw[fill=SkyBlue, opacity=0.8]
    (0,0,1) -- (1,0,1) -- (1,1,1) -- (0,1,1) -- cycle;
    
    \draw[fill=SkyBlue, opacity=0.8]
    (0,1,0) -- (1,1,0) -- (1,1,1) -- (0,1,1) -- cycle;
    \end{scope}

    \begin{scope}[shift={(2,-1.5)}, transform shape]
    \node at (-3.5,0) {\huge $\bm{b}_{-1}$};
    \node[circle, draw,  fill = black, minimum size=8, inner sep=0] (C1) at (0,0) {};
    \node[circle, draw,  fill = black, minimum size=8, inner sep=0pt] (C2) at (0.5,0) {};
    \node[circle, draw, fill = black, minimum size=8, inner sep=0pt] (C3) at (1.0,0) {};
    \node[circle, draw, fill = black, minimum size=8, inner sep=0] (C4) at (1.5,0) {};
    \node at (2,0) {$\cdots$};
    \node[circle, draw, minimum size=8, inner sep=0] (C5) at (2.5,0) {};
    \node[circle, draw, minimum size=8, inner sep=0] (C6) at (3,0) {};
    \node[circle, draw, minimum size=8, inner sep=0] (C7) at (3.5,0) {};
    \end{scope}
    
    \begin{scope}[shift={(0,-4.5)}]
    
    \draw[fill=SkyBlue, opacity=0.8]
    (0,0,0) -- (0,1,0) -- (0,1,1) -- (0,0,1) -- cycle;
    
    \draw[fill=SkyBlue, opacity=0.8]
    (0,0,0) -- (1,0,0) -- (1,1,0) -- (0,1,0) -- cycle;
    \draw[fill=SkyBlue, opacity=0.8]
    (0,0,0) -- (1,0,0) -- (1,0,1) -- (0,0,1) -- cycle;
    
    \draw[fill=SkyBlue, opacity=0.8]
    (1,0,0) -- (1,1,0) -- (1,1,1) -- (1,0,1) -- cycle;
    
    \draw[fill=SkyBlue, opacity=0.8]
    (0,0,1) -- (1,0,1) -- (1,1,1) -- (0,1,1) -- cycle;
    
    \draw[fill=SkyBlue, opacity=0.8]
    (0,1,0) -- (1,1,0) -- (1,1,1) -- (0,1,1) -- cycle;
    \end{scope}

    \begin{scope}[shift={(2,-4)}, transform shape]
    \node at (-3.5,0) {\huge $\bm{b}_{k-1}$};
    \node[circle, draw,  fill = black, minimum size=8, inner sep=0] (C1) at (0,0) {};
    \node[circle, draw,  fill = black, minimum size=8, inner sep=0pt] (C2) at (0.5,0) {};
    \node[circle, draw, fill = black, minimum size=8, inner sep=0pt] (C3) at (1.0,0) {};
    \node[circle, draw, minimum size=8, inner sep=0] (C4) at (1.5,0) {};
    \node at (2,0) {$\cdots$};
    \node[circle, draw, minimum size=8, inner sep=0] (C5) at (2.5,0) {};
    \node[circle, draw, minimum size=8, inner sep=0] (C6) at (3,0) {};
    \node[circle, draw, fill = black, minimum size=8, inner sep=0] (C7) at (3.5,0) {};
    
    \end{scope}

    \begin{scope}[shift={(0,-6.5)}]
    
    \draw[fill=SkyBlue, opacity=0.8]
    (0,0,0) -- (0,1,0) -- (0,1,1) -- (0,0,1) -- cycle;
    
    \draw[fill=SkyBlue, opacity=0.8]
    (0,0,0) -- (1,0,0) -- (1,1,0) -- (0,1,0) -- cycle;
    \draw[fill=SkyBlue, opacity=0.8]
    (0,0,0) -- (1,0,0) -- (1,0,1) -- (0,0,1) -- cycle;
    
    \draw[fill=SkyBlue, opacity=0.8]
    (1,0,0) -- (1,1,0) -- (1,1,1) -- (1,0,1) -- cycle;
    
    \draw[fill=SkyBlue, opacity=0.8]
    (0,0,1) -- (1,0,1) -- (1,1,1) -- (0,1,1) -- cycle;
    
    \draw[fill=SkyBlue, opacity=0.8]
    (0,1,0) -- (1,1,0) -- (1,1,1) -- (0,1,1) -- cycle;
    \end{scope}

    \begin{scope}[shift={(2,-6)}, transform shape]
    \node at (-3.5,0) {\huge $\bm{b}_{1}$};
    \node[circle, draw,  fill = black, minimum size=8, inner sep=0] (C1) at (0,0) {};
    \node[circle, draw,  fill = black, minimum size=8, inner sep=0pt] (C2) at (0.5,0) {};
    \node[circle, draw, minimum size=8, inner sep=0pt] (C3) at (1.0,0) {};
    \node[circle, draw, minimum size=8, inner sep=0] (C4) at (1.5,0) {};
    \node at (2,0) {$\cdots$};
    \node[circle, draw, minimum size=8, inner sep=0] (C5) at (2.5,0) {};
    \node[circle, draw, fill = black, minimum size=8, inner sep=0] (C6) at (3,0) {};
    \node[circle, draw, fill = black, minimum size=8, inner sep=0] (C7) at (3.5,0) {};
    \draw[decorate, decoration={brace, amplitude=5pt, mirror}] (3.7,-0.3) -- (3.7,2.5) node[midway, right=2pt]{};
    \node at (6,1.1) {\huge $4k(k-1)$};
    \node at (2,1.1) {\huge $\vdots$};
    \node at (-1.75,0.9) {\huge $\vdots$};
    \end{scope}
    
    \begin{scope}[shift={(0,-9)}]
    
    \draw[fill=SkyBlue, opacity=0.8]
    (0,0,0) -- (0,1,0) -- (0,1,1) -- (0,0,1) -- cycle;
    
    \draw[fill=SkyBlue, opacity=0.8]
    (0,0,0) -- (1,0,0) -- (1,1,0) -- (0,1,0) -- cycle;
    \draw[fill=SkyBlue, opacity=0.8]
    (0,0,0) -- (1,0,0) -- (1,0,1) -- (0,0,1) -- cycle;
    
    \draw[fill=SkyBlue, opacity=0.8]
    (1,0,0) -- (1,1,0) -- (1,1,1) -- (1,0,1) -- cycle;
    
    \draw[fill=SkyBlue, opacity=0.8]
    (0,0,1) -- (1,0,1) -- (1,1,1) -- (0,1,1) -- cycle;
    
    \draw[fill=SkyBlue, opacity=0.8]
    (0,1,0) -- (1,1,0) -- (1,1,1) -- (0,1,1) -- cycle;
    \end{scope}

    \begin{scope}[shift={(2,-8.5)}, transform shape]
    \node at (-3.5,0) {\huge $\bm{b}_{0}$};
    \node[circle, draw,  fill = black, minimum size=8, inner sep=0] (C1) at (0,0) {};
    \node[circle, draw, minimum size=8, inner sep=0pt] (C2) at (0.5,0) {};
    \node[circle, draw, minimum size=8, inner sep=0pt] (C3) at (1.0,0) {};
    \node[circle, draw, minimum size=8, inner sep=0] (C4) at (1.5,0) {};
    \node at (2,0) {$\cdots$};
    \node[circle, draw, fill = black, minimum size=8, inner sep=0] (C5) at (2.5,0) {};
    \node[circle, draw, fill = black, minimum size=8, inner sep=0] (C6) at (3,0) {};
    \node[circle, draw, fill = black, minimum size=8, inner sep=0] (C7) at (3.5,0) {};
    \end{scope}
    
    \begin{scope}[shift={(0,-11)}, transform shape]
    
    \draw[fill=SkyBlue, opacity=0.8]
    (0,0,0) -- (0,1,0) -- (0,1,1) -- (0,0,1) -- cycle;
    
    \draw[fill=SkyBlue, opacity=0.8]
    (0,0,0) -- (1,0,0) -- (1,1,0) -- (0,1,0) -- cycle;
    \draw[fill=SkyBlue, opacity=0.8]
    (0,0,0) -- (1,0,0) -- (1,0,1) -- (0,0,1) -- cycle;
    
    \draw[fill=SkyBlue, opacity=0.8]
    (1,0,0) -- (1,1,0) -- (1,1,1) -- (1,0,1) -- cycle;
    
    \draw[fill=SkyBlue, opacity=0.8]
    (0,0,1) -- (1,0,1) -- (1,1,1) -- (0,1,1) -- cycle;
    
    \draw[fill=SkyBlue, opacity=0.8]
    (0,1,0) -- (1,1,0) -- (1,1,1) -- (0,1,1) -- cycle;
    \end{scope}
    
    \begin{scope}[shift={(2,-10.5)}, transform shape]
    \node at (-3.5,0) {\huge $\bm{b}'_{0}$};
    \node[circle, draw,  fill = black, minimum size=8, inner sep=0] (C1) at (0,0) {};
    \node[circle, draw, minimum size=8, inner sep=0pt] (C2) at (0.5,0) {};
    \node[circle, draw, minimum size=8, inner sep=0pt] (C3) at (1.0,0) {};
    \node[circle, draw, minimum size=8, inner sep=0] (C4) at (1.5,0) {};
    \node at (2,0) {$\cdots$};
    \node[circle, draw, fill = black, minimum size=8, inner sep=0] (C5) at (2.5,0) {};
    \node[circle, draw, fill = black, minimum size=8, inner sep=0] (C6) at (3,0) {};
    \node[circle, draw, fill = black, minimum size=8, inner sep=0] (C7) at (3.5,0) {};
    \draw[decorate, decoration={brace, amplitude=5pt, mirror}] 
    (3.7,-0.3) -- (3.7,2.5) node[midway, right=15pt]{};
    \node at (5,1.1) {\huge $4k$};
    \end{scope}
    
    \begin{scope}[shift={(0,-12)}, transform shape]
    \draw[fill=SkyBlue, opacity=0.8] 
    (0,0,0) -- (1,0,0) -- (1,0,1) -- (0,0,1) -- cycle;
    \end{scope}

    
    
    
    
    
    \begin{scope}[shift={(2,-12)}, transform shape]
    \node at (-3.5,0) {\huge $\bm{f}^{(m)}_\mathrm{bot}$};
    \node[circle, draw, minimum size=8, inner sep=0] (C1) at (0,0) {};
    \node[circle, draw, minimum size=8, inner sep=0pt] (C2) at (0.5,0) {};
    \node[circle, draw, minimum size=8, inner sep=0pt] (C3) at (1.0,0) {};
    \node[circle, draw, minimum size=8, inner sep=0] (C4) at (1.5,0) {};
    \node at (2,0) {$\cdots$};
    \node[circle, draw, fill = black, minimum size=8, inner sep=0] (C5) at (2.5,0) {};
    \node[circle, draw, fill = black, minimum size=8, inner sep=0] (C6) at (3,0) {};
    \node[circle, draw, fill = black, minimum size=8, inner sep=0] (C7) at (3.5,0) {};
    \end{scope}
    \end{tikzpicture}
    \end{subfigure}
    \hspace{0.5em} 
    \begin{subfigure}[c]{0.45\textwidth}
    \centering
    \begin{tikzpicture}[scale = 0.55]
    \begin{scope}[shift={(0,4)}, transform shape]
    \draw[fill=SkyBlue, opacity=0.8] 
    (0,0,0) -- (1,0,0) -- (1,0,1) -- (0,0,1) -- cycle;    
    \end{scope}

    \begin{scope}[shift={(2,4)}, transform shape]
    \node at (-3.5,0) {\huge $\bm{f}^{(1)}_\mathrm{top}$};
    \node[circle, draw, minimum size=8, inner sep=0] (C1) at (0,0) {};
    \node[circle, draw, fill = black, minimum size=8, inner sep=0pt] (C2) at (0.5,0) {};
    \node[circle, draw, fill = black, minimum size=8, inner sep=0pt] (C3) at (1.0,0) {};
    \node[circle, draw, fill = black, minimum size=8, inner sep=0] (C4) at (1.5,0) {};
    \node at (2,0) {$\cdots$};
    \node[circle, draw, minimum size=8, inner sep=0] (C5) at (2.5,0) {};
    \node[circle, draw, minimum size=8, inner sep=0] (C6) at (3,0) {};
    \node[circle, draw, minimum size=8, inner sep=0] (C7) at (3.5,0) {};
    \end{scope}

    \begin{scope}[shift={(0,2)}]
    
    \draw[fill=SkyBlue, opacity=0.8]
    (0,0,0) -- (0,1,0) -- (0,1,1) -- (0,0,1) -- cycle;
    
    \draw[fill=SkyBlue, opacity=0.8]
    (0,0,0) -- (1,0,0) -- (1,1,0) -- (0,1,0) -- cycle;
    \draw[fill=SkyBlue, opacity=0.8]
    (0,0,0) -- (1,0,0) -- (1,0,1) -- (0,0,1) -- cycle;
    
    \draw[fill=SkyBlue, opacity=0.8]
    (1,0,0) -- (1,1,0) -- (1,1,1) -- (1,0,1) -- cycle;
    
    \draw[fill=SkyBlue, opacity=0.8]
    (0,0,1) -- (1,0,1) -- (1,1,1) -- (0,1,1) -- cycle;
    
    \draw[fill=SkyBlue, opacity=0.8]
    (0,1,0) -- (1,1,0) -- (1,1,1) -- (0,1,1) -- cycle;
    \end{scope}

    \begin{scope}[shift={(2,2.5)}, transform shape]
    \node at (-3.5,0) {\huge $\bm{b}'_{-1}$};
    \node[circle, draw, fill = black, minimum size=8, inner sep=0] (C1) at (0,0) {};
    \node[circle, draw, fill = black, minimum size=8, inner sep=0pt] (C2) at (0.5,0) {};
    \node[circle, draw, fill = black, minimum size=8, inner sep=0pt] (C3) at (1.0,0) {};
    \node[circle, draw, fill = black, minimum size=8, inner sep=0] (C4) at (1.5,0) {};
    \node at (2,0) {$\cdots$};
    \node[circle, draw, minimum size=8, inner sep=0] (C5) at (2.5,0) {};
    \node[circle, draw, minimum size=8, inner sep=0] (C6) at (3,0) {};
    \node[circle, draw, minimum size=8, inner sep=0] (C7) at (3.5,0) {};
    \draw[decorate, decoration={brace, amplitude=5pt, mirror}] 
    (3.7,-2.5) -- (3.7,0.3) node[midway, right=15pt]{};
    \node at (5,-1.1) {\huge $4k$};
    
    \end{scope}
    
    \begin{scope}[shift={(0,0)}]
    
    \draw[fill=SkyBlue, opacity=0.8]
    (0,0,0) -- (0,1,0) -- (0,1,1) -- (0,0,1) -- cycle;
    
    \draw[fill=SkyBlue, opacity=0.8]
    (0,0,0) -- (1,0,0) -- (1,1,0) -- (0,1,0) -- cycle;
    \draw[fill=SkyBlue, opacity=0.8]
    (0,0,0) -- (1,0,0) -- (1,0,1) -- (0,0,1) -- cycle;
    
    \draw[fill=SkyBlue, opacity=0.8]
    (1,0,0) -- (1,1,0) -- (1,1,1) -- (1,0,1) -- cycle;
    
    \draw[fill=SkyBlue, opacity=0.8]
    (0,0,1) -- (1,0,1) -- (1,1,1) -- (0,1,1) -- cycle;
    
    \draw[fill=SkyBlue, opacity=0.8]
    (0,1,0) -- (1,1,0) -- (1,1,1) -- (0,1,1) -- cycle;
    \end{scope}

    \begin{scope}[shift={(2,0.5)}, transform shape]
    \node at (-3.5,0) {\huge $\bm{b}_{-1}$};
    \node[circle, draw, fill = black, minimum size=8, inner sep=0] (C1) at (0,0) {};
    \node[circle, draw, fill = black, minimum size=8, inner sep=0pt] (C2) at (0.5,0) {};
    \node[circle, draw, fill = black, minimum size=8, inner sep=0pt] (C3) at (1.0,0) {};
    \node[circle, draw, fill = black, minimum size=8, inner sep=0] (C4) at (1.5,0) {};
    \node at (2,0) {$\cdots$};
    \node[circle, draw, minimum size=8, inner sep=0] (C5) at (2.5,0) {};
    \node[circle, draw, minimum size=8, inner sep=0] (C6) at (3,0) {};
    \node[circle, draw, minimum size=8, inner sep=0] (C7) at (3.5,0) {};
    \end{scope}
    
    \begin{scope}[shift={(0,-2.5)}]
    
    \draw[fill=SkyBlue, opacity=0.8]
    (0,0,0) -- (0,1,0) -- (0,1,1) -- (0,0,1) -- cycle;
    
    \draw[fill=SkyBlue, opacity=0.8]
    (0,0,0) -- (1,0,0) -- (1,1,0) -- (0,1,0) -- cycle;
    \draw[fill=SkyBlue, opacity=0.8]
    (0,0,0) -- (1,0,0) -- (1,0,1) -- (0,0,1) -- cycle;
    
    \draw[fill=SkyBlue, opacity=0.8]
    (1,0,0) -- (1,1,0) -- (1,1,1) -- (1,0,1) -- cycle;
    
    \draw[fill=SkyBlue, opacity=0.8]
    (0,0,1) -- (1,0,1) -- (1,1,1) -- (0,1,1) -- cycle;
    
    \draw[fill=SkyBlue, opacity=0.8]
    (0,1,0) -- (1,1,0) -- (1,1,1) -- (0,1,1) -- cycle;
    \end{scope}

    \begin{scope}[shift={(2,-2)}, transform shape]
    \node at (-3.5,0) {\huge $\bm{b}_{k-1}$};
    \node[circle, draw, fill = black, minimum size=8, inner sep=0] (C1) at (0,0) {};
    \node[circle, draw, fill = black, minimum size=8, inner sep=0pt] (C2) at (0.5,0) {};
    \node[circle, draw, fill = black, minimum size=8, inner sep=0pt] (C3) at (1.0,0) {};
    \node[circle, draw, minimum size=8, inner sep=0] (C4) at (1.5,0) {};
    \node at (2,0) {$\cdots$};
    \node[circle, draw, minimum size=8, inner sep=0] (C5) at (2.5,0) {};
    \node[circle, draw, minimum size=8, inner sep=0] (C6) at (3,0) {};
    \node[circle, draw, fill = black, minimum size=8, inner sep=0] (C7) at (3.5,0) {};
    
    \end{scope}

    \begin{scope}[shift={(0,-4.5)}]
    
    \draw[fill=SkyBlue, opacity=0.8]
    (0,0,0) -- (0,1,0) -- (0,1,1) -- (0,0,1) -- cycle;
    
    \draw[fill=SkyBlue, opacity=0.8]
    (0,0,0) -- (1,0,0) -- (1,1,0) -- (0,1,0) -- cycle;
    \draw[fill=SkyBlue, opacity=0.8]
    (0,0,0) -- (1,0,0) -- (1,0,1) -- (0,0,1) -- cycle;
    
    \draw[fill=SkyBlue, opacity=0.8]
    (1,0,0) -- (1,1,0) -- (1,1,1) -- (1,0,1) -- cycle;
    
    \draw[fill=SkyBlue, opacity=0.8]
    (0,0,1) -- (1,0,1) -- (1,1,1) -- (0,1,1) -- cycle;
    
    \draw[fill=SkyBlue, opacity=0.8]
    (0,1,0) -- (1,1,0) -- (1,1,1) -- (0,1,1) -- cycle;
    \end{scope}

    \begin{scope}[shift={(2,-4)}, transform shape]
    \node at (-3.5,0) {\huge $\bm{b}_{1}$};
    \node[circle, draw, fill = black, minimum size=8, inner sep=0] (C1) at (0,0) {};
    \node[circle, draw, fill = black, minimum size=8, inner sep=0pt] (C2) at (0.5,0) {};
    \node[circle, draw, minimum size=8, inner sep=0pt] (C3) at (1.0,0) {};
    \node[circle, draw, minimum size=8, inner sep=0] (C4) at (1.5,0) {};
    \node at (2,0) {$\cdots$};
    \node[circle, draw, minimum size=8, inner sep=0] (C5) at (2.5,0) {};
    \node[circle, draw, fill = black,minimum size=8, inner sep=0] (C6) at (3,0) {};
    \node[circle, draw, fill = black,minimum size=8, inner sep=0] (C7) at (3.5,0) {};
    \draw[decorate, decoration={brace, amplitude=5pt, mirror}] 
    (3.7,-0.3) -- (3.7,2.5) node[midway, right=2pt]{};
    \node at (6,1.1) {\huge $4k(k-1)$};
    \node at (2,1.1) {\huge $\vdots$};
    \node at (-1.75,0.9) {\huge $\vdots$};
    \end{scope}

    \begin{scope}[shift={(0,-6.5)}]
    
    \draw[fill=SkyBlue, opacity=0.8]
    (0,0,0) -- (0,1,0) -- (0,1,1) -- (0,0,1) -- cycle;
    
    \draw[fill=SkyBlue, opacity=0.8]
    (0,0,0) -- (1,0,0) -- (1,1,0) -- (0,1,0) -- cycle;
    \draw[fill=SkyBlue, opacity=0.8]
    (0,0,0) -- (1,0,0) -- (1,0,1) -- (0,0,1) -- cycle;
    
    \draw[fill=SkyBlue, opacity=0.8]
    (1,0,0) -- (1,1,0) -- (1,1,1) -- (1,0,1) -- cycle;
    
    \draw[fill=SkyBlue, opacity=0.8]
    (0,0,1) -- (1,0,1) -- (1,1,1) -- (0,1,1) -- cycle;
    
    \draw[fill=SkyBlue, opacity=0.8]
    (0,1,0) -- (1,1,0) -- (1,1,1) -- (0,1,1) -- cycle;
    \end{scope}

    \begin{scope}[shift={(2,-6)}, transform shape]
    \node at (-3.5,0) {\huge $\bm{b}_{0}$};
    \node[circle, draw, fill = black, minimum size=8, inner sep=0] (C1) at (0,0) {};
    \node[circle, draw, fill = black, minimum size=8, inner sep=0pt] (C2) at (0.5,0) {};
    \node[circle, draw, minimum size=8, inner sep=0pt] (C3) at (1.0,0) {};
    \node[circle, draw, minimum size=8, inner sep=0] (C4) at (1.5,0) {};
    \node at (2,0) {$\cdots$};
    \node[circle, draw, minimum size=8,inner sep=0] (C5) at (2.5,0) {};
    \node[circle, draw, fill = black,minimum size=8, inner sep=0] (C6) at (3,0) {};
    \node[circle, draw, fill = black,minimum size=8, inner sep=0] (C7) at (3.5,0) {};
    \draw[decorate, decoration={brace, amplitude=3.5pt, mirror}] 
    (3.7,-0.7) -- (3.7,0.7) node[midway, right=2pt]{};
    \node at (5,0) {\huge $2k$};
    \end{scope}
    
    \begin{scope}[shift={(-0.4,-7.5)}]
    \draw[fill=SkyBlue, opacity=0.8] 
    (0,0,0) -- (1,0,0) -- (1,0,-1) -- (0,0,-1) -- cycle;
    \end{scope}
    
    \begin{scope}[shift={(2,-7.5)}, transform shape]
    \node at (-4.4,0) {\huge $\bm{f}_\mathrm{bot}(2k)$};
    \node[circle, draw, minimum size=8, inner sep=0] (C1) at (0,0) {};
    \node[circle, draw, fill = black, minimum size=8, inner sep=0pt] (C2) at (0.5,0) {};
    \node[circle, draw, minimum size=8, inner sep=0pt] (C3) at (1.0,0) {};
    \node[circle, draw, minimum size=8, inner sep=0] (C4) at (1.5,0) {};
    \node at (2,0) {$\cdots$};
    \node[circle, draw, minimum size=8, inner sep=0] (C5) at (2.5,0) {};
    \node[circle, draw, fill = black,minimum size=8, inner sep=0] (C6) at (3,0) {};
    \node[circle, draw, fill = black,minimum size=8, inner sep=0] (C7) at (3.5,0) {};
    \end{scope}
    
    \begin{scope}[shift={(-0.4,-9)}]
    \draw[fill=SkyBlue, opacity=0.8] 
    (0,0,0) -- (1,0,0) -- (1,1,0) -- (0,1,0) -- cycle;
    \end{scope}
    
    \begin{scope}[shift={(2,-8.5)}, transform shape]
    \node at (-4.6,0) {\huge $\bm{f}_\mathrm{bot}(2k-1)$};
    \node[circle, draw, fill = black, minimum size=8, inner sep=0] (C1) at (0,0) {};
    \node[circle, draw, minimum size=8, inner sep=0pt] (C2) at (0.5,0) {};
    \node[circle, draw, minimum size=8, inner sep=0pt] (C3) at (1.0,0) {};
    \node[circle, draw, minimum size=8, inner sep=0] (C4) at (1.5,0) {};
    \node at (2,0) {$\cdots$};
    \node[circle, draw, minimum size=8, inner sep=0] (C5) at (2.5,0) {};
    \node[circle, draw, fill = black,minimum size=8, inner sep=0] (C6) at (3,0) {};
    \node[circle, draw, fill = black,minimum size=8, inner sep=0] (C7) at (3.5,0) {};
    \end{scope}
    
    \begin{scope}[shift={(-0.4,-11)}, transform shape]
    \draw[fill=SkyBlue, opacity=0.8] 
    (0,0,0) -- (1,0,0) -- (1,1,0) -- (0,1,0) -- cycle;
    \node at (0.5,1.6) {\huge $\vdots$};
    \end{scope}
    
    \begin{scope}[shift={(2,-10.5)}, transform shape]
    \node at (-3.8,0) {\huge $\bm{f}_\mathrm{bot}(1)$};
    \node[circle, draw, fill = black, minimum size=8, inner sep=0] (C1) at (0,0) {};
    \node[circle, draw, minimum size=8, inner sep=0pt] (C2) at (0.5,0) {};
    \node[circle, draw, minimum size=8, inner sep=0pt] (C3) at (1.0,0) {};
    \node[circle, draw, minimum size=8, inner sep=0] (C4) at (1.5,0) {};
    \node at (2,0) {$\cdots$};
    \node[circle, draw, minimum size=8, inner sep=0] (C5) at (2.5,0) {};
    \node[circle, draw, fill = black,minimum size=8, inner sep=0] (C6) at (3,0) {};
    \node[circle, draw, fill = black,minimum size=8, inner sep=0] (C7) at (3.5,0) {};
    \node at (1,1.1) {\huge $\vdots$};
    \draw[decorate, decoration={brace, amplitude=5pt, mirror}] 
    (3.7,-0.3) -- (3.7,3.5) node[midway, right=2pt]{};
    \node at (5,3.2/2) {\huge $2k$};
    \end{scope}
    
    \begin{scope}[shift={(-0.4,-12.5)}, transform shape]
    \draw[fill=SkyBlue, opacity=0.8] 
    (0,0,0) -- (1,0,0) -- (1,1,0) -- (0,1,0) -- cycle;
    \end{scope}
    
    \begin{scope}[shift={(2,-12)}, transform shape]
    \node at (-3.5,0) {\huge $\bm{f}^{(m)}_\mathrm{bot}$};
    \node[circle, draw, fill = black, minimum size=8, inner sep=0] (C1) at (0,0) {};
    \node[circle, draw, minimum size=8, inner sep=0pt] (C2) at (0.5,0) {};
    \node[circle, draw, minimum size=8, inner sep=0pt] (C3) at (1.0,0) {};
    \node[circle, draw, minimum size=8, inner sep=0] (C4) at (1.5,0) {};
    \node at (2,0) {$\cdots$};
    \node[circle, draw, minimum size=8, inner sep=0] (C5) at (2.5,0) {};
    \node[circle, draw, fill = black,minimum size=8, inner sep=0] (C6) at (3,0) {};
    \node[circle, draw, fill = black,minimum size=8, inner sep=0] (C7) at (3.5,0) {};
    \end{scope}
    \end{tikzpicture}
    \end{subfigure}
    \caption{Schematic illustration of the concatenation procedure, which joins two $3$-faces, $\bm{f}^{(m)}_\mathrm{bot}$ and $\bm{f}^{(1)}_\mathrm{top}$, via some $4$-cubes. 
    The filled circles represent the elements of $\mathrm{HalfInt}$ associated with a given face or cube, 
    which indicate the orientation. 
    The left hand side denotes the case where $1 \notin \mathrm{HalfInt}(\bm{f}^{(m)}_\mathrm{bot}), \mathrm{HalfInt}(\bm{f}^{(1)}_\mathrm{top})$, while the  
    right hand side denotes the case where $1 \in \mathrm{HalfInt}(\bm{f}^{(m)}_\mathrm{bot})$ but 
    $1 \notin \mathrm{HalfInt}(\bm{f}^{(1)}_\mathrm{top})$.
    }
    \label{fig:concatenation}
\end{figure}

    For SAMs, this concatenation procedure is guaranteed to add the number of boundary components. Boundaries of hypercubes have no boundary. Therefore, adding them does not extend the boundary.
    
    In the case where $1 \not \in \mathrm{HalfInt}(\bm{f}^{(m)}_\mathrm{bot})$ and or $1 \not \in \mathrm{HalfInt}(\bm{f}^{(1)}_\mathrm{top})$, where we added individual hyperfaces, the boundary is extended. However, this is always followed by a layer of added hypercubes. This prevents the boundary of the top component from merging with the boundary of the bottom component.

    In the case of SOMs, the added hypercubes are made to only connect to $\bm{f}^{(1)}_\mathrm{top}$ or $\bm{f}^{(m)}_\mathrm{bot}$ on hyperedges that have more than 2 neighbours.

\end{proof}

\subsection{Lower bounds from `directed walk' manifold configurations} \label{sec:lower-bound}
Here, we will construct lower bounds from `directed walk' manifold configurations, which are a subset of all SAMs. This will lead to the following lemmas,

\begin{lemma} \label{lem:sam-h-1-lower-bound} The number of SAMs with hyperarea $n$ with one connected boundary component is lower bounded as
    \begin{align}
        c^{\mathrm{SAM}}_{(d,k),n}(h=1) \geq {d \choose k} \left(k(d-k+1)\right)^{n-1},
    \end{align}
\end{lemma}
Such configurations also allow for the following bounds for the number of closed SAMs,
\begin{lemma} \label{lem:sam-h-0-lower-bound}
    The number of SAMs with hyperarea $n$ with no boundaries is lower bounded as
    \begin{align}
    c^{\mathrm{SAM}}_{(d,k),n}(h=0) \geq {d \choose k + 1} \left((k+1)(d-k)\right)^{\frac{n-2(k+1)}{2k}}.
\end{align}
\end{lemma}

\begin{proof}[Proof of \Cref{lem:sam-h-1-lower-bound,lem:sam-h-0-lower-bound}]
    This comes from doing a `directed walk' configuration of manifolds. Start with one $k$-face, which there are $d \choose k$ orientations. Then, there are $k$ places which neighbours a $(k-1)$-edge which increase the coordinates. On each $(k-1)$-edge, there are $(d-k+1)$ ways of attaching a new $k$-face which increases the coordinates in the Cartesian representation. This is illustrated in \Cref{fig:directed-walk-config}.
    
    These configurations also add a constant number of boundary components, which allows us to come up with the following lower bound.
    
    Taking the boundary of the initial configuration of one $k$-face gives a closed $(k-1)$-manifold with $2k$ $(k-1)$-edges. Whenever we attach a new $k$-face in the above scheme, then take the boundary, we increase the number of $(k-1)$ edges by $(2k-2)$. Therefore the relationship between the number of $k$-faces, $n$, and number of $(k-1)$-edges, $m$, is $m=(2k-2)(n-1) + 2k$. Hence $c_m^{d,k-1,h=0} \geq {d \choose k} \left(k(d-k+1)\right)^{\frac{m-2k}{(2k-2)}}$. Sending $k-1 \rightarrow k$, and $m \rightarrow n$, we have retrieve the lower bound.
\end{proof}

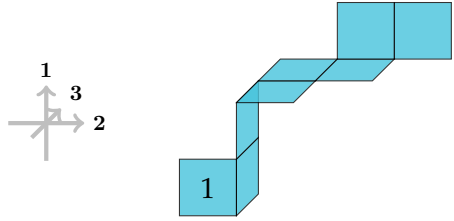
\begin{figure}[H]
    \centering
\begin{tikzpicture}[baseline={([yshift=-.5ex] current bounding box.center)}, scale=0.5]
    \draw [gray!50, ultra thick, ->] (0, -1, 0) -> (0, 1, 0) node[above, scale = 0.75, black] {$\bm{1}$};
    \draw [gray!50, ultra thick, ->] (-1, 0, 0) -> (1, 0, 0) node[right, scale = 0.75, black] {$\bm{2}$};
    \draw [gray!50, ultra thick, ->] (0, 0, 1) -> (0, 0, -1) node[above right, scale = 0.75, black] {$\bm{3}$};
    \end{tikzpicture}
    \quad \quad
    \begin{tikzpicture}[baseline={([yshift=-.5ex] current bounding box.center)}, scale=0.75]
    \draw[fill=SkyBlue, opacity=0.8]
        (0,0,0) -- (1,0,0) -- (1,1,0) -- (0,1,0) -- cycle;
    \draw[fill=SkyBlue, opacity=0.8]
        (1,0,0) -- (1,1,0) -- (1,1,-1) -- (1,0,-1) -- cycle;
    \draw[fill=SkyBlue, opacity=0.8]
        (1,1,0) -- (1,2,0) -- (1,2,-1) -- (1,1,-1) -- cycle;
    \draw[fill=SkyBlue, opacity=0.8]
        (1,2,0) -- (2,2,0) -- (2,2,-1) -- (1,2,-1) -- cycle;
    \draw[fill=SkyBlue, opacity=0.8]
        (1,2,-1) -- (2,2,-1) -- (2,2,-2) -- (1,2,-2) -- cycle;
    \draw[fill=SkyBlue, opacity=0.8]
        (2,2,-1) -- (3,2,-1) -- (3,2,-2) -- (2,2,-2) -- cycle;
    \draw[fill=SkyBlue, opacity=0.8]
        (3,2,-2) -- (2,2,-2) -- (2,3,-2) -- (3, 3,-2) -- cycle;
    \draw[fill=SkyBlue, opacity=0.8]
        (4,2,-2) -- (3,2,-2) -- (3,3,-2) -- (4, 3,-2) -- cycle;
    \node[] at (1/2, 1/2, 0) {1};
    \end{tikzpicture}
\caption{Illustration of a `directed walk' manifold configuration for $(d, k) = (3, 2)$. $k$-faces are only added in directions which increase coordinates.}
\label{fig:directed-walk-config}
\end{figure}

\section{Twig method for upper bounds for growth constants} \label{sec:twig-method}

In this section, we generalise the `twig' method originally developed in \cite{eden1961two,klarner1973procedure} to restricted surfaces on the cubic lattice. We first discuss the construction for SASs on the square lattice, also known as fixed polyominoes (without any prefixes), before generalising to the cubic lattice.

Using these methods we prove the following.

\begin{theorem}
    The connective constant for SASs on the cubic lattice verifies the upper bound
    \begin{align}
        \mu^\rm{SAS}_\cube \leq 17.11728.
    \end{align}
\end{theorem}

This can be compared to Monte Carlo estimate $\mu^\rm{SAS}_\cube \approx 12.798 \pm 0.018$ \cite{glaus1986monte,van1989self}.

\subsection{Twig method for polyominoes}
\label{Twig method}

We begin with a brief revision of the `twig' method presented in \cite{klarner1973procedure}, which is a method to find an upper bound for the growth constant of polyominoes (i.e. elements of $\mathrm{SAM}_{(2,2)}$). We refer readers to the original article for more details.

The idea of the twig method is that each polyomino can be associated with a planted tree (those with a starting node), in the way that was described in \Cref{sec:existence-and-upper-bound}, where the first face is the lexicographically smallest face, then connected to other faces in a tree structure. Twigs are particular subtrees of such a tree, such that any polyomino can be specified as a sequence of twigs. The set of twigs, $\mathrm{Twigs}(\ell)$, is specified by its level $\ell$. As $\ell$ is increased, we consider larger subtrees, each of which is a sequence of twigs of previous levels. The reason behind the improved upper bound with increasing $\ell$ is because not all sequences of twigs of previous levels is valid.

Each face of a twig has an assigned state, depending on whether it can be connected to a subsequent twig or not. A face of a twig can be
\begin{itemize}
    \item dead, if it cannot be connected to the \textit{next} twig. It is indicated by a black dot.
    \item It can be alive, if it can be connected to a subsequent twig. It is marked with a white circle.
    \item Unoccupied faces near the faces of the twig can be marked with a cross. This will be useful as a twig of level $\ell$ will be used to generate twigs of level $(\ell+1)$. The crosses denote faces that cannot be occupied at higher level of twigs generated from a given twig.
\end{itemize}
Geometrically, a twig in $\mathrm{Twigs}(\ell)$ is a set of connected faces with a fixed starting face that is dead (black), connected to living faces (if there are any), at Manhattan distance $\ell$ from the first face, via dead faces.

When two twigs are connected, a living face of the first twig is overlapped with the first face of the second twig (which is dead). In this way, we can connect twigs to form a given polyomino.
Note that that the first face of a polyominoes (the lexicographically smallest face) is guaranteed to not have any faces below it \Cref{fig:1 4 q-variables}. Therefore, one can assign a `zeroth face' to it, and the first face can be thought to have `entered' from the zeroth face (i.e. it is connected to this face). Subsequent faces always enter from previous faces. Therefore, the face below the first face of a twig can be marked with a cross, since it is guaranteed that there is a face there.
The zeroth face of a polyomino justifies the notion of `entering' face of a twig. This is the face that is below the first face of the twig which we mark with a cross (see \Cref{fig:twigs-level-1}). The entering face is successively occupied when connecting twigs together; therefore, we can guarantee that there is a face there, and can mark it with a cross. We can also have a notion of `entering' edge, which is the edge between the `entering' face and the first face of the twig, which is marked in red in \Cref{fig:twigs-level-1}.


Twigs are built recursively, starting from level one. $\mathrm{Twigs}(1)$ is denoted diagrammatically in \Cref{fig:twigs-level-1}, where the faces with white squares denote faces that can either be an occupied living face, or unoccupied crossed face. These are all nearest neighbours of the first face. To generate $\mathrm{Twigs}(\ell +1)$ from $\mathrm{Twigs}(\ell)$, for each twig in $\mathrm{Twigs}(\ell)$, we turn living faces into `new' dead faces or cross (unoccupied) and consider all possible such combinations. The recursive procedure to determine the twigs at the next level is summarised in \Cref{alg:twig_construction}.
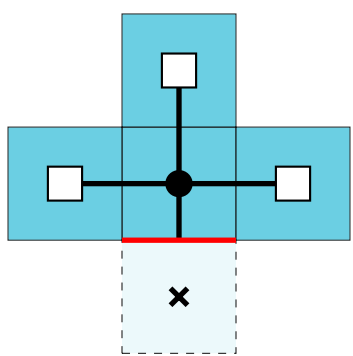
\begin{figure}[H]
    \centering
 \begin{tikzpicture}[scale=1.5]
     \draw[fill=SkyBlue, fill opacity=0.1, draw opacity=0.9, dashed] 
        (0,0) -- (1,0) -- (1,1) -- (0,1) -- cycle;
    \draw[line width=2pt] plot[mark=x, mark size=3pt] coordinates {(0.5,0.5)};

    \draw[fill=SkyBlue, opacity=0.8] 
        (0,1) -- (1,1) -- (1,2) -- (0,2) -- cycle;

    \fill[black] (0.5,1.5) circle[radius=0.12];

    \draw[fill=SkyBlue, opacity=0.8] 
        (1,1) -- (2,1) -- (2,2) -- (1,2) -- cycle;

    \draw[fill=SkyBlue, opacity=0.8] 
        (0,2) -- (1,2) -- (1,3) -- (0,3) -- cycle;

    \draw[fill=SkyBlue, opacity=0.8] 
        (0,1) -- (-1,1) -- (-1,2) -- (0,2) -- cycle;

    \draw[black, line width = 2pt] (0.5,1) -- (0.5,1.5);

    \draw[black, line width = 2pt] (0.5,1.5) -- (-0.5,1.5);
    \draw[black, line width = 2pt] (0.5,1.5) -- (1.5,1.5);
    \draw[black, line width = 2pt] (0.5,1.5) -- (0.5,2.5);

    \draw[red, line width=2pt] (0,1) -- (1,1);

    \draw[fill=white, draw=black, line width=0.8pt] 
        (1.35,1.35) -- (1.65,1.35) -- (1.65,1.65) -- (1.35,1.65) -- cycle;

    \begin{scope}[shift={(-1,1)}]
        \draw[fill=white, draw=black, line width=0.8pt] 
        (1.35,1.35) -- (1.65,1.35) -- (1.65,1.65) -- (1.35,1.65) -- cycle;
    \end{scope}

    \begin{scope}[shift={(-2,0)}]
        \draw[fill=white, draw=black, line width=0.8pt] 
        (1.35,1.35) -- (1.65,1.35) -- (1.65,1.65) -- (1.35,1.65) -- cycle;
    \end{scope}

\end{tikzpicture}
    \caption{Compact representation of set of twigs at level~1, $\mathrm{Twigs}(\ell=1)$, for polyominoes on the square lattice. Faces with white squares denote faces which could be living faces (which would be denoted with white circles), or unoccupied crossed faces. The cross represents unoccupied faces which higher level twigs generated from this twig cannot occupy. The face with a cross is the `entering face' and the red edge is the `entering edge'. 
}
    \label{fig:twigs-level-1}
\end{figure}
\begin{algorithm}[H]
\caption{Recursive construction of higher level twigs}
\label{alg:twig_construction}
\begin{enumerate}
    \item Start with the set of twigs from the previous level, $\mathrm{Twigs}(\ell-1)$.
    
    \item Identify all faces marked with a white circle (alive).
    
    \item Convert each of these alive faces into new dead faces by replacing the white circle with a black dot. Store this set of \say{partial} new twigs.
    
    \item For each new dead face (white circles turned into black dots in Step~3) of the \say{partial} twig, find all nearest-neighbor faces, excluding:
        \begin{itemize}
            \item Any faces that have already been marked as dead (either cross or black dot).
            \item Neighbors of faces previously marked as dead.
        \end{itemize}
    
    \item For each set of valid neighboring faces, consider all possible combinations in which each neighboring face is:
    \begin{itemize}
        \item Alive (represented by a white circle), or
        \item Not part of the polyomino (represented by a cross).
    \end{itemize}
    Each new configuration of faces defines a new twig.
    
    \item Collect all such twigs from Step~5, along with the twigs from all previous levels that contain no white circles (and hence only black dots), to form the full set of twigs at the next level, $\mathrm{Twigs}(\ell)$.
\end{enumerate}
\end{algorithm}
Having established a procedure for generating the twigs at each level, we now turn to the task of extracting an upper bound for the growth constant. We note that a polyomino is a sequence of twigs such that the total number of black dots across all twigs equals the total number of cells in the polyomino, $n$. By a “cell” we mean either a face with a black dot or an alive face (white circle).
This follows from the fact that, given a white face of a twig, we can connect another twig to it through its first face (black dot). Therefore, for each white face in a twig, there exists a corresponding face with a black dot in another twig that is connected to it.

We assign a monomial $x^{N_c-1}y^{N_{\mathrm{b}}}$ to each twig where $N_c$ is the number of cell present in the twig (excluding crosses) and $N_{\mathrm{b}}$ is the number of black dots. We define the monomial corresponding to a sequence of twigs as the product of each monomials for each twigs. Define
\begin{equation}
    N^{\mathrm{tot}}_{\mathrm{c}} = \sum_{i=1}^{N_{\mathrm{twigs}}}N^{(i)}_{\mathrm{c}} \hspace{10pt} N^{\mathrm{tot}}_{\mathrm{b}} = \sum_{i=1}^{N_{\mathrm{twigs}}}N^{(i)}_{\mathrm{b}} \hspace{1pt},
\end{equation}
where $N^{(i)}_{\mathrm{b}}$ and $N^{(i)}_{\mathrm{c}}$ is the number of black dots and cells contained in the $i$-th twigs and $N_{\mathrm{twigs}}$ is the total number of twigs. Since for each white faces there is a twig connected to it, we have the equality $N_{\mathrm{twigs}} = N^{\mathrm{tot}}_{w} +1$, where the additional $1$ accounts for the initial twig. Therefore the monomial for the sequence of twigs is given by
\begin{equation}
    x^{N^{\mathrm{tot}}_{\mathrm{c}}-N_{\mathrm{twigs}}}y^{N^{\mathrm{tot}}_{\mathrm{b}}} = x^{N^{\mathrm{tot}}_{\mathrm{b}}+N^{\mathrm{tot}}_{w}-N_{\mathrm{twigs}}}y^{N^{\mathrm{tot}}_{\mathrm{b}}} = x^{N^{\mathrm{tot}}_{\mathrm{b}}-1}y^{N^{\mathrm{tot}}_{\mathrm{b}}} = x^{n-1}y^{n} \hspace{1pt}.
\end{equation}

As a consequence of this analysis, polyominoes of size~$n$ are represented by sequences of twigs whose associated monomials are of the form $x^{n-1} y^n$. This is because, for every polyomino $X$ with number of faces $n = n(X)$, we can assign an injection from $X$ to a sequence of twigs. 

However, there is a caveat: not all sequences of twigs with monomials of the form $x^{n-1} y^n$ correspond to valid polyominoes. This is because some sequences of twigs may contain overlapping faces. Therefore, we have constructed an injection from the set of polyominoes of size~$n$ into the set of suitable sequences of twigs—those whose associated monomials are of the form $x^{n-1} y^n$. Consequently, the number of such sequences provides an upper bound on the number of polyominoes of size~$n$. Given the set of twigs at level $\ell$, we can count all the sequences of the form $x^{n-1}y^{n}$ as follows. First, we generate the polynomial  
\begin{equation}
    p_\ell(x, y) = \sum_{i=1}^{N_{\mathrm{twigs}}(\ell)} x^{N^{(i)}_{\mathrm{c}} - 1} y^{N^{(i)}_{\mathrm{b}}} \hspace{1pt},
\end{equation}
which is the sum of the monomials corresponding to the twigs generated at level~$\ell$.
Then, we use the geometric series expansion:
\begin{equation}
\label{eqn: function d=2}
    x \sum_{n \geq 0} \left(p_\ell(x, y)\right)^n = \frac{x}{1-p_\ell(x, y)}=\sum_{m_x, m_y} c_{m_x, m_y}(h) \, x^{m_x} y^{m_y} \hspace{1pt}.
\end{equation}
The diagonal coefficient \( c_{n,n}(\ell) \) corresponds to the number of sequences of twigs of the form \( x^{n-1} y^n \) generated using twigs at level~$\ell$. From this reasoning, we obtain the following upper bound for the growth constant:
\begin{equation}
    \mu^{\mathrm{SAS}}_{\square} \leq \lim_{n \to \infty} c_{n n}(\ell)^{1/n} \hspace{1pt},
\end{equation}
which becomes progressively sharper as the level $\ell$ increases. The task of finding an upper-bound for the growth constant is reduced to determining the radius of convergence of the diagonal function:
\begin{equation}
f_{\ell}(z) = \sum_{n} c_{n, n}(\ell) \, z^{n} \hspace{1pt},
\end{equation}
obtained from the function $x/(1-p_{\ell}(x, y))$. For a generic vale of $\ell$, this task can be accomplished numerically as we explain in Appendix~\ref{apdx: diagonal function}. However, the specific value $\ell=1$ can be solved analytically. In this case, the twigs are those in Fig.~\ref{fig:twigs-level-1} and they corresponds to the following polynomial 
\begin{equation}
    p_1(x,y) = y(1+x)^3 \hspace{1pt}.
\end{equation}
Which corresponds to the following geometric series:
\begin{equation}
    f_1(x,y) = \frac{x}{1-y(1+x)^3} = x\sum_{n}y^n(1+x)^{3n} = \sum_{nm}\binom{3n}{m}y^nx^{m+1} \hspace{1pt}.
\end{equation}
Hence, in this case the upper bound is
\begin{equation}
\label{eqn: Bound SAS square}
\mu^{\rm{SAS}}_{\square} \leq \lim_{n \to \infty}\binom{3n}{n-1}^{1/n} = \frac{27}{4} \hspace{1pt},  
\end{equation}
which is what was obtained in \cite{eden1961two}.
\subsection{Twig method for self-avoiding surfaces on the cubic lattice}
The previous method can be generalized to three dimensions (i.e., for $\mathrm{SAM}_{(3,2)}$), but with an important caveats. In $d=2$, given a face and an edge of that face which is known to be doubly occupied (i.e., shared by two faces), the orientation of the connecting face is fully determined. As a result, we can uniquely assign an entering face attached to the first face of the given twig. In contrast, in $d=3$, given a face and a doubly occupied edge, there are multiple possible choices for the orientation of the connecting face. Consequently, in $d=3$ we cannot assign a unique zeroth face (see Fig.~\ref{fig:1st and 0th face}). We remind the reader that the zeroth face of a twig in $d=2$ acts as a dead face. Consequently, the zeroth face restricts the number of possible twigs, since faces containing either white circles or black dots cannot neighbor it, except for the first face (see Step~4 of the Algorithm~ \ref{alg:twig_construction}).

While in $d=3$ we cannot uniquely determine the orientation of the entering face, we can still identify an entering edge. Therefore, when constructing twigs in $d=3$, no cells may be added that share the entering edge, apart from the first face.

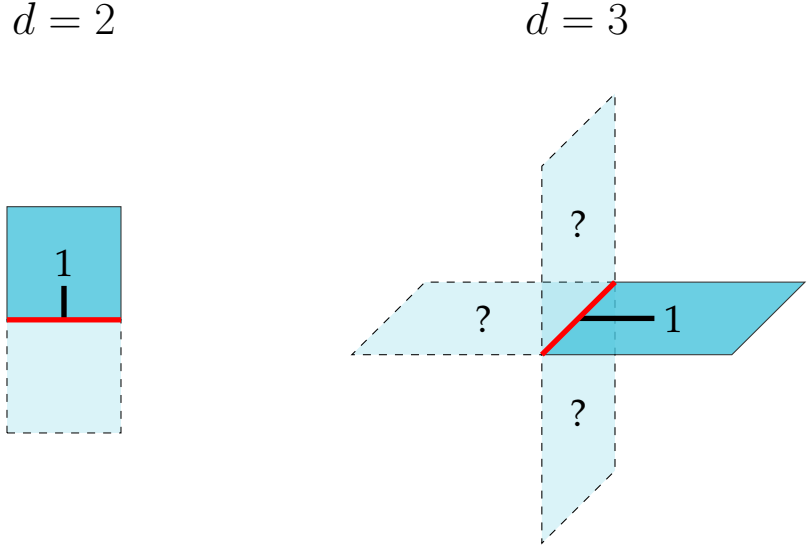
\begin{figure}[H]
    \centering
\begin{tikzpicture}
    \node at (0.75,5.5) {\LARGE \textbf{$d=2$}};
    \node at (7.5,5.5) {\LARGE \textbf{$d=3$}};
    \begin{scope}[scale=1.5]
    \draw[fill=SkyBlue, fill opacity=0.2, draw opacity=0.95, dashed] 
        (0,0) -- (1,0) -- (1,1) -- (0,1) -- cycle;
    \node at (0.5,0.5) {};

    \draw[fill=SkyBlue, opacity=0.8] 
        (0,1) -- (1,1) -- (1,2) -- (0,2) -- cycle;
    \node at (0.5,1.5) {\Large 1};

    \draw[black, line width=2pt] (0.5,1.3) -- (0.5,1.0);
    
    \draw[red, line width=2pt] (0,1) -- (1,1);
    \end{scope}
    
    \begin{scope}[shift={(8,2)}, scale= 2.5]
    \draw[fill=SkyBlue, fill opacity=0.2, dashed]
        (0,0,0) -- (-1,0,0) -- (-1,0,1) -- (0,0,1) -- cycle;
    
    \draw[fill=SkyBlue, fill opacity=0.2, dashed]
        (0,0,0) -- (0,0,1) -- (0,1,1) -- (0,1,0) -- cycle;

    \draw[fill=SkyBlue, fill opacity=0.2, dashed]
        (0,0,0) -- (0,0,1) -- (0,-1,1) -- (0,-1,0) -- cycle;
        
    \draw[fill=SkyBlue, opacity=0.8]
        (0,0,0) -- (1,0,0) -- (1,0,1) -- (0,0,1) -- cycle;

    \draw[black, line width=2pt] (0.4, 0, 0.5) -- (0, 0, 0.5);
    
    \draw[red, line width=2pt] (0,0,1) -- (0,0,0);
    
    \node at (0.5,0,0.5) {\Large 1};
    \node at (-0.5,0,0.5) {\Large ?};
    \node at (0,0.5,0.5) {\Large ?};
    \node at (0,-0.5,0.5) {\Large ?};
    \end{scope}

\end{tikzpicture}
\caption{First face, entering faces, and entering edge of a twig in $d=2$ and $d=3$. In $d=2$, given the first face (indicated with a $1$) and the entering edge (red edge), we can assign a unique entering face (dashed). In $d=3$, this is no longer possible.}
\label{fig:1st and 0th face}
\end{figure}

Another important difference with the case in $d=2$ is that the number of twigs grows much faster as the level of the recursion is increases. 
This is a consequence of the fact that there are more self-avoiding surfaces in $d=3$ due to the extra orientations.
In fact, using \Cref{thm:sam-som-upper-bound} with $d=3$ and $k=2$, yields an upper-bound
\begin{equation}
    \mu^{\rm{SAS}}_{\cube} \leq \frac{87}{4} = 20.25 \hspace{1pt},
\end{equation}
which is $3$ times larger than the corresponding upper-bound in $d=2$ given by Eq.~\eqref{eqn: Bound SAS square}. This is the first computational limitation, which hinders the algorithm to find the set of twigs. The second limitation comes from constructing the polynomial once the set of twigs are obtained. 
In \( d = 2 \), we can efficiently determine the polynomial associated with a set of twigs. Instead of storing each twig explicitly, we store the configuration of black dots and white squares, where each white square represents a choice between a white circle (alive) and a cross (excluded); see Fig.~\ref{fig:twigs-level-1}.

For such a set of twigs, the corresponding polynomial is given by
\[
y^{N_{\mathrm{b}}}(1 + x)^{N_{\mathrm{s}}},
\]
where \( N_{\mathrm{b}} \) is the number of black dots and \( N_{\mathrm{s}} \) is the number of white squares. This form arises because each white square independently contributes either a factor of \( x \) (if chosen as a white circle) or \( 1 \) (if excluded). Therefore, the factor \( (1 + x)^{N_{\mathrm{s}}} \) accounts for all possible combinations of alive and excluded neighbors. For example, the set of twigs shown in Fig.~\ref{fig:twigs-level-1} corresponds to the polynomial \( y(1 + x)^3 \).

In contrast, in $d=3$ we cannot determine the polynomial so easily. Whilst we can still store the twigs as collections of black dots and white squares, the white squares cannot be turned on independently of each other. To address this difficulty, our strategy is to consider the edges on the boundary between the black dots and the white squares. Each of these edge can then activate a white square, meaning that the white square can be turned into a white circle. However, different edges on the boundary can activate the same white square as show in Fig.~\ref{fig: not free edge}.
If this is the case, we say that the edges are not free. They are not free because if one edge turns on the shared face, then the second edge is locked as it cannot turn on any other faces (refer to Fig.~\ref{fig: not free edge}).

The optimal way to determine the polynomials from a given set of black dots and white squares is to first find the list of edges on the boundary \( E \), and then determine which of these edges are free and which are not.
The contribution to the polynomial from each free edge is simply a factor
\begin{equation}
    (1 + x)^{N_{1}} (1 + 2x)^{N_{2}} (1 + 3x)^{N_{3}}
\end{equation}
where \( N_{i} \) is the number free edges connected to $i$ white squares.

In contrast, the contribution from non-free edges is more cumbersome to determine, as we need to try all combinations in which the edges can be turned on, and then run a check to see which of these configurations are allowed, so that no edge is shared by more than two faces, preserving the self-avoiding condition. 

There is another caveat in $d=3$ compared to $d=2$, which concerns the choice of the first face of an element of $\mathrm{SAM}_{(3,2)}$. In $d=2$, as previously discussed, we can identify the first face of a given surface in $\mathrm{SAM}_{(2,2)}$ with the first face of a suitable twig from the list of twigs. The first face of a surface is the one with smallest lexicographic order (bottommost–leftmost face). Therefore, the first face of a surface in $\mathrm{SAM}_{(2,2)}$ is connected to at most $2$ other faces.
However, this property fails in $d=3$: the first face (i.e., the one with smallest lexicographic order, such as the bottommost–leftmost–frontmost face) may be connected to four other faces instead of two (see Fig.~\ref{fig:1 4 q-variables}).

\begin{figure}[H]
    \centering
\begin{tikzpicture}
    \node at (0.75,5.5) {\LARGE \textbf{$d=2$}};
    \node at (9.0,5.5) {\LARGE \textbf{$d=3$}};
    \begin{scope}[scale=1.5]
     \draw[fill=SkyBlue, fill opacity=0.2, dashed] 
        (0,0) -- (1,0) -- (1,1) -- (0,1) -- cycle;
    \node at (0.5,0.5) {\Large 0};

    \draw[fill=SkyBlue, opacity=0.8] 
        (0,1) -- (1,1) -- (1,2) -- (0,2) -- cycle;
    \node at (0.5,1.5) {\Large 1};

    \draw[fill=SkyBlue, opacity=0.8] 
        (1,1) -- (2,1) -- (2,2) -- (1,2) -- cycle;

    \draw[fill=SkyBlue, opacity=0.8] 
        (0,2) -- (1,2) -- (1,3) -- (0,3) -- cycle;
        
    \draw[black, line width=2pt] (1,1) -- (1,2) -- (0,2);
    
    \draw[red, line width=2pt] (0,1) -- (1,1);
\end{scope}
    
    \begin{scope}[shift={(8,2)}, scale= 2.5]

\draw[fill=SkyBlue, opacity=0.8] 
    (0,0,0) -- (1,0,0) -- (1,0,1) -- (0,0,1) -- cycle;

\draw[fill=SkyBlue, opacity=0.8]
    (0,0,0) -- (1,0,0) -- (1,1,0) -- (0,1,0) -- cycle;

\draw[fill=SkyBlue, opacity=0.8]
    (0,0,0) -- (0,1,0) -- (0,1,1) -- (0,0,1) -- cycle;

\draw[fill=SkyBlue, opacity=0.8]
    (0,0,1) -- (1,0,1) -- (1,1,1) -- (0,1,1) -- cycle;

\draw[fill=SkyBlue, opacity=0.8]
    (1,0,0) -- (1,1,0) -- (1,1,1) -- (1,0,1) -- cycle;

\node at (0.5,0,0.5) {\Large 1};
\draw[black, line width = 1.8pt]
    (0.02,0,0) -- (0.98,0,0) -- (0.98,0,1) -- (0.02,0,1) -- cycle;

\end{scope}

\end{tikzpicture}
\caption{In $d=2$, the first face of a surface (bottomost-leftmost face) is at most connected to two surfaces (we do not count the zeroth face). In $d=3$, the first face of a surface (bottommost-leftmost-frontmost face) could be connected to $4$ faces (thick black lines).}
\label{fig:1 4 q-variables}
\end{figure}
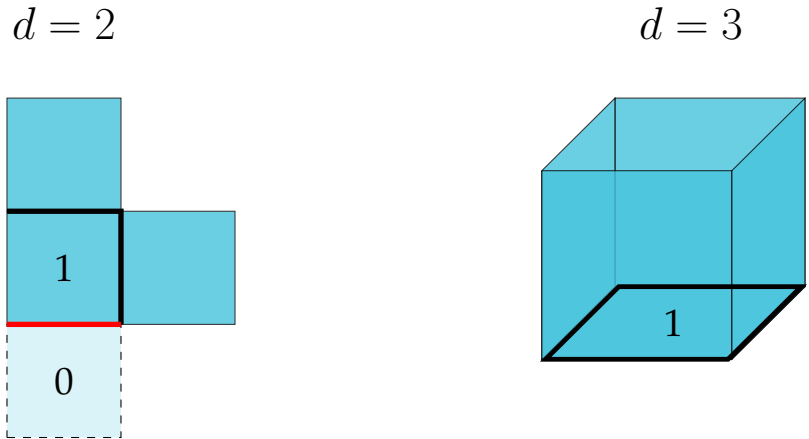
Because a zeroth face cannot be assigned in $d=3$, the function that generates sequences of twigs must be modified. In particular, we cannot use Eq.~\eqref{eqn: function d=2}; the numerator has to enforce that the first twig has no zeroth face. The correct generating function is
\begin{equation}
  \frac{xy(1+3x)^4}{1-p_{\ell}(x,y)} \, .
\end{equation}
The prefactor $xy(1+3x)^4$ encodes the constraints on the first face: it can be adjacent to at most four faces, giving the fourth power, and for each potential adjacency we either do not attach a face (corresponding to \say{$1$}) or we attach a face with any of three possible orientations (corresponding to \say{$3x$}). However, this change in the numerator does not ultimately change the resulting upper bound.

Using these considerations, we run the algorithm to find improved upper bounds, which is tabulated in \Cref{tab:upper-bounds-sas cubic}.
\begin{table}[H]
    \centering
\begin{tabular}{r r}
    Twig level & Upper bound for $\mu_\cube^\rm{SAS}$ \\
    \hline
    1 & 20.25000 \\
    2 & 18.23447 \\
    3 & 17.11728
    \end{tabular}
    \caption{Rigorous upper bounds for SASs on the cubic lattice, obtained using the Twig method.}
    \label{tab:upper-bounds-sas cubic}
\end{table}

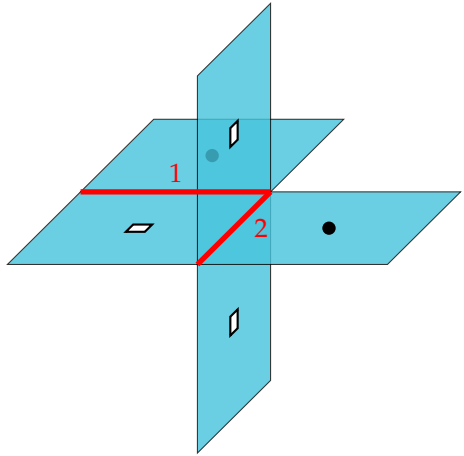
\begin{figure}[H]
    \centering
\begin{tikzpicture}[scale=2.5]


    \draw[fill=SkyBlue, opacity=0.8]
        (-1,0,0) -- (0,0,0) -- (0,0,-1) -- (-1,0,-1) -- cycle;
    \fill (-0.5, 0, -1+0.5) circle (1pt);
    \draw[fill=SkyBlue, opacity=0.8]
        (0,0,0) -- (-1,0,0) -- (-1,0,1) -- (0,0,1) -- cycle;
    \filldraw[fill=white, draw=black, thick]
    (-0.45, 0,0.45) -- (-0.55, 0,0.45) -- (-0.55, 0,0.55) -- (-0.45, 0,0.55) -- cycle;
    
    \draw[fill=SkyBlue, opacity=0.8]
        (0,0,0) -- (0,0,1) -- (0,1,1) -- (0,1,0) -- cycle;
    \fill (0,0.5,0.5) circle (0.4pt); 
    \filldraw[fill=white, draw=black, thick]
    (0, 0.45,0.45) -- (0, 0.55,0.45) -- (0, 0.55,0.55) -- (0, 0.45,0.55) -- cycle;

    \draw[fill=SkyBlue, opacity=0.8]
        (0,0,0) -- (0,0,1) -- (0,-1,1) -- (0,-1,0) -- cycle;
    \filldraw[fill=white, draw=black, thick]
    (0, -0.45,0.45) -- (0, -0.55,0.45) -- (0, -0.55,0.55) -- (0, -0.45,0.55) -- cycle; 

    \draw[fill=SkyBlue, opacity=0.8]
        (0,0,0) -- (1,0,0) -- (1,0,1) -- (0,0,1) -- cycle;
    \fill (0.5, 0, 0.5) circle (1pt);

    \draw[red, line width=2pt] (-1,0,0)--node[midway, xshift=-7pt,yshift = 7pt,  right]{\small 1} (0,0,0);
    \draw[red, line width=2pt] (0,0,1) -- node[midway, xshift=3pt, right]{\small 2} (0,0,0);
\end{tikzpicture}
\caption{Edges $1$ and $2$ lie on the boundary between the black dots and the white squares, and they are not free.
In fact, if edge $1$ turns on the shared white square, then edge $2$ is not allowed to turn on any of the other white squares.}
\label{fig: not free edge}
\end{figure}

\section{Conclusion and outlook} \label{sec:conclusion}
In this work, we introduced various classes of restricted walks, such as self-osculating walks (SOWs) and osculating domain wall walks (ODWs). By generalising the `automata' method of \cite{ponitz2000improved}, we obtained upper bounds for the connective constants of these walks:
\[
\mu^\rm{SOW}_\square \leq 2.73911, \qquad 
\mu^\rm{SOW}_\triangle \leq 4.44931, \qquad
\mu^\rm{ODW}_\triangle \leq 4.44867.
\]
We also introduced the analogue of restricted walks in higher dimensions: $(d,k)$-restricted surfaces. For example, we introduced self-osculating surfaces (SOSs) for $k=2$, and self-osculating manifolds (SOMs) for $k=d-1$. Using a concatenation-based argument generalising the procedure in \cite{van1989self}, we proved the existence of the growth constant for such restricted manifolds. We also proved their upper and lower bounds.

Finally, we applied the twig method, as introduced in \cite{eden1961two,klarner1973procedure}, to obtain an upper bound for the growth constant of self-avoiding surfaces (SASs) on the cubic lattice:
\[
\mu^\rm{SAS}_{\cube} \leq 17.11728.
\]

Many questions remain. The first is developing or modifying existing methods to obtain improved lower bounds of these models. 

The second pertains to the twig method. For polyominoes, it is known that an improved set of twigs, called `L-contexts', can be used for polyominoes, or any $(d,d)$-fixed polyominoids (equivalent to $(d,d)$-SAMs or $(d,d)$-SOMs) \cite{barequet2022improved}. The crucial information here is that the orientations of $d$-cubes cannot change, and therefore one can assume further restrictions on possible occupations of $d$-cubes. For general $(d,k)$-SAMs, the orientations can change, which prevented us from applying it for improved upper bounds. An improved set of twigs for general $(d,k)$, or even $(d,d-1)$ would lead to drastically improved upper bounds. Adaptation of the twig method for SOSs and fixed polyominoids (XDs) is another remaining task. The trouble here is that the notion of `distance' from the first face is more difficult to define, which prevented us in straightforwardly developing an efficient algorithm to generate the twigs for them. 

The third has to do with upper bounds for closed surfaces or hypersurfaces. Can a method be developed to systematically obtain improved upper bounds for closed SASs? Current upper bound for closed SASs in $d=3$ is $\mu^\mathrm{SAS}_\cube(h=0) \leq 3$. However, Monte Carlo estimates suggest $\mu^\mathrm{SAS}_\cube \approx 1.733 \pm 0.006$ \cite{glaus1986monte}. Closed SOSs can also be defined such that if odd number of faces neighbour an edge, then there is no boundary there, and if an even number, then there is a boundary. Whether it is possible to upper bound the connective constant of closed SOSs remains unexplored.

Other questions regarding the restricted manifold models but unrelated the their growth constants is the critical exponents, which are related to polynomial corrections to $c_n$, the number of configurations with hyperarea $n$. It would be interesting to see if they fall in the same universality class as self-avoiding manifolds.


\section*{Acknowledgments}
SWPK would like to thank Max McGinley for useful discussions, especially for motivating osculating domain walls. We acknowledge the use of \cite{create2025} for our numerical work.

\appendix

\section{Upper bounds for modified self-avoiding walks on the square lattice using the automata method} \label{apdx:modified-walk-square}
We can use the automata method to obtain an upper bound for a modified version of the self-avoiding walk on the square lattice. In this model, paths are not allowed to self-intersect, and we additionally impose the constraint that no two consecutive steps may be taken in the same direction. This restriction leaves us with only ``L-shaped'' directions, as depicted in Fig.~\ref{fig: Vertices modified SAW}.
\begin{figure}[H]
    \centering
\begin{center}
\begin{tikzpicture}[baseline={([yshift=-.5ex] current bounding box.center)}, scale=0.5]
    \draw [gray!50, ultra thick] (0, -1) to (0, 1);
    \draw [gray!50, ultra thick] (-1, 0) to (1, 0);
    \draw [ultra thick, blue!75] (0, 1) to (0, 0) to (-1,0);
    \draw[dashed] (1,1) -- (-1,1) -- (-1,-1) -- (1,-1) -- (1, 1);
\begin{scope}[shift={(5,0)}]
    \draw [gray!50, ultra thick] (0, -1) to (0, 1);
    \draw [gray!50, ultra thick] (-1, 0) to (1, 0);
    \draw [ultra thick, blue!75] (0, -1) to (0, 0) to (-1,0);
    \draw[dashed] (1,1) -- (-1,1) -- (-1,-1) -- (1,-1) -- (1, 1);
\end{scope}
\begin{scope}[shift={(10,0)}]
    \draw [gray!50, ultra thick] (0, -1) to (0, 1);
    \draw [gray!50, ultra thick] (-1, 0) to (1, 0);
    \draw [ultra thick, blue!75] (0, -1) to (0, 0) to (1,0);
    \draw[dashed] (1,1) -- (-1,1) -- (-1,-1) -- (1,-1) -- (1, 1);
\end{scope}
\begin{scope}[shift={(15,0)}]
    \draw [gray!50, ultra thick] (0, -1) to (0, 1);
    \draw [gray!50, ultra thick] (-1, 0) to (1, 0);
    \draw [ultra thick, blue!75] (0, 1) to (0, 0) to (1,0);
    \draw[dashed] (1,1) -- (-1,1) -- (-1,-1) -- (1,-1) -- (1, 1);
\end{scope}
\end{tikzpicture}  
\end{center}
    \caption{Bulk vertex configurations for the modified self-avoiding walks on the square lattice}
    \label{fig: Vertices modified SAW}
\end{figure}
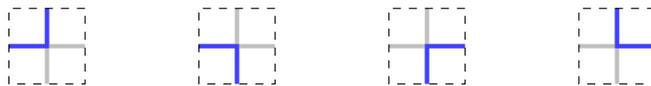
There is a trivial upper bound that can be obtained from considering the fact that once path goes through an unvisited vertex, there are only $2$ ways in which the path can leave the vertex, thus $\mu^\mathrm{L}_{\square} \leq 2$. We can then refine this upper bound using the automata method which yields $\mu^{\rm{L}}_{\square} \leq 1.57790$, the full are shown in Table~\ref{tab:upper-bounds-modified-saw}.
\begin{table}[H]
    \centering
    \begin{tabular}{r r}
        Loop size \(n\) & Upper bound for \(\mu_{\square}^{\rm L}\) \\
        \hline
        4  & 1.61804 \\
        8  & 1.61804 \\
        12 & 1.60135 \\
        16 & 1.59511 \\
        20 & 1.59021 \\
        24 & 1.58672 \\
        28 & 1.58408 \\
        32 & 1.58202 \\
        36 & 1.58037 \\
        40 & 1.57902 \\
        44 & 1.57790 \\
    \end{tabular}
    \caption{Upper bounds for \(\mu_{\square}^{\rm L}\) obtained using the automata method for the modified self-avoiding walk on the square lattice.}
    \label{tab:upper-bounds-modified-saw}
\end{table}

\section{Radius of convergence of the diagonal of a function $g(x, y)$}
\label{apdx: diagonal function}
In this appendix, we briefly explain how to extract the radius of convergence of the diagonal function following \cite{klarner1973procedure}. Define a bivariate function
\begin{equation}
\label{eqn: f(x,y)}
    g(x,y) = \sum_{nm}g_{nm}x^ny^m \hspace{1pt},
\end{equation}
and its diagonal part
\begin{equation}
    g_d(z) = \sum_{n}g_{nn}z^n \hspace{1pt}.
\end{equation}
Let $R_x,R_y$ be the two radii of convergence of $g(x,y)$ meaning that the summation in Eq.~\eqref{eqn: f(x,y)} converges absolutely for $|x| <R_x$ and $|y|<R_y$.
The aim is to find an expression for the radius of convergence of the diagonal part $R^{-1}_d = \lim_{n\to \infty}g_{nn}^{1/n}$. Notice that $g_d(z)$ has the following contour integral representation
\begin{equation}
\label{eqn: contour integral}
g_d(z) = \frac{1}{2\pi \mathrm{i}}\oint_{\Gamma}dss^{-1}g(s,zs^{-1}) \hspace{1pt},    
\end{equation}
where $\Gamma$ is a contour in a region where $g(s,zs^{-1})$ is analytic, $|s|<R_x \cup |zs^{-1}|<R_y$, which implies $|z/R_y|<|s|<R_x$.
In our case the function $g(x,y)$ has the form 
\begin{equation}
    g(x,y) = \frac{q(x,y)}{p(x,y)} \hspace{1pt},
\end{equation}
where $p(x,y)$ and $q(x,y)$ are polynomials.
The polynomial $q(x,y)$ depends on wether we are dealing with twigs in $d=2$ or $d = 3$ to construct $\mathrm{SAM}_{(d,2)}$:
\begin{equation}
q(x,y) =
\begin{cases}
    x & d=2 \\
    x(1+3x)^4 y & d = 3 \;.
\end{cases}
\end{equation}
This reflects a subtlety about the 1st face of a surface. 
The polynomial is of the form
\begin{equation}
    p(x,y) = \sum_{n=0}^{N_\mathrm{x}}\sum_{m=0}^{N_\mathrm{y}}a_{nm}x^ny^m = \sum_{m=0}^{N_\mathrm{y}}\sum_{s=-m}^{N_\mathrm{x}-m}a_{s+m,m}x^s(xy)^m = \sum_{s=-N_\mathrm{y}}^{N_\mathrm{x}}\sum_{m= \max(-s,0)}^{\min(N_\mathrm{x}-s,N_\mathrm{y})}a_{s+m,m}x^s(xy)^m \hspace{1pt},
\end{equation}
where in the last equality we have swapped the summation over $s$ and $m$. It follows that the polynomial can be written as
\begin{equation}
    p(x,y) = \sum_{s=0}^{N_\mathrm{y}}x^{-s}P_{N_\mathrm{y}-s}(xy) + \sum_{s=1}^{N_\mathrm{x}}x^{s}P_{N_\mathrm{y}+s}(xy) \hspace{1pt},
\end{equation}
where $P_j(xy)$ is a polynomial of degree $j$. Therefore,
\begin{equation}
    p(s,zs^{-1}) = s^{-N_\mathrm{y}}\left[\sum_{j=0}^{N_\mathrm{y}}s^{-j+N_\mathrm{y}}P_{N_\mathrm{y}-j}(z) + \sum_{j=1}^{N_\mathrm{x}}s^{j+N_\mathrm{y}}P_{N_\mathrm{y}+j}(z)\right] = s^{-N_\mathrm{y}}\sum_{j=0}^{N_\mathrm{x}+N_\mathrm{y}}s^{j}P_{j}(z) \hspace{1pt}.
\end{equation}
In our case, the integrand of Eq.~\eqref{eqn: contour integral} is therefore given by
\begin{equation}
    s^{-1}g(s,zs^{-1}) = 
    \begin{cases}
    \frac{1}{ p(s,zs^{-1})} = \frac{s^{N_\mathrm{y}}}{\sum_{j=0}^{N_\mathrm{x}+N_\mathrm{y}}s^{j}P_j(z)} = \frac{s^{N_\mathrm{y}}}{P_{N_\mathrm{x}+N_\mathrm{y}}(z) \prod_{j=1}^{N_\mathrm{x}+N_\mathrm{y}}(s-r_j(z))} & d=2\\
    \frac{zs^{-1}(1+3s)^{4}}{ p(s,zs^{-1})} =  \frac{z(1+3s)^{4}s^{N_\mathrm{y}-1}}{\sum_{j=0}^{N_\mathrm{x}+N_\mathrm{y}}s^{j}P_j(z)} = \frac{z(1+3s)^{4}s^{N_\mathrm{y}-1}}{P_{N_\mathrm{x}+N_\mathrm{y}}(z) \prod_{j=1}^{N_\mathrm{x}+N_\mathrm{y}}(s-r_j(z))} & d=3
    \end{cases}
\end{equation}
To find the contour integral, we need to determine the residue of the expression. In both cases, the singularities arises due to the roots of the polynomial in the denominator; hence, we need to determine:
\begin{equation}
\label{eqn: residue}
    \operatorname{Res}_{s = r_j}\left[P_{N_\mathrm{x}+N_\mathrm{y}}(z)\prod_{k=1}^{N_\mathrm{x}+N_\mathrm{y}}(s-r_k(z))\right]^{-1} \hspace{1pt}.
\end{equation}
For example, if we consider $d=2$ and assume that no roots are repeated then we have only simple poles and we obtain the expression
\begin{equation}
 g_d(z) = \sum_{j = 1}^{t}r^{N_\mathrm{y}}_j\left[P_{N_\mathrm{x}+N_\mathrm{y}}(z)\prod_{k=1, k \neq j}^{N_\mathrm{x}+N_\mathrm{y}}(r_k(z)-r_j(z))\right]^{-1}  \hspace{1pt},   
\end{equation}
where $t$ is the number of roots that lie inside the contour $\Gamma$.
The singularity of $g_d(z)$ are among the roots of the common denominator:
\begin{equation}
P_{N_\mathrm{x}+N_\mathrm{y}}(z)\prod_{k \neq j}^{N_\mathrm{x}+N_\mathrm{y}}(r_k(z)-r_j(z)) \hspace{1pt}.
\end{equation}
These set of roots are the same as the roots of the discriminant of $p(s,zs^{-1})$ with respect to the $s$ variable. We remind the reader that the discriminant of a polynomial of degree $n$ $P_n(x) = \sum_{j=0}a_jx^j$ is given by
\begin{equation}
    \Delta(P_n) = (-1)^{n(n+1)/2}a^{2n-2}_n\prod_{i \neq j}(r_i-r_j) \hspace{1pt},
\end{equation}
where $r_i$ are the roots of the polynomial. Notice that if a root is repeated $m$ times, then we need to calculate a residue of a pole of order $m$. This implies that we need calculate an expression of the form 
\begin{equation}
    \frac{1}{{(m-1)}!}\lim_{s \to r_j}\frac{d^{m-1}}{ds^{m-1}}(s-r_j)^{m}\left[P_{N_\mathrm{x}+N_\mathrm{y}}(z)\prod_{k=1}^{N_\mathrm{x}+N_\mathrm{y}}(s-r_k(z))\right]^{-1} \hspace{1pt}.
\end{equation}
However, note that the singularities of this expression still lie within the roots of the discriminant (for both $d=2$ and $d=3$),  
since taking the \( m \)-th derivative essentially raises the denominator to the power \( m+1 \) and multiplies it by an analytic function.  

Therefore, the singularities of $g_d(z)$ lies within the roots of the discriminant of $p(s,zs^{-1})$ with respect to $s$. The strategy is to numerically compute these roots and then find the smallest $r_{\rm{min}}$, this corresponds to the smallest radius of convergence. The upper bound for the inverse radius of convergence of $g_d(z)$ is then
\begin{equation}
R^{-1}_d \leq 1/r_{\rm{min}} \hspace{1pt}.   
\end{equation}
However, some care is needed. It is not necessarily true that all roots of the discriminant correspond to singularities of the function $g_d(z)$ since some roots of the common denominator in Eq.~\eqref{eqn: residue} might actually be removable singularities. Therefore, it is not guaranteed a priori that $r^{-1}_{\rm{min}}$ provides a valid upper bound. 

In order to select the right inverse root at a given level $\ell$ we use the fact that the upper bound must decrease as the twig level increases (see Section~\ref{Twig method}). Let $\mathcal{R}^{-1}(\ell)$ be the set of inverse roots at level $\ell$ and assume that $r^{-1}(\ell-1)$ is the correct upper bound at level $\ell-1$, then the valid upper bound at level $\ell$ is given by the largest inverse root that is smaller than $r^{-1}(\ell-1)$:
\begin{equation}
    r^{-1}(\ell) = \max \left\{ r^{-1} \in \mathcal{R}^{-1}(\ell) \,\middle|\, r^{-1} < r^{-1}(\ell-1) \right\} \hspace{1pt}.
\end{equation}
This defines $r^{-1}(\ell)$ recursively; the initial condition $r^{-1}(0)$ is given by \Cref{thm:sam-som-upper-bound} with $(d=3,k=2)$.


\bibliographystyle{apalike} 
\bibliography{bibliography} 

\end{document}